\documentclass{amsart}
\usepackage{amsmath,amssymb,amsthm,amscd}
\usepackage[dvipsnames]{xcolor}
\usepackage[hypertexnames=false,linkcolor=black, colorlinks=true, citecolor=black, filecolor=black,urlcolor=black]{hyperref}
\usepackage{graphicx}
\usepackage{tikz}

\numberwithin{equation}{section}

\newcounter{CountAlpha}

\theoremstyle{plain}
\newtheorem{MainThm}[CountAlpha]{Theorem}
\newtheorem{MainProp}[CountAlpha]{Proposition}

\newtheorem{Thm}{Theorem}[section]
\newtheorem{Lem}[Thm]{Lemma}
\newtheorem{Prop}[Thm]{Proposition}
\newtheorem{Cor}[Thm]{Corollary}

\theoremstyle{definition}
\newtheorem{Def}[Thm]{Definition}
\newtheorem{Rk}[Thm]{Remark}
\newtheorem{Ex}[Thm]{Example}

\newtheorem{Not}[Thm]{Notation}

\newcommand{\poly}[3]{\Delta(  #1  ;  #2  ;  #3  ) }
\newcommand{\cpoly}[2]{\Delta(  #1  ; #2  ) }
\newcommand{\ord}{\mathrm{ord}}
\newcommand{\Diff}{\mathrm{Diff}}
\newcommand{\Spec}{\mathrm{Spec\,}}
\newcommand{\Sing}{\mathrm{Sing\,}}
\newcommand{\Dir}{\mathrm{Dir}}
\newcommand{\Rid}{\mathrm{Rid}}
\newcommand{\IDir}{\mathbb{D}\mathrm{ir}}
\newcommand{\IRid}{\mathbb{R}\mathrm{id}}
\newcommand{\wrt}{with respect to }
\newcommand{\vin}{\rotatebox{90}{\ensuremath{\in}}}
\newcommand{\LSB}{LSB}
\newcommand{\atA}{ \alpha {\bf\cdot} A }
\newcommand{\btB}{ \beta {\bf\cdot} B }
\newcommand{\hR}{\widehat{R}}
\newcommand{\hf}{\widehat{f}}

\newcommand{\hy}{\widehat{y}}
\newcommand{\hz}{\widehat{z}}

\newcommand{\car}[1]{\mathrm{char}(#1)}
\newcommand{\coeff}{C}
\newcommand{\NW}{{N}}
\newcommand{\resfield}{K}
\newcommand{\gr}{\operatorname{gr}}
\newcommand{\ini}{\operatorname{in}}
\newcommand{\TC}{T}
\newcommand{\ITC}{\mathbb{\TC}}
\newcommand{\IA}{\mathbb{A}} 
\newcommand{\IC}{\mathbb{C}}	
\newcommand{\ID}{\mathbb{D}}	
\newcommand{\IE}{\mathbb{E}}

\newcommand{\IQ}{\mathbb{Q}}	
\newcommand{\IR}{\mathbb{R}}	
\newcommand{\IZ}{\mathbb{Z}}	
\newcommand{\cD}{\mathcal{D}}

\newcommand{\cO}{\mathcal{O}}
\newcommand{\maxIdeal}{M}

\newcommand{\red}[1]{\textcolor{black}{#1}}

\definecolor{gruen}{rgb}{0, 0.666, 0}

\begin{document}
	\thispagestyle{empty}

	\title[Idealistic exponents: Tangent cone, ridge, characteristic polyhedra]{Idealistic exponents: Tangent cone, ridge, characteristic polyhedra}

	\author{Bernd Schober}
	\address{
		Bernd Schober\\
		Fakult\"at f\"ur Mathematik\\
		Universit\"at Regensburg\\
		Universit\"atsstr. 31\\
		93053 Regensburg \\
		Germany}
	\curraddr{Institut f\"ur Mathematik, Carl von Ossietzky Universit\"at Oldenburg, 26111 Oldenburg, Germany}
	\email{bernd.schober@uni-oldenburg.de}

	\thanks{Supported by the Emmy Noether Programme ``Arithmetic over finitely generated fields" (Deutsche Forschungsgemeinschaft, KE 1604/1-1) and by a Research Fellowship of the Deutsche Forschungsgemeinschaft (SCHO 1595/1-1).}


	\begin{abstract}
		We study Hironaka's idealistic exponents over $ \Spec ( \IZ ) $.
		We give an idealistic interpretation of the tangent cone, the directrix, and the ridge.
		The main purpose is to introduce the notion of characteristic polyhedra of idealistic exponents and deduce from them intrinsic data on the idealistic exponent.
	\end{abstract}

	\maketitle


	%
	%
	%
	%
	%
	%
	%
	%
	%
	%
	%
	%
	%
	%

	\section*{Introduction}
	\label{Intro}
	
	Polyhedra are important tools to study the local nature of singularities.
	For example, in \cite{HiroBowdoin} Hironaka used characteristic polyhedra of singularities to prove resolution of excellent hypersurface singularities in dimension two. 
	Further, he introduced the notion of characteristic polyhedra of an ideal in \cite{HiroCharPoly}, which has been used to extend \cite{HiroBowdoin} to the case of excellent schemes of dimension at most two  in \cite{CJS}.
	It plays also an essential role in the work of Cossart and the author \cite{HomeworkDim2}, where a strictly decreasing invariant for the strategy of \cite{CJS} is constructed. 
	Moreover, the characteristic polyhedron appears in the proof for resolution of singularities for (arithmetic) threefolds by Cossart and Piltant \cite{CP2, CPdim3}. 
	
	Based on Hironaka's work we develop in this paper the notion of characteristic polyhedra of idealistic exponents.
	The starting point for this study was the task to show that the invariant introduced by Bierstone and Milman \cite{BMheavy} in order to prove constructive resolution of singularities in characteristic zero can be purely determined by considering certain polyhedra and their projections.
	This result and thus a first application of characteristic polyhedra of idealistic exponents is discussed in \cite{BerndBM}.
	
	Nevertheless the theory of these polyhedra goes beyond the situation over fields of characteristic zero.
	Another interesting application would be the reinterpretation of the strategy of \cite{CJS} in the language of idealistic exponents.
	Since Hironaka's characteristic polyhedron is intensively used in \cite{CJS} and \cite{HomeworkDim2} there is the need to introduce appropriate polyhedra in the setting of idealistic exponents.

	\smallskip
	
	Given a regular scheme $ Z $ of finite type over $ \Spec ( \IZ ) $ we recall the notion of pairs $ \IE = ( J, b ) $, where $ J  \subset \cO_Z $ is a quasi-coherent ideal sheaf and $ b \in \IZ_+ $ an integer.
	(Later this will be extended to $ b \in \IQ_+ $).
	Roughly speaking, two pairs are equivalent, denoted by $ \sim $, if they undergo the same resolution process.
	Additionally, there is a technical part in the definition of $ \sim $ that the first condition is even true after adding finitely many new variables.
	The latter is a crucial property for proofs and is sometimes called Hironaka's trick.  
	An idealistic exponent $ \IE_\sim $ denotes then the equivalence class of a pair $ \IE $ \wrt $ \sim $.
	
	To a pair $ \IE $ and a point $ x \in Z $ we can associate the tangent cone, the directrix and the ridge of $ \IE$ at $ x $. 
	The latter two are closely related:
	for example, if the residue field of the regular local ring $ \cO_{Z,x} $ is perfect, then the reduced ideal of the ridge and the ideal of the directrix coincide.
	A detailed discussion of this is given in Remark \ref{Rk:DirRid}.
	
	It was already shown in \cite{HiroIdExp} that the directrices of equivalent pairs coincide (perfect base field!). 
	In order to reveal a similar relation for the tangent cone (resp.~the ridge) we introduce the concepts of \emph{idealistic tangent cone $ \ITC_x ( \IE ) $, idealistic directrix $ \IDir_x ( \IE ) $, and idealistic ridge $ \IRid_x ( \IE ) $} of a pair $ \IE $ at a singular point $ x $.
	Whereas the first concept (of an idealistic variant of the tangent cone) already appears in the work of Benito and Villamayor 
	(see \cite[section 1.2]{AngelicaOrlandoElimination})
	or Kawanoue and Matsuki 
	(see \cite[Definition 1.1.2]{IFPObergurgl}),
	the latter two are new.
	We have
	
	\begin{MainProp}[Proposition~\ref{Prop:ItcDirRidUnique}]
		\label{MainProp:IdTanRidDir}
		Let $ \IE_1 \sim \IE_2 $ be two equivalent pairs and $ x \in \Sing ( \IE_1 ) = \Sing ( \IE_2 ) $.
		Then 
		\begin{center}
			
			$ \ITC_x ( \IE_1 ) \sim \ITC_x ( \IE_2 )$,
			\hspace{7pt}
			$ \IDir_x ( \IE_1 ) \sim \IDir_x ( \IE_2 ) $, 
			\hspace{7pt}and
			\hspace{7pt}
			$ \IRid_x ( \IE_1 ) \sim \IRid_x ( \IE_2 ) $.
		\end{center}
	\end{MainProp}
	\noindent
	This means the idealistic tangent cone, the idealistic directrix and the idealistic ridge are actually invariants of the idealistic exponent $ \IE_\sim $.
	
	\smallskip
	
	For fixed data $ ( \IE, u, y ) $ we define its associated polyhedron $ \poly \IE uy $ and investigate its properties.
	Unfortunately, they behave badly under the equivalence relation $ \sim $;
	in Example \ref{Ex:PolyNotUnique} we show that there exist equivalent pairs whose associated polyhedra are not equal.
	Thus the polyhedron is not an invariant of the idealistic exponent $ \IE_\sim $.
	Nevertheless there is some hope:
	We can recover a numerical invariant 
	(which is closely related to the so called coefficient ideal)
	from the associated polyhedron which is an invariant of $ \IE_\sim $ (Proposition \ref{Prop:deltaEuyInv}).

	\smallskip
	
	Imitating the construction of Hironaka's characteristic polyhedron for a singularity we define the characteristic polyhedron $ \cpoly \IE u $ of a pair $ \IE $ \wrt a certain set of elements $ ( u ) = ( u_1, \ldots, u_e ) $ in $ R $ that can be extended to a regular system of parameters of $ R $.
	More precisely, $ \cpoly \IE u $ is the intersection over all possible choices for $ ( \hy )$ such that $ ( u, \hy ) $ is a regular system of parameters for $ \hR$ (the completion of $ R$ \wrt the unique maximal ideal),
	$$
	\cpoly \IE u := \bigcap_{(\hy)} \poly \IE u\hy.
	$$
	An interesting question is then if there is a good choice for $ ( \hy ) $ such that $ \poly \IE u \hy = \cpoly \IE u $ and whether we can choose such an $ (\hy ) $ in $ R $ instead of $ \hR $.
	While the first part turns out to be false in full generality (Example~\ref{Ex:AssumpThmNecessary}),
	we give an affirmative answer assuming a minor technical condition which is typically true, when studying a problem in the context of resolution of singularities:  
	Given $ \IE= (J,b) $ on $ R $, we define 
	$ \IE' := (J', b) $ to be the pair on $ R' := R/ \langle u \rangle $ induced by $ J' := J \cdot R' $.
	Denote by $ x' $ the point in $ \Spec(R') $ defined by the ideal generated by images of $ (y) $ in $ R' $.
	In particular, we can speak of the directrix $ \Dir_{x'} (\IE' )$.
	
	\begin{MainThm}[Theorem~\ref{Thm:BerndHiro}]
		\label{MainThm:CharPolyAchieved}		
		Let $ \IE = (J,b) $ be a pair on a regular local ring $ R $ and let $ ( u, y ) = ( u_1, \ldots, u_e; y_1, \ldots, y_r ) $  be a regular system of parameters for $ R $ such that the initial forms of $ ( y ) $ define the directrix of $ \Dir_{x'}( \IE') $.
		
		There exist elements $ ( y^* ) = ( y^*_1, \ldots, y^*_r ) $ in $ \hR $ such that $ ( u, y^* ) $ is a regular system of parameters for $ \hR $, $ ( y^* ) $ yields $ \Dir_{x'}( \IE') $, and 
		\[ 
		\poly{ \IE }{ u }{ y^* } = \cpoly{ \IE }{ u }
		\] 
	\end{MainThm}
	
	Moreover, later we can give a simple proof for the following result: if $ V ( y ) $ has maximal contact then the polyhedron $ \poly \IE uy $  associated to a pair $ \IE $ is independent of the choice of the maximal contact variables $ ( y ) $ (Proposition \ref{Prop:maxSamePoly}).
	
	Still the characteristic polyhedra of pairs do not behave well under $ \sim $.
	Thus, at the end of section \ref{sec:3}, we discuss the concept of a unique characteristic polyhedron of an idealistic exponent $ \IE_\sim $.
	Further, we sketch in Remark \ref{Rk:qh_setting} how the notion of characteristic polyhedra can be extended to the quasi-homogeneous situation, i.e., where we put certain weights on each element of the regular system of parameters of $ R $.
	
	For our purposes it is not crucial that we obtain a unique characteristic polyhedron for an idealistic exponent. 
	The important point is that the information which we obtain from the polyhedra are invariants of the idealistic exponent $ \IE_\sim $.
	In this context we prove for 
	$$ 
	\delta_x ( \IE, u )  :=  
	\inf \{ |v| = v_1 + \ldots + v_e \mid v \in \cpoly \IE u \} \in \frac{1}{b!} \, \IZ_{>1} \cup \{ \infty \}
	$$ 
	\textcolor{black}{the following result.}
	(Note that $ \delta_x ( \IE, u ) = \infty $ if and only if $ \cpoly \IE u = \varnothing $.)
	
	\begin{MainThm}[Theorem~\ref{Thm:deltaINV}]
		\label{MainThm:deltaINV}
		The number $ \delta_x ( \IE, u ) $ does not depend on $ ( y ) $ and is invariant under the equivalence relation $ \sim $.
		Therefore $ \delta_x ( \IE, u ) $ is an invariant of the idealistic exponent $ \IE_\sim $ and $ ( u ) $.
	\end{MainThm}
	
	In the situation over fields of characteristic zero this will be the essential ingredient to deduce the connection between the invariant of Bierstone and Milman and the characteristic polyhedra of idealistic exponents in \cite{BerndBM}.

	\smallskip

	\color{black} 
	Analogous to 	
	$ \delta_x (\IE;u) $,
	we may define
	$ \delta_x (\IE;u;y) $ 
	using $ \poly{\IE}{u}{y} $ instead of $ \cpoly{\IE}{u} $,
	where $ ( u, y) $ is any regular system of parameters. 
	In this situation, we introduce another pair 
	$ ( \ini_\delta (\IE) , b ) $ for a given pair $ \IE = (J,b) $,
	which is a refinement of the idealistic tangent cone $ \ITC_x ( \IE) $
	(Definition~\ref{Def:in_delta_E}).
	We can prove 
	(using Proposition~\ref{Prop:deltaEuyInv})
	that this pair is also an invariant of the idealistic exponent.

	\begin{MainThm}[Theorems~\ref{Thm:deltaINV_2}]
		\label{MainThm:deltaINV_2} 
	Let $ \IE_1 = ( J_1, b_1) \sim \IE_2 = (J_2, b_2) $ be two equivalent pairs and $ x \in \Sing (\IE_1 ) $.	
	Let $ (u,y) = (u_1, \ldots, u_d; y_1, \ldots, y_s) $ be a regular system of parameters for $ R  $.
	If we define $ \delta := \delta_x ( \IE_1; u; y ) $, 
	then we have  
	\[ 
	(\ini_{\delta}(\IE_1),b_1) \sim (\ini_{\delta}(\IE_2),b_2)
	.
	\]
\end{MainThm}
	
	\color{black} 
	
	\smallskip
	
	Another interesting tool for the study of singularities is the coefficient ideal \wrt $ V( y ) $, where $ ( y ) = (y_1, \ldots, y_r ) $ is part of a regular system of parameters $ ( u, y ) = ( u_1, \ldots, u_e, y ) $ of the local ring $ R = \cO_{Z,x} $ at a singular point $ x $.
	We define its counterpart in the language of idealistic exponents, called \emph{coefficient pairs} $ \ID_x ( \IE, u ,y ) $, where $ x \in \Sing ( \IE ) $. 
	So far no special assumptions on the system $ ( y ) $ are made except for being part of a regular system of parameters
	We then can show
	
	\begin{MainThm}[Theorem~\ref{Thm:IDequiv} and Proposition~\ref{Prop:IDEz=IDEy} ]
		\label{MainProp:IdCoeffExpIndp}
		Let $ \IE_1 \sim \IE_2 $ be two equivalent pairs on a regular local ring $ R  $ and let $ ( u, y ) $ be a regular system of parameters for $ R $.
		\begin{enumerate}
			\item[(1)]
			Then the coefficient pairs \wrt the same $ V(y) $ are equivalent,
			$$
			\ID_x ( \IE_1, u ,y ) \sim \ID_x ( \IE_2, u ,y ) .
			$$
			\item[(2)]	
			If we have two choices $ V ( y ) $ and $ V ( z ) $ for a fixed system $ ( u ) $ and a fixed pair $ \IE $ such that $ \IE \cap ( y, 1 ) \sim \IE \cap ( z, 1 ) $, then the coefficient pairs are also equivalent,
			$$
			\ID_x ( \IE, u ,y ) \sim \ID_x ( \IE, u , z) .
			$$
		\end{enumerate}
	\end{MainThm}
	\noindent
	In particular, $ \ID_x ( \IE, u ,y )_\sim $ is an invariant of the idealistic exponent $ \IE_\sim $.

	\smallskip
	
	\noindent 
	{\em Remark on notation:}
	Throughout the article, we use multi-index notation, e.g., we abbreviate $y^B = y_1^{B_1} \cdots y_r^{B_r} $ 
	and $ \frac{B}{d} = (\frac{B_1}{d}, \ldots, \frac{B_r}{d}) $,
	for $ ( y ) = (y_1, \ldots, y_r ) $, $ B = (B_1, \ldots, B_r) \in \IZ_{ \geq 0 }^r $, $ d \in \IZ_+ $.
	Further, we abuse notation by writing $ V(I) $ to denote the scheme $ \Spec(R/I) $, where $ I \subset R $ is an ideal in a regular local ring $ R $.
	
	\smallskip 
	
	\noindent 
	\emph{Acknowledgement:} 
	The results presented here are part of my thesis \cite{BerndThesis}.
	I thank my advisors Uwe Jannsen and Vincent Cossart for countless discussions on the topic and all their support.
	\red{Furthermore, I thank the anonymous referee for many helpful comments improving the article and for suggesting to emphasize more the role of the ridge.}
	I'm grateful to the Laboratoire de Math\'emathiques Versailles for their hospitality during numerous visits.

	%
	%
	%
	%
	%
	%
	%
	%
	%
	%
	%
	%
	%
	%
	
	\section{Pairs and idealistic exponents}
	
	Originally, Hironaka introduced idealistic exponents on regular schemes which are of finite type over a perfect field, see \cite{HiroIdExp, HiroKorean}.
	Later he extended this notion to idealistic exponents on excellent regular Noetherian schemes which are not necessarily of finite type over a base field, see \cite{HiroThreeKey}. 
	In \cite{BerndThesis} the author recalled the definitions and worked out the proofs of the first properties for the case over arbitrary fields in detail. 
	In this article we focus on idealistic exponents on a regular irreducible schemes of finite type over $ \IZ $ which is sufficient for our purposes.
	(We follow the usual convention and write \emph{over $ \IZ $} instead of \emph{over $ \Spec ( \IZ ) $}).
	
	\smallskip
	
	Let $ Z $ be a regular irreducible scheme of finite type over $ \IZ $.
	
	\begin{Def}
		\label{Def:idexp}
		A \emph{pair} $ \IE = (J,b) $ on $ Z $ is a pair consisting of a quasi-coherent ideal sheaf $ J \subset \cO_Z $ and a positive integer $ b \in \IZ_+ $.
		We define its order at a point $ x \in Z $ (not necessarily closed) 
		as 
		$$ 
		\ord_x (\IE) := \left\{ \begin{tabular}{cl}
		$ \frac{\ord_x (J) }{b} $ 	& , if $ \ord_x (J)  \geq b $ and \\[8pt]
		$ 0 $				& , otherwise,
		\end{tabular} \right.
		$$ 
		where $ \ord_x (J) = \sup \{ d \in \IZ_{ \geq 0 } \cup \{\infty\} \mid J_x \subseteq \maxIdeal_x^d \} $ (and $ \maxIdeal_x $ denotes the maximal ideal in the local ring at $ x $).
		Further, we define the singular locus (or (co)support) of $ \IE $ as 
		$$
		\Sing(\IE):= \{ x \in Z \mid \ord_x ( J ) \geq b \}.
		$$
	\end{Def}
	
	\smallskip
	
	Since the order is a upper-semi continuous function
	(see \cite[Chapter III \S 3 Corollary 1 (p.220)]{Hiro64}), 
	$ \Sing (\IE) $ is closed.
	
	We denote by $ X \subseteq Z $ the closed subscheme corresponding to $ J $.
	Further, if $ Z = \Spec (R) $ is affine, then we also say $ \IE $ is a pair on $ R $.
	
	\begin{Def}
		\label{Def:perm}
		Let $ \IE = (J, b) $ be a pair on $ Z $.
		A blowup $ \pi: Z' \to Z $ with center $ D $ is called \emph{permissible} for $ \IE $, if $ D $ is regular and $ D \subseteq \Sing (\IE) $.
		The transform of $ \IE $ is then given by $ \IE' = (J', b) $, where $ J' $ is defined via $ J \cO_{Z'} = J' H^b $, where $ H $ denotes the ideal sheaf of the exceptional divisor.
	\end{Def}
	
	\smallskip
	
	In other literature there exist different notions of permissible centers, see for example \cite{CJS}.
	There these are regular subschemes $ D \subset X $ such that additionally $ X $ is normally flat along $ D $ (\cite[Definition 2.1]{CJS}).

	\begin{Def}
		We define a \emph{local sequence of regular blowups (or {\LSB} for short)} over $ Z $ as a sequence of the form
		\begin{equation}
		\label{eq:LSB}
		\begin{array}{ccccccccc}
		Z=Z_0 \supset U_0		& \stackrel{\pi_1}{\leftarrow} &	Z_1 \supset U_1	 &	\ldots	& \stackrel{\pi_{\ell-1}}{\leftarrow} &	Z_{\ell-1}  \supset U_{\ell-1}	& \stackrel{\pi_\ell}{\leftarrow} 	&	Z_\ell	\\[5pt]
		\hspace{45pt}	\cup	&&	\hspace{25pt}	\cup	&			&&	\hspace{35pt}	\cup	\\[5pt]
		\hspace{45pt}	D_0		&&	\hspace{25pt}	D_1		&	\ldots	&&	\hspace{35pt}	D_{\ell-1}	
		\end{array}
		\end{equation}
		%
		where $ \ell \in \IZ_+ \cup \{ \infty\} $, each $ U_i \subset Z_i $ is an open subscheme, $ D_i \subset U_i $ is a regular closed subscheme and $ \pi_{i+1}: Z_{i+1} \to U_i $ is the blowup with center $ D_i $, $ 0 \leq i \leq \ell - 1 $. 
	\end{Def}
	
	\begin{Rk}
		\label{Rk:permLSB}
		Let $ \IE = (J,b) $ be a pair on $ Z $ and consider a {\LSB} as in \eqref{eq:LSB}.
		In Definition \ref{Def:perm} we have introduced when a blowup is permissible for $ \IE $ and further we have defined the transform of $ \IE $ under such a blowup.
		Denote by $ \IE_i $ the transform of $ \IE_0 := \IE $ in $ Z_i $, for $ 0 \leq i \leq \ell -1 $.
		Then we say that the {\LSB} (\ref{eq:LSB}) is permissible for $ \IE $ if each blowup $ \pi_{i + 1} $ is permissible for $ \IE_{ i  } $, for $ 0 \leq i \leq \ell - 1 $.
	\end{Rk}
	
	\smallskip
	
	Let $ (t) = (t_1, \ldots, t_a) $ be a finite system of independent variables.
	Then we use the notation
	$$
	Z[t] := Z \times_{\mbox{\footnotesize Spec} (\IZ)} \Spec (\IZ[t]).
	$$
	We consider the pair $ \IE[t] = (J[t], b) $, where $ J[t] := J \cO_{Z[t]} $ (\wrt the canonical projection).

	\begin{Def}
		\label{Def:equiv}
		Let $ \IE_1 = (J_1, b_1) $ and $ \IE_2 = (J_2, b_2) $ be two pairs on $ Z $.
		Then we define 
		$$
		\IE_1 \subset \IE_2
		$$
		if the following condition holds:
		\begin{equation}
		\label{eq:equiv}
		\parbox{9.8cm}{
			Let $ (t) = ( t_1, \ldots, t_a ) $ be an arbitrary finite system of independent variables and let $ \IE_i[t] = (J_i[t], b_i) $, $ i \in \{ 1,2 \} $.
			If any {\LSB} over $ Z[t] $ is permissible for $ \IE_1[t] $, then it is also permissible for $ \IE_2[t] $. 
		}	
		\end{equation}		
		We say $ \IE_1 $ and $ \IE_2 $ are \emph{equivalent},
		$$
		\IE_1 \sim \IE_2,
		$$
		if both $ \IE_1 \subset \IE_2 $ and $ \IE_1 \supset \IE_2 $. 
		An {\em idealistic exponent} $ \IE_\sim $ is the equivalence class of a pair $ \IE $. 
		
		Set $ c := \mathrm{lcm} (b_1, b_2) $ and $ c_i := \frac{c}{b_i} $, $ i \in \{ 1, 2 \} $. 
		We define the intersection of two pairs as the following pair on $ Z $,
		$$ 
		\IE_1 \cap \IE_2 := ( J_1^{c_1} + J_2^{c_2} , c).
		$$ 
	\end{Def}

	In other literature pairs are sometimes also called idealistic exponents (e.g. \cite{HiroThreeKey}). 
	In order to avoid confusion when coming to results and the dependence on the choice of a representative of the equivalence class, we use the original terminology \cite{HiroIdExp} of pairs and idealistic exponents.
	
	\smallskip
	
	Note that a {\LSB} over $ Z[t] $ is permissible for $ (\IE_1 \cap \IE_2)[t] $ if and only if it is permissible for $ \IE_1[t] $ and $ \IE_2 [t] $.
	We have $ \Sing ( \IE_1 \cap \IE_2 ) = \Sing ( \IE_1 ) \cap \Sing ( \IE_2 ) $ 
	and $ \ord_x ( \IE_1 \cap \IE_2 ) = \min  \{  \ord_x ( \IE_1 ),  \ord_x ( \IE_2 ) \} $, for  $ x \in \Sing ( \IE_1 \cap \IE_2 ) $.
	
	\smallskip

	The following basic properties of pairs hold:
	
	\begin{Lem}
		\label{Lem:Basic}
		Let $ \IE = (J,b) $ and $ \IE_i = (J_i,b_i) $, $ i \in \{1, 2, 3, 4 \}$, be pairs on $ Z $.
		\begin{enumerate}
			\item 
			For every $ a \in \IZ_+ $, we have $ (J^a, a b) \sim (J,b) $.
			\item 
			Let $ m \in \IZ_+ $ with $ b_1 \mid m $ and $ b_2 \mid m $.
			Then  
			$$
			(J_1,b_1) \cap (J_2,b_2) \sim \left(J_1^{\frac{m}{b_1}} + J_2^{\frac{m}{b_2}}, m\right).
			$$
			In particular, if $ J_1 \subseteq J_2 $ and $ b_1 = b_2 =: b$,
			then 
			\[
				(J_2, b) = (J_1 + J_2, b ) \sim (J_1, b) \cap (J_2, b) \subset (J_1, b ).
			\] 
			
			\item 
			We always have $ (J_1 J_2, b_1+ b_2) \supset (J_1,b_1) \cap (J_2,b_2) $.
			If we have additionally $ \Sing (J_i, b_i + 1) = \varnothing $, for $ i \in \{ 1, 2\} $, then the inclusion becomes an equivalence.
			\item 
			If $ \IE_1 \subset \IE_2 $ and $ \IE_3 \subset \IE_4 $, then $ \IE_1 \cap \IE_3 \subset \IE_2 \cap \IE_4 $.
			In particular $ \IE_1 \sim \IE_2 $ implies by symmetry $ \IE_1 \cap \IE_3 \sim \IE_2 \cap \IE_3 $.
			\item 
			Let $ \pi : Z' \to Z $ be a permissible blowup for $ \IE_1 $ and $ \IE_2 $.
			Then $ (\IE_1 \cap \IE_ 2)' \sim \IE'_1 \cap \IE'_2 $.
		\end{enumerate}
	\end{Lem}
	
	\begin{proof}
		All claims follow by looking at the definitions.
		For a more detailed proof see \cite[Lemma 1.1.8]{BerndThesis}, where we only have to replace the base field $ k $ by $ \IZ $.
	\end{proof}
	
	By (1) we may extend the definition of pairs $ (J,b) $ to $ b \in \IQ_+ $:
	Suppose $ b = \frac{c}{d} \in \IQ_+ $, where the greatest common divisor of $ c, d \in \IZ_+ $ is $ 1 $.
	Then we define $ (J,b) $ to be a pair with assigned number $ b \in \IQ_+ $ which is equivalent to $ ( J^d, b\, d ) $.
	
	\smallskip
	
	The following is an example, where the assumptions of the second part of (3) do not hold.
	For a strategy for constructing such examples see \cite[Remark 1.1.9]{BerndThesis}.
	
	\begin{Ex}
		Consider the ideals $ J_1 = \langle x^3 + y^5 \rangle $, $ J_2 = \langle x z^2 + y^3 \rangle \subset \IZ [ x, y, z ] $.
		Since the blowup with center $ V( x,y ) $ is permissible for $ (J_1 J_2, 4) $ but not for $(J_2,2) $, we have $ (J_1 J_2, 4) \not\subset (J_1,2) \cap (J_2,2) $.
	\end{Ex}
	
	Another important result is the following
	
	\begin{Prop}[Numerical Exponent Theorem; {\cite[Theorem 5.1]{HiroThreeKey}}]
		\label{Prop:NumExp}
		Let $ \IE_1 = (J_1, b_1) $ and $ \IE_2 = (J_2, b_2) $ be pairs on $ Z $.
		If $ \IE_1 \subset \IE_2 $, then 
		$$ 
		\ord_x (\IE_1) \leq \ord_x (\IE_2),  \hspace{10pt}\mbox{ for all } x \in Z .
		$$
		By symmetry $ \IE_1 \sim \IE_2 $ implies $ \ord_x (\IE_1) = \ord_x (\IE_2) $, for all $ x \in Z $.
		In particular, we get $ \Sing(\IE_1) = \Sing (\IE_2) $ if $ \IE_1 \sim \IE_2 $.
	\end{Prop}
	
	The last statement implies that $ \Sing ( \IE_\sim ) $ is an invariant of the idealistic exponent.
	
	\begin{proof}
		We only explain the idea. 
		For more details we refer to Hironaka's original proof (\cite[Theorem 5.1]{HiroThreeKey}), or alternatively \cite[Proposition 1.1.10]{BerndThesis}.
		
		Consider the local situation at $ x \in Z $:
		Let $ R := \cO_{Z,x} $ and $ ( w ) = ( w_1, \ldots, w_n ) $ be a regular system of parameters for the (excellent) regular local ring $ R $.
		We introduce a new variable $ t $ and construct a {\LSB} $ S(\alpha, \beta) $ over $ R[t] $, $ \alpha, \beta \in \IZ_+ $, in the following way:
		
		First, we blow up with center $ V(w, t ) $ and we consider the origin of the $ T $-chart, i.e., the point with coordinates $ ( w', t) := ( \frac{w_1}{t}, \ldots, \frac{w_n}{t}, t) $.
		If $ \alpha = 1 $, the first part of $ S(\alpha, \beta ) $ ends.
		Otherwise we repeat the previous (blow up $ V(w',t) $ and consider the origin of the $ T $-chart) until we have blown up $ \alpha $-times.
		Hence we eventually finish with the point with coordinates $ (w^{(\alpha)}, t ) := ( \frac{w_1}{t^\alpha}, \ldots, \frac{w_n}{t^\alpha}, t) $.
		After this we blow up $ \beta $-times with center $ V( t ) $.
		(Keep in mind how the transform of a pair under a permissible blowup is defined, Definition \ref{Def:perm}).
		
		Suppose there exists some $ x_0 \in Z $ with $ \ord_{x_0} (\IE_1) > \ord_{x_0} (\IE_2) $.
		Set $ \alpha_0 := b_1 b_2 $ and $ \beta_0 = ( \ord_{x_0} ( \IE_1) - 1 ) \alpha_0 $.
		Then $ S(\alpha_0, \beta_0) $ is permissible for $ \IE_1 $, but not for $ \IE_2 $. 
		This contradicts $ \IE_1 \subset \IE_2 $ and the claim follows.
	\end{proof}
	
	In general, the converse of the Numerical Exponent Theorem is false. 
	More precisely, the condition $  \ord_x (\IE_1) \leq \ord_x (\IE_2) $, for all $ x \in Z $, is not stable under permissible blowups.
	An easy example is $ \IE_1 = ( y^2 + x^5, 2) $ and $ \IE_2 = ( x^2 + y^5 , 2) $ over $ \IA^2_\IC $.
	Both pairs have order $ 1 $ at the point $ V(x,y) $ and order $ 0 $ at any other point.
	Thus the blowup with center $ V(x,y) $ is permissible for both.
	We observe that the origin of the $ X $-chart (coordinates $ (x, \frac{y}{x} ) $) is permissible for $ \IE_1' $, but not for $ \IE_2' $.
	Therefore $ \IE_1 \not \subset \IE_2 $.
	
	\begin{Not}
		Let $ d \in \IZ_{\geq 0} $ be a non-negative integer and $ Z $ as usual a regular scheme (resp. let $ R $ be a regular ring).
		Then we denote by $ \Diff^{\leq d}_{\IZ} (Z) $ (resp. $ \Diff^{\leq d}_{\IZ} (R) $) the (absolute) differential operators of order at most $ d $ of $ \cO_Z $ (resp. $ R $) on itself.
	\end{Not}
	
	\begin{Prop}[Diff Theorem; {\cite[Theorem 3.4]{HiroThreeKey}}]
		\label{Prop:Diff}
		Let $ \IE = (J,b) $ be a pair on $ Z $ and $ d \in \IZ_{\geq 0} $.
		Let $ \cD $ be a left $ \cO_Z $-submodule of $ \Diff^{\leq d}_{\IZ}(Z) $. 
		Then
		$$
		(J, b) \subset (\cD J, b - d)
		$$
		or equivalently $ (J, b) \sim  (\cD J, b - d) \cap (J,b) $.
	\end{Prop}
	
	If $ d \geq b $, then the assigned number of $ (\cD J, b - d) $ is not positive and hence is a priori not defined. 
	But the singular locus $ \{ x \in Z \mid \ord_x (\cD J) \geq b - d \} $ is still defined.
	Then $ \Sing (\cD J, b - d) = Z $ and the claim follows immediately.
	
	\smallskip
	
	We also use this proposition in the case of a single differential operator $ \cD \in \Diff^{\leq d}_{\IZ}(Z) $;
	here we identify $ \cD $ with the submodule of $ \Diff^{\leq d}_{\IZ}(Z) $ generated by $ \cD $.
	
	\begin{proof}
		Again, we explain the idea, and refer to Hironaka's original proof (\cite[Theorem 3.4]{HiroThreeKey}), or alternatively \cite[Proposition 1.1.13]{BerndThesis}.
		
		Since the situation does not change by extending $ Z $ to $ Z [t] $, we may assume $ ( t ) = \varnothing $.
		Let $ D \subset Z $ be an arbitrary regular closed subscheme and let $ \pi : Z' \to Z $ be the blowup with center $ D $.
		We show:
		\begin{enumerate}
			\item 
			If $ \pi $ is permissible for $ (J, b) $, then so it is for $ (\cD J, b - d) $.
			\item 
			The relation between the transforms of $ (J, b) $ and $ (\cD J, b - d) $ under $ \pi $ is the same as before.
		\end{enumerate} 
		Let $ y \in Z $ be a generic point of $ D $.
		The first part follows by the fact that $ J_y \subset \maxIdeal^\ell $,
		for some $ \ell \in \IZ_+ $
		and $ \maxIdeal $ the maximal ideal of $ {\mathcal O}_{Z,y} $, implies $ \cD J_y \subset \maxIdeal^{ \ell - d } $.
		(For example see \cite[Lemma 3.1]{HiroThreeKey}).
		
		For (2) 
		we have to show that there exists a left $ \cO_{Z'} $-submodule $ \cD' $ of  $ \Diff^{\leq d}_{\IZ}(Z') $ such that
		$
		(\cD J)' = \cD' J'
		$,
		with  $ (\cD J, b - d)' = ( (\cD J)', b - d) $ and $ (J,b)' = (J', b) $.
		
		\emph{Caution:} Do not forget that the transforms are given by different laws, namely $ (\cD J) \cO_{Z'} = H^{b- d}(\cD J)' $ and $ J \cO_{Z'} = H^b J' $, where $ H \subset \cO_{Z'} $ denotes the ideal sheaf of the exceptional divisor.
		
		Let $ Q \subset \cO_Z $ be the ideal sheaf corresponding to the center $ D $.
		Then the required property holds for
		$
		\cD' := H^{- b + d } \cdot \cD \cdot Q^b
		$
		(viewed as an left $ \cO_{Z'} $-module in the function field of $ Z $).
		It is only left to verify $ \cD' \subset \Diff^{\leq d}_{\IZ}(Z') $.
		%
	\end{proof}
	
	Assume $ Z = \Spec(R) $ is affine and let $ J \subset R $ be an ideal.
	Let $ ( f_1, \ldots, f_m ) $ denote a set of generators for $ J $ and let $ \cD $ be as before.
	Instead of $ \cD J $ we want to apply the Diff Theorem for the ideal generated by $ ( \cD f_1, \ldots, \cD f_m ) $.
	In general, these two ideals do not coincide.
	We sometimes use the Diff Theorem in the following slightly modified version:
	
	\begin{Cor}
		Let $ \IE = (J,b) $ be a pair on $  Z =\Spec(R) $ and $ \cD \subset  \Diff^{\leq d}_{\IZ}(R) $ be a left $ R $-submodule.
		Further, let $ ( f_1, \ldots, f_m ) $ be a set of generators of the ideal $ J $.
		Then
		$$
		(J, b) \subset ( \, \langle \cD f_1, \ldots, \cD f_m \rangle,\, b - d \,)
		$$
		or equivalently $ (J, b) \sim  ( \, \langle \cD f_1, \ldots, \cD f_m \rangle,\, b - d \,) \cap (J,b) $.
	\end{Cor}	
	
	\begin{proof}
		By Proposition \ref{Prop:Diff} we have
		$
		(J, b) \subset (\cD J, b - d)
		$.
		Moreover, $ \langle \cD f_1, \ldots, \cD f_m \rangle \subseteq \cD J $ implies
		$
		(\cD J, b - d) \subset (  \langle \cD f_1, \ldots, \cD f_m \rangle,\, b - d ) 	
		$
		(use Lemma \ref{Lem:Basic}(2)).		
		This shows the assertion.	
	\end{proof}
	
	Since this is an immediate consequence of the Diff Theorem we do not distinguish between the corollary and the proposition.
	If we use them then we refer only to the Diff Theorem, Proposition \ref{Prop:Diff}.

	\smallskip
	
	Let $ R $ be a regular local ring with regular system of parameters $ (w_1, \ldots, w_n ) $.
	The following example shows that the equivalence of pairs is not stable if we pass from $ R $ to $ R_* := R [u_1, \ldots, u_n]/ \langle u_i - w_i^{a_i} \mid 1 \leq i \leq n \rangle $,
	for $ a_1, \ldots, a_n \in \IZ_+ $.

\begin{Ex} 
	Let $ R $ be the localization of $  \IC[x,y,z,t] $ at the maximal ideal $ \langle x,y,z,t \rangle $. 
	Consider the two pairs 
	\[ 
		\IE_1 = (z^3 - x^2 y^2,\, 3) \;\cap\; (t,\, 1)  	
		\hspace{5pt} \mbox{ and } \hspace{5pt}
		\IE_2 = (z^3 - x^2 y^2,\, 3) \;\cap\; (t^2-x y^2,\, 2). 
	\] 
	The Diff Theorem, Proposition \ref{Prop:Diff}, provides 
	(apply $ \frac{\partial}{\partial x}$ to $ \IE_1 $ and $ \frac{\partial}{\partial t}$ to $ \IE_2 $) 
	the equivalences
	\[ 
	\IE_1 \sim (z^3 - x^2 y^2,\, 3) \cap ( x y^2,\, 2)  \cap (t,\, 1) \sim \IE_2.
	\]	
	Let $ R_* := R [u]/ \langle u - x^3 \rangle $. 
	Then we obtain  
	\[ 
	\IE_{1,*} := (z^3 - u^6 y^2,\, 3) \;\cap\; (t,\, 1)  	
	\hspace{5pt} \mbox{ and } \hspace{5pt}
	\IE_{2,*} := (z^3 - u^6 y^2,\, 3) \;\cap\; (t^2 - u^3  y^2,\, 2). 
	\] 
	
	We claim that $ \IE_{1,*} \not\sim \IE_{2,*} $:
	We blow up with center $ V (u,z,t) $. 
	At the origin of the $ U $-chart, we have coordinates $ (u',y',z', t') = (u,y,\frac{z}{u}, \frac{t}{u}) $ and the transform of the pairs are 
	\[ 
	\IE_{1,*}' := (z'^3 - u'^3 y'^2,\, 3) \;\cap\; (t',\, 1),  	
	\hspace{5pt}\hspace{5pt}
	\IE_{2,*}' := (z'^3 - u'^3 y'^2,\, 3) \;\cap\; (t'^2 - u'  y'^2,\, 2). 
	\] 
	We observe that $ V(u',z',t') $ is permissible for $ \IE_{1,*}' $, 
	while it is not permissible for $ \IE_{2,*}' $.
	Hence,  $ \IE_{1,*} \not\sim \IE_{2,*} $.

	In Example~\ref{Ex:PolyNotUnique}, a more general variant of the example reappears in the context of polyhedra. 
	There, we also briefly discuss the origins of the example.  
	\end{Ex}

	%
	%
	%
	%
	%
	%
	%
	%
	%
	%
	%
	%
	%

	\section{Tangent cone, directrix and ridge}

	Let $ x \in Z $ be an arbitrary point and let $ R = \cO_{Z,x} $ be the regular local ring with maximal ideal $ \maxIdeal $ and residue field $ \resfield = R/\maxIdeal $. 
	Therefore we can associate the tangent space of $ Z $ at $ x $
	$$
	T_x ( Z ) := \Spec ( \gr_x (Z)),
	$$
	where $ \gr_x (Z) := \gr_M(R) := \bigoplus_{a \geq 0} \maxIdeal^a/ \maxIdeal^{a+1} $.
	
	Let $ \IE = (J,b) $ be a pair on $ Z $.
	By abuse of notation we neglect in $ \IE_x = (J_x, b) $ the index $ x $ and write also $ \IE = (J,b) $.
	In what follows we introduce the tangent cone, the directrix and the ridge of $ \IE $ at $ x $ and we discuss the aspect of their uniqueness up to equivalence.
	In order to achieve the last point we give an interpretation of these objects as idealistic exponents.
	
	\smallskip
	
	\begin{Def}
		Let $ f \in R $ and $ b \in \IQ_+ $ with $ b \leq \ord_x(f) $.
		We define the {\em $ b $-initial form of $ f $ (\wrt $ \maxIdeal $)}, \red{$ \ini (f,b) \in \maxIdeal^b / \maxIdeal^{b+1} $,} as
		$$
		\ini (f,b) := \left\{ \hspace{5pt} \begin{tabular}{cc}
		$ f \mod \maxIdeal^{ b + 1 } $, 	&	 if $ b \in \IZ_+ $, \\[5pt]
		$ 0 $,					&	 if $ b \notin \IZ_+ $.
		\end{tabular}
		\right.		
		$$
	\end{Def}
	
	Note that $ b <  \ord_x(f) $ implies $ \ini (f,b) = 0 $.
	
	\begin{Def}
		Let $ \IE = (J, b) $ be a pair on $ Z $ and $ x \in \Sing (\IE) $.
		Then we define the tangent cone $ \TC_x (\IE) \subset T_x ( Z ) $ of $ \IE $ at $ x $ as the subscheme defined by the homogeneous ideal $ \ini_x (J,b) \red{\,\subseteq \maxIdeal^b / \maxIdeal^{b+1}} \subset  \gr_x (Z) $, where
		$$
		\ini_x (J,b) :=  \left\{ \begin{tabular}{cc}
		$   J \mod \maxIdeal^{ b + 1 }  = \langle \ini (f, b) \mid f \in J \rangle $, 	&	 if $ b \in \IZ_{\ge 0} $, \\[5pt] 
		$ \langle 0 \rangle $,					&	 if $ b \notin \IZ_{\ge 0} $.
		\end{tabular}
		\right.		
		$$
	\end{Def}

	\begin{Rk}
		\label{Rk:2.3}
		\
		\begin{enumerate}
			\item[(1)]	 The ideal $ \ini_x (J,b) \subset \gr_x (Z) $ is well-defined and generated by homogeneous elements of degree $ b $, as $ x \in \Sing (\IE) $ and thus $ \ord_x ( J) \geq b $.			
			\red{If $ \ord_x (J) > b $, then $ T_x (\IE) = T_x (Z) $.}
			\item[(2)] 		The tangent cone $ \TC_x (\IE) $  is \emph{not} invariant under the equivalence relation $ \sim $.
			We overcome this later by using an idealistic interpretation of the tangent cone. 
			\item[(3)]	If $ \ord_x (\IE_2) > 1 $, then $ \ini_x(\IE_2)  = \langle 0 \rangle $,
			 and thus 
			$
			\TC_x ( \IE_1 \cap \IE_2 ) = \TC_x ( \IE_1). 
			$
		\end{enumerate}
	\end{Rk}

	\smallskip
	
	Let us for the moment consider a more general situation:
	Let $ K $ be a field, consider the polynomial ring $ S = K [ W ] = K [W_1, \ldots, W_n] $ as a graded ring and let $ I \subset S $ be a homogeneous ideal.
	Then $ I $ defines a cone $ C = C(I) = \Spec ( S/I ) $.
	In this setting we can define the directrix and the ridge of $ C $ which go back to Hironaka and Giraud.
	
	\begin{Def}
		{\red{The {\em directrix} $ \Dir(C) $ of the cone $ C = C(I) $ is the maximal vector subspace of $ \IA^n_K $ leaving the cone $ C $ stable under translations, i.e., such that we have $ C + \Dir(C) = C .$}}
	\end{Def}

	{\red{Hence, the directrix of $ C $ corresponds to the smallest $ K $-subvectorspace $ T = \bigoplus_{ j = 1 }^r K Y_j \subset S_1 = \bigoplus_{ i = 1 }^n K W_i $ generated by elements $ Y_1,\ldots Y_r \in  S_1 $ (homogeneous of degree one) such that
		$$
		(\, I \,\cap\, K[Y_1,\ldots,Y_r ] \,)\cdot S = I .
		$$
	In other words,}} $ T  = \bigoplus_{ j = 1 }^r K Y_j $ is the minimal $ K $-subspace such that $ I $ is generated by elements in $ K[ Y_1, \ldots, Y_r ] $.
			(i.e., $ ( Y_1, \ldots, Y_r ) $ is the smallest list of variables such that there exists a system of generators for $ I $ which can be expressed by $ ( Y) $).
	
	\begin{Def} 		
		We call $ I\Dir ( C ) := \langle Y_1, \ldots Y_r \rangle $ the \emph{ideal of the directrix}, 
		$$
		\Dir(C) = \Spec ( S /I\Dir ( C ) ) \subset  \IA^n_K = \Spec ( S ).
		$$
		We also say $ ( Y ) = ( Y_1, \ldots, Y_r ) $ defines the directrix and we implicitly assume that $ r $ is minimal.
	\end{Def}

	Recall that a polynomial $ \sigma \in K [W] = S $ is called additive if for any $ x, y \in K^n $ we have $ \sigma ( x + y ) = \sigma ( x ) + \sigma ( y ) $.
	
	\begin{Def}[{\red{\cite{GiraudEtude}, Ch. I, \S 5}}] 
		The {\em ridge} (or {\em fa\^ite} in French) $ \Rid (C) $
		{\red{is the maximal additive subgroup of $ \IA^n_k $
		leaving the cone $ C $ stable under translations,
		where we consider $ \IA^n_k $ as an additive group scheme.}}
	\end{Def} 

	{\red{Analogous to the directrix, the ridge of $ C $ corresponds to}}
	the smallest additive subspace $ K[ \sigma_1, \ldots, \sigma_s ] \subset S $ generated by additive homogeneous polynomials $  \sigma_1, \ldots, \sigma_s  \in  S $ such that
		$$
		(\, I \,\cap\, K[ \sigma_1, \ldots, \sigma_s ]\,)\cdot S = I .
		$$
		
	\begin{Def} 
		We call  $ I\Rid ( C ) := \langle \sigma_1, \ldots, \sigma_s \rangle $ the \emph{ideal of the ridge} and denote
		$$
		\Rid(C) := \Spec ( S /I\Rid ( C ) ) \subset  \IA^n_k
		\ \ \ \mbox{(group subscheme).}
		$$
		We say $ (  \sigma_1, \ldots, \sigma_s)  $ defines the ridge and implicitly assume that $ s $ is minimal.
	\end{Def}

	\begin{Rk}
		\label{Rk:DirRid}	
		In the case of $ \car{ K } = 0 $ the additive polynomials are those homogeneous of degree one. 
		Thus the previous definitions coincide in this situation, $ \Dir (C) = \Rid (C) $.
		
		If $ p = \car{ K } > 0 $ is positive, then the additive homogeneous polynomials are of the form $ \sigma = \sum_{i = 1 }^n \lambda_i W_i^q $, $ \lambda_i \in K $ and $ q = p^e $, $ e \in \IZ_{ \geq 0 } $.
		If moreover $ K $ is perfect, then $ \sigma = \psi^q $ for some $ \psi \in K [W] $ homogeneous of degree one.
		Hence the directrix is the reduction of the ridge, $ \Dir (C) = (\Rid (C))_{red} $, if $ K $ is perfect.
		
		For arbitrary $ K $ and $ \lambda \in K $, we do not know whether there is an element $ \rho  \in K $ such that $ \rho^q = \lambda $, $ q = p^e $ as before.
		But there is a purely inseparable finite extension $ K(\lambda) /K $ such that this property holds in $ K(\lambda) $;
		e.g. $ K(\lambda) = K[X]/ \langle X^q - \lambda \rangle $.
		Since the set of coefficients appearing in $ ( \sigma_1, \ldots, \sigma_s) $, $ \big\{ \lambda_i^{(j)} \in K \mid \sigma_j = \sum_{ i = 1 }^n \lambda_i^{(j)}\, W_i^{q_j}, \, q_j = p^{ e_j }, e_j \in \IZ_{ \geq 0 } ,\, j \in \{ 1, \ldots, s \} \big\} $, is a finite set, there exists a finite pure-inseparable extension $ K'/K $ such that $ \Dir (C)_{K'} = (\Rid (C)_{K'})_{red} $, where $ (.)_{K'} = (.) \times_K K' $.
	\end{Rk}
	
	For more details on the ridge (and in particular an intrinsic definition) see \cite{GiraudEtude, ComputeRidge}.
	Further, in \cite{CPS} the role of the ridge as an invariant of a singularity is discussed. 
	
	\smallskip
	
	We come back  to our situation (that is, where $ S = \gr_x ( Z ) $, $ C = \TC_x (\IE) = \Spec ( S / \ini_x (\IE)) $).
	
	\begin{Def}
		Let $ \IE = (J,b) $ be a pair on $ Z $ and $ x \in \Sing ( \IE) $.
		We define
		the {\em directrix of $ \IE $ at $ x $} by
		$ 
		\Dir_x (\IE) := \Dir (\TC_x (\IE)) 
		$,
		and the {\em ridge of $ \IE $ at $ x $} by
		$ 
		\Rid_x (\IE) := \Rid (\TC_x (\IE)) 
		$.
	\end{Def}
	
	\smallskip	
	
	Let us provide the \emph{idealistic interpretation} of the tangent cone $ \TC_x ( \IE ) $, the directrix $ \Dir_x ( \IE ) $ and the ridge $ \Rid_x ( \IE ) $ of $ \IE $ at $ x \in \Sing {\IE} $.
	First, we point out the following:
	
	\begin{Rk}
		\label{Rk:idealistic}
		Let $ \IE = (J,b) $ be a pair on $ Z $.
		By Lemma \ref{Lem:Basic}(1), $ \IE $ is equivalent to $ \IE^a := (J^a, a b ) $, for all $ a \in \IZ_+ $.	
		Let $ x \in \Sing ( \IE ) = \Sing (\IE^a) $ and $ R = \cO_{Z,x} $ as before.
		We denote by $ \resfield $ the residue field of $ R $.
		Let $ ( w ) = ( w_1, \ldots, w_n ) $ be a regular system of parameters for $ R $. 
		Then $ \gr_x ( Z ) \cong \resfield[W] = \resfield[W_1, \ldots, W_n] $, where $ W_i $ denotes the image of $ w_i $ in $ \maxIdeal / \maxIdeal^2 $ ($ \maxIdeal = \langle w_1, \ldots, w_n \rangle $), and $ \ini_x (\IE) = \langle \ini (f, b) \mid f \in J \rangle $.
		We want to show the following equality of ideals in $ \gr_x ( Z) $:
		$$
		\ini_x (\IE^a) = ( \ini_x (\IE) )^a. 
		$$
		
		Clearly, $ \ini (f + g , b) = \ini ( f , b ) +  \ini ( g, b) $ for $ f, g \in J $.
		Consider an element $ g \in J^a $ which is of the form $ g = g_1 \cdots g_a $, for $ g_1, \ldots, g_a \in J $.
		Since $ x \in \Sing ( \IE ) $ the initials $ \ini ( g_i, b ) $ are either zero or homogeneous of degree $ b $ ($ \ord_x ( g_i ) \geq b $), for all $ i \in \{ 1, \ldots, a \} $.
		Thus $ \ini (g, a b ) = \prod_{ i = 1}^a  \ini ( g_i, b ) $ and we get the desired equality $ \ini_x (\IE^a) = ( \ini_x (\IE) )^a $.
		
		If we put $ \ITC_x (\IE) := ( \ini_x (\IE), b ) $ and $ \ITC_x (\IE^a) := ( \ini_x (\IE^a), a b ) $, then the   last equation implies that these are equivalent pairs on $ T_x ( Z ) = \Spec (  \gr_x ( Z ) ) $.

		Let  $ I\Dir_x ( \IE ) = \langle Y_1, \ldots, Y_r \rangle $ be the ideal of the directrix with elements $ Y_j \in \resfield[W] , 1 \leq j \leq r $, which are homogeneous of degree one.
		By definition of the directrix, there exists a system of generators for $ \ini_x ( \IE ) $ which are contained in $ \resfield[Y_1,\ldots,Y_r] $ and $ (Y) $ is minimal with this property.
		This implies that there exists a system of generators for $ \ini_x ( \IE^a ) $ contained in $ \resfield[Y_1,\ldots,Y_r] $ and $ (Y) $ is also minimal:
		Suppose this is wrong, say there exist generators that are contained in $ \resfield[Z_1,\ldots,Z_s] $ for some $ s < r $.
		Then the same is true for the generators of $ \ini_x ( \IE ) $ which is a contradiction.
		This shows 
		$$ 
		\Dir_x (\IE) = 	\Dir_x (\IE^a).
		$$
		Thus the pairs $ \IDir_x ( \IE ) := (I\Dir_x ( \IE ), 1) $ and $ \IDir_x ( \IE^a ) := (I\Dir_x ( \IE^a ), 1) $ (defined on $ T_x ( Z ) $) are equal. 
		
		Now let  $ I\Rid_x (\IE) = \langle \sigma_1, \ldots, \sigma_s \rangle $  be the ideal of the ridge with additive homogeneous polynomials $ \sigma_{\textcolor{black}{j}} \in \resfield[ Y_1, \ldots, Y_r ] \subset \resfield [W] $ \textcolor{black}{of degree $ q_j =  p^{d_j} $ for some $ d_j \in \IZ_{\geq 0} $ ($ p =  \car{ \resfield } $)}, $ 1 \leq {\textcolor{black}{j}} \leq s $.
		\color{black}
		Giraud proved in \cite[3.3.3, 3.3.4, Definition 3.1(d)]{GiraudMaxPos} that there exist homogeneous generators $ G_1, \ldots, G_m $ of $ \ini_x (\IE) $ 
				such that
				every $ \sigma_j $ is given by $ \sigma_j = P_j (D^Y_A(G_i)\mid A,i ) $ for certain polynomials $ P_j = P_j (X_{i,A} \mid i, A ) $ 
				which are homogeneous of degree $ q_j $ if we assign to the variable $ X_{i,A} $ the weight $ b - |A| $. 
		Let $ D^Y_A $ be the Hasse-Schmidt derivative with respect to the variables $ (Y) $, for $ A = (A_1, \ldots, A_r) $ with $ |A|< b $.  
			The Diff Theorem (Proposition~\ref{Prop:Diff}) 
			implies that, for every $ i \in \{ 1, \ldots m \} $, we have 
			\begin{equation}
			\label{eq:Diff_Hasse_sim}
			(G_i , b ) \sim (G_i, b) \cap (D^Y_A (G_i), b - |A| ).
			\end{equation} 
			Note that $ D^Y_A (G_i) $ is homogeneous of degree $  b - |A| $ (if it is non-zero) since $ G_i $ is homogeneous of degree $ b $.
			Moreover, observe that Lemma~\ref{Lem:Basic}(2) and (3) imply
			\begin{equation}
			\label{eq:sum_and_product_polynomials}
			\left\{ 
			\begin{array}{c}
			(H_1,a ) \cap  (H_2, a ) \sim (H_1,a ) \cap  (H_2, a )  \cap (H_1 + H_2, a) \\[3pt]
			(H_1, a_1) \cap (H_2,a_2) \sim (H_1, a_1) \cap (H_2,a_2) \cap  (H_1 H_2, a_1 + a_2), 
			\end{array} 
			\right. 
			\end{equation}
			for $ H_1, H_2 \in \resfield[U,Y] $ and $ a, a_1, a_2 \in \IZ_+ $.
	Therefore, we have 
	\begin{equation}
	\label{eq:IT_sim_IDir}
	\hspace{-10pt}
	\left\{ 
	\hspace{-2pt} 
	\begin{array}{c} 
	\displaystyle 
	\ITC_x (\IE) =  ( \ini_x (\IE), b )
	\sim \bigcap_{i=1}^m (G_i,b) 
	\stackrel{\eqref{eq:Diff_Hasse_sim}}{\sim} 
	\bigcap_{i=1}^m (G_i,b) \cap \bigcap_{A,i} (D^Y_A (G_i), b - |A| )
	\stackrel{\eqref{eq:sum_and_product_polynomials}}{\sim} 
	\\[10pt]
	\displaystyle 
	\stackrel{\eqref{eq:sum_and_product_polynomials}}{\sim} 
	\underbrace{\bigcap_{i=1}^m (G_i,b)  \cap \bigcap_{A,i} (D^Y_A (G_i), b - |A| )}_{
		\displaystyle 
		\sim ( \ini_x (\IE), b ) } 
	\cap \bigcap_{j=1}^s \big( \, \underbrace{P_j (D^Y_A(G_i)}_{\displaystyle = \sigma_j}, \, q_j  \, \big)
	\sim
	\\[10pt]
	\displaystyle 
	{\sim} 
	( \ini_x (\IE), b )	\cap 
	\bigcap_{j=1}^s ( \sigma_j, q_j )
	\stackrel{\eqref{eq:sum_and_product_polynomials}}{\sim}
	\bigcap_{j=1}^s ( \sigma_j, q_j ) =: \IRid_x ( \IE ) ,
	\end{array}
	\hspace{-10pt} 
	\right. 
	\end{equation} 
	where we use in the last equivalence that $ \ini_x (\IE) $ has a generator systems contained in $ \resfield[\sigma_1, \ldots, \sigma_s] $ by definition of the ridge. 
	
	Let $ I\Rid_x (\IE^a ) = \langle  \gamma_1, \ldots, \gamma_t \rangle $,
	for additive polynomials $ \gamma_\ell \in \resfield[Y] $ homogeneous of degree $ p^{d_\ell} $, for $ 1 \leq \ell \leq t $. 
	Using the same arguments above, we get
	\[
	 ( \ini_x (\IE^a), a b )  
	\sim  \bigcap_{\ell=1}^t ( \gamma_\ell, p^{d_\ell}  )
	 =: \IRid_x ( \IE^a ).
	\]
	Using the previously shown equivalence, we get that
	\begin{equation}
	\label{eq:sim_in_ab_in_b_rid} 
		\IRid_x ( \IE^a ) \sim ( \ini_x (\IE^a), a b ) \sim ( \ini_x (\IE), b )  
		\sim  \IRid_x ( \IE ) 
	\end{equation} 
	are equivalent pairs on  $ T_x ( Z ) $.
	\end{Rk}
	
	This observation gives the hint that the tangent cone (resp. the ridge) of equivalent pairs might be related if we use an idealistic interpretation.
	Hence we introduce the following definitions of the tangent cone, the directrix and the ridge as pairs and prove that these actually provide well-defined idealistic exponents. 
	The variant of the tangent cone appeared already in the language of idealistic filtrations (\cite[Definition 1.1.2]{IFPObergurgl})
	and also in the language of Rees algebras
	(\cite[section 1.2]{AngelicaOrlandoElimination}), 
	but the concepts of directrix and ridge considered as pairs are new.
	
	\begin{Def}
		Let $ \IE = (J,b) $ be a pair on $ Z $ and $ x \in \Sing ( \IE ) $.
		Recall that $ \resfield $ denotes the residue field of $ Z $ at $ x $ and $ p =  \car{ \resfield } \geq 0 $.
		Let $ I\Dir_x ( \IE ) = \langle Y_1, \ldots, Y_r \rangle $ and $ I\Rid_x ( \IE ) = \langle \sigma_1, \ldots, \sigma_s \rangle $ for elements $ Y_j $ homogeneous of degree one, $ 1 \leq j \leq r $, and additive homogeneous polynomials $ \sigma_i $ of order $ p^{d_i} $, $ 1 \leq i \leq s $. 	
		We define the following pairs on $ T_x (Z) = \Spec ( \gr_x (Z) ) $:
		\begin{center}
			\begin{tabular}{ll}
				$ \ITC_x ( \IE ) = (\, \ini_x ( \IE ),\, b \,) $ & \em (idealistic tangent cone of $ \IE $ at $ x $), \\[5pt]
				$ \IDir_x ( \IE ) = (\, I\Dir_x ( \IE ),\, 1 \,) $  & \em (idealistic directrix of $ \IE $ at $ x $), \\[5pt]
				$ \IRid_x ( \IE ) = \bigcap_{i = 1 }^s (\, \sigma_i ,\, p^{d_i} \,) $  & \em (idealistic ridge of $ \IE $ at $ x $) .
			\end{tabular} 
		\end{center}	
	\end{Def}
	
	\smallskip
	
	\begin{Rk}
		\begin{itemize}
			\item[(1)]
			By Remark \ref{Rk:idealistic}, we have for a pair $ ( J, b ) $ and a positive integer $ a \in \IZ_+ $
			$$ 
			\begin{array}{c}
			\ITC_x (J,b) \sim \ITC_x (J^a, a b) ,\;\; 
			\IDir_x (J,b) = \IDir_x (J^a, a b)  ,\\[2pt]
			\IRid_x (J,b) \sim \IRid_x (J^a, a b ) .
			\end{array}
			$$ 
			\item[(2)]
			Let $ \IE_1 = (J_1, b_1) $ and $ \IE_2 = ( J_2, b_2) $ be two pairs on $ Z $ and $ x \in \Sing ( \IE_1 \cap \IE_2 ) $.
			Suppose $ b_1 = b_2 = : b $.
			By definition,
			$ \IE_1 \cap \IE_2 = ( J_1 + J_2 , b) $ and hence	
			$$
			\ini_x (J_1 + J_2 , b) =
			\ini_x ( J_1, b ) + \ini_x ( J_2, b ).
			$$
			
			For $ b_1 \neq b_2 $, let $ c := \mathrm{lcm}(b_1, b_2) $ and $ c_i := \frac{c}{b_i} $, $ i \in \{ 1, 2 \} $.
			Then the definition of $ \IE_1 \cap \IE_2 $, the previous, and Remark \ref{Rk:idealistic} imply
			$$
			\ini_x (\IE_1 \cap \IE_2 ) = \ini_x (\IE_1)^{c_1} + \ini_x (\IE_2)^{c_2}.
			$$
			
			Therefore we get: 
			$$
			\begin{array}{rcl}
			\ITC_x ( \IE_1 \cap \IE_2 ) & \sim &  \ITC_x ( \IE_1 ) \, \cap \,  \ITC_x (\IE_2 ) ,
			\\[5pt]
			\IDir_x ( \IE_1 \cap \IE_2 ) & = & \IDir_x ( \IE_1 ) \,\cap\,  \IDir_x ( \IE_2 )  , 
			\\[5pt]
			\IRid_x ( \IE_1 \cap \IE_2 ) & \sim &  \IRid_x ( \IE_1 ) \, \cap \,  \IRid_x ( \IE_2 ) .
			\end{array} 
			$$	
		\end{itemize}
	\end{Rk}
	
	\smallskip
	
	\begin{Prop}
		\label{Prop:SubsetAndSing}
		Let $ \IE = (J,b) $ be a pair on $ Z $ and $ x \in \Sing ( \IE ) $.
		We have
		\begin{enumerate}
			\item 
			$ \IDir_x ( \IE ) \subset \IRid_x( \IE ) {\, \red{\sim} \, } \ITC_x(\IE) $.
			\item 
			$ \Dir_x (\IE) =  \Sing (\IDir_x(\IE))  \subseteq \Sing (\IRid_x(\IE)) \subseteq \Sing (\ITC_x(\IE))$.
		\end{enumerate}
	\end{Prop}
	
	\begin{proof}
		Let $ (Y) = ( Y_1, \ldots, Y_r ) $ be elements (homogeneous of degree one) which determine $ \Dir_x (\IE) $ and extend these by $ (U) = (U_1, \ldots, U_e ) $ such that $ \gr_x ( Z ) \cong \resfield [U,Y] $.
		Further let $ (\sigma) = ( \sigma_1, \ldots ,\sigma_s ) $ be additive homogeneous polynomials which yield $ \Rid_x ( \IE ) $.

		Let $ ( T ) = (T_1, \ldots, T_a ) $ be a set of variables that are independent of $ ( U , Y ) $.
		If we pass from $ \resfield[U,Y] $ to $ \resfield[U,Y] \times_\resfield \resfield[T] = \resfield [U, Y , T] $, then the situation does not change since the generators of all three pairs are homogeneous.
		Hence, it suffices to consider the pairs as pairs on $  \resfield [U,Y] $.
		{\red{Since by definition there exist generators of $ \ini_x ( \IE ) $ which are contained in $ \resfield [\sigma] \subset \resfield [Y] $, any center which is permissible for $ \IDir_x ( \IE ) $ 
		is so for $ \IRid_x( \IE ) $. 
		After blowing up, either  $ \IDir_x ( \IE ) $ 
		is resolved or the situation is still the same.
		This shows the first inclusion of (1)}}.
		\red{The remaining equivalence  
		$
			\IRid_x(\IE) \sim \ITC_x (\IE)
		$
		has already been discussed in \eqref{eq:IT_sim_IDir}}.

		The first equality of (2) follows by the definition of the singular locus of a pair 
		and (1) implies the rest via the Numerical Exponent Theorem, Proposition \ref{Prop:NumExp}.
	\end{proof}
	
	In characteristic zero or if the characteristic $ p  > 0 $ is greater than $ b $, we have the following
	
	\begin{Cor}
		\label{Cor:DirRidITCchar0}
		Let $ \IE = (J,b) $ be a pair on $ Z $ and $ x \in \Sing (\IE) $.
		Assume $ \car{ \resfield }  = 0 $ or $ b < \car{ \resfield }  $, where $ \resfield $ denotes the residue field of $ Z $ at $ x $.
		Then
		$$
		\IDir_x ( \IE ) \sim \IRid_x( \IE ) \sim \ITC_x(\IE).
		$$
		In particular, $ \Dir_x (\IE) =  \Sing (\IDir_x(\IE)) = \Sing (\IRid_x(\IE)) = \Sing (\ITC_x(\IE)) $.
	\end{Cor}
	
	\begin{proof}
		By Proposition~\ref{Prop:SubsetAndSing} we only have to show $ \ITC_x(\IE) \subset \IDir_x ( \IE ) $.
		Let $ (R = \cO_{Z,x}, \maxIdeal, \resfield ) $ be the local ring of $ Z $ at $ x $ and consider a regular system of parameters $ (u,y) = (u_1, \ldots, u_e, y_1, \ldots, y_r) $ for $ R $ such that $ I\Dir_x ( \IE ) = \langle Y_1, \ldots, Y_r \rangle $, where $ Y_j $ denotes the image of $ y_j $ in $ \maxIdeal / \maxIdeal^2 $.
		Then  
		$$
		\IDir_x ( \IE ) = ( \langle Y_1, \ldots, Y_r \rangle, 1 ) \sim  (Y_1, 1) \cap \ldots \cap (Y_r, 1 ).
		$$
		Recall that $ \ITC_x(\IE)  = ( \ini_x (\IE), b ) $.
		By definition of the directrix, there exists a system of generators for $ \ini_x (\IE) $ such that every element is contained in $ \resfield [Y] $ and each $ Y_j $ appears to a non-zero power.
		Hence they lie in $ \langle Y \rangle^b \setminus \langle Y \rangle^{ b + 1 } $ and every element $ F \in \ini_x (\IE) $ in the given system of generators can be written as 
		$$
		F = \sum_{ \substack{B \in\IZ^r_{ \geq 0 } \\[3pt] |B| = b} } \coeff_B \,Y^B,
		$$
		for $ \coeff_B \in \resfield $.
		Further, for every $ 1 \leq j \leq r $, there exists a generator $ F(j) \in \ini_x (\IE) $ such that there is $ B(j) = (B_1, \ldots, B_r) \in \IZ_{ \geq 0 }^r $ with $ \coeff_{B(j)} \neq 0 $ and $ B_j \geq 1 $.
		Therefore, this is an element of $ \ini_x (\IE) $, where $ Y_j $ appears.
		
		Set $ M (j) := B(j) - e_j \in \IZ^r_{ \geq 0 } $.
		(Here $ e_j \in  \IZ^r_{ \geq 0 } $ denotes the $ j $-th unit vector with zero everywhere except the $ j $-th place, at which there is a one).
		Note that $ | M(j) | = b - 1 $.
		Let $ \cD_{ M(j) } \in \Diff_\resfield^{ \leq b - 1 } ( \resfield [Y] ) $ be the differential operator which is defined via
		$$
		\cD_{ M(j) } ( \coeff \, Y^B ) = \binom{B}{ M(j) } \, \coeff \, Y^{ B - M(j) }.
		$$	
		for $ \coeff \in \resfield $ and $ B \in \IZ^r_{ \geq 0 } $.
		In particular, $ \cD_{ M(j) } ( \coeff \, Y^{ B(j) } ) =  \coeff \,\binom{ B(j) }{ M(j) } \, Y_j = \coeff \,B_j \, Y_j $ and $ \cD_{ M(j) } ( \coeff \, Y^B ) = 0 $ if $ | B | = b $ and $ B \notin M(j) + \IZ_+^r $;
		consequently, 
		$$ 
		\cD_{ M(j) } (F (j)) = \coeff_{B(j)} \, B_j \, Y_j + \sum_{B'(i)} \coeff_{ B'(i) } \, B'_i \, Y_i,
		$$
		where $ B' ( i ) = ( B'_1, \ldots, B'_r ) \in \{ M(j) + e_i \mid i \in \{ 1, \ldots, r \} \setminus \{ j \} \} $.
		Since we have $ 1 \leq B_j \leq b $ and $ \car{ \resfield }  = 0 $ or $ b < \car{ \resfield }  $, we get that $  B_j $ (and thus $  \coeff_{B(j)}\, B_j $) is a unit in $ \resfield $.
		We define $ Y^\ast_j \in \resfield [ Y ] $ via
		$$
		Y^\ast_j := ( \coeff_{B(j)}\, B_j  )^{ - 1 } \, \cD_{ M(j) } \, F(j) = Y_j + \sum_{B'(i)} ( \coeff_{B(j)}\, B_j  )^{ - 1 }\coeff_{ B'(i) } \, B'_i \, Y_i .
		$$
		Let $ j = 1 $.
		We choose in $ R $ a system of representatives for $ \resfield = R / \maxIdeal $ and define with this $ y^\ast_1 \in R $ by replacing $ (Y) $ by $ ( y ) $ in the definition of $ Y^\ast_1 $.
		The system $ ( y^\ast_1, y_2, \ldots, y_r ) $ fulfills the same properties as $ ( y ) $. 
		So we may consider the regular system of parameters $ ( u; y^\ast_1, y_2, \ldots, y_r ) $ instead of $ ( u, y ) $ and put $ \cD_1 := \cD_{M(1)} $.
		(Note that $ \cD_1 $ is defined in terms of $ ( Y ) $).
		We repeat the above procedure to obtain $ y_2^\ast $ and $ \cD_2 $.
		We continue until we obtain $ ( y^\ast) = (y_1^\ast, \ldots, y_r^\ast ) $.
		The Diff Theorem \ref{Prop:Diff} yields, for all $ j \in \{ 1, \ldots, r \} $, that $ ( F(j), b ) \subset  ( \cD_j F(j), 1 )  = ( Y_j^\ast, 1 ) $.
		This implies
		\[
		\ITC_x(\IE)  = ( \ini_x (\IE), b ) \subset ( Y_1^\ast, 1) \cap \ldots \cap ( Y_r^\ast, 1 ) =  \IDir_x ( \IE ).
		\]
	\end{proof}
	
	\begin{Rk}
		\label{Rk:failurePoscharMax}
		For the arbitrary case the equivalences need not hold.
		One reason is that $ B_j $ may be zero in $ \resfield $. 
		Thus the assumption $ \car{ \resfield }  = 0 $ or $ b < \car{ \resfield }  $ is essential.
		%
		If $ \resfield $ is perfect, then we have that
		$ \IDir_x ( \IE ) = (Y_1, 1) \cap \ldots \cap (Y_r, 1 ) \sim (Y_1^{p^{d_1}}, p^{d_1}) \cap \ldots \cap (Y_r^{p^{d_r}}, p^{d_r} )  = \IRid_x ( \IE ) $, for certain $ d_j \in \IZ_{ \geq 0 }$, i.e., Corollary \ref{Cor:DirRidITCchar0} is also true in this case.
	\end{Rk}
	
	\smallskip
	
	There is not necessarily a relation between the tangent cones $ \TC_x ( \IE ) $ of equivalent pairs.
	(For example, $ ( x^2 y - z^3 , 3) \sim ( x^2, 2 ) \cap ( x^2 y - z^3 , 3)  \sim (x,1) \cap ( y,1) \cap ( z,1) $ over $ \IC $).
	For idealistic interpretations we have the following strong result.
	
	\begin{Prop}[Proposition \ref{MainProp:IdTanRidDir}]
		\label{Prop:ItcDirRidUnique}
		Let $ \IE_1 = (J_1, b_1) $ and $ \IE_2 = (J_2, b_2) $ be two pairs on $ Z $ with $ \IE_1 \subset \IE_2 $ and $ x \in \Sing (\IE_1) \subseteq \Sing (\IE_2) $.
		Then we have 
		\begin{enumerate}
			\item 
			$ \ITC_x (\IE_1) \subset  \ITC_x (\IE_2) $.
			\item 
			$ \Dir_x (\IE_1) \subseteq \Dir_x (\IE_2) $ and hence $ \IDir_x (\IE_1) \subset \IDir_x (\IE_2) $.
			\item 
			$ \IRid_x (\IE_1) \subset \IRid_x (\IE_2) $.
		\end{enumerate}
		By symmetry we get equivalence $ \sim $ and equality if $ \IE_1 \sim \IE_2 $.
	\end{Prop}
	
	This yields that the idealistic version of the tangent cone, the directrix and the ridge are uniquely determined by $ x $ and the equivalence class of $ \IE $, i.e., by the idealistic exponent $ \IE_\sim $.
	Thus $ \IE_\sim $ yields well-defined idealistic exponents $ \ITC_x (\IE)_\sim $, $ \IDir_x (\IE)_\sim $, and $ \IRid_x (\IE)_\sim $.
	In the special situation over perfect fields the implication $ \IE_1 \sim \IE_2 \Rightarrow \Dir_x (\IE_1)  = \Dir_x (\IE_2)  $ was already proven in \cite[Proposition 19.2]{HiroIdExp}.
	
	\begin{Rk}
		For the proof, we need the following: 
		Let $ f \in R $ and $ b \in \IQ_+ $ with $ \ord_M(f) \geq b $.
		We fix a regular system of parameters $ ( w ) = (w_1, \ldots, w_n) $ for $ R $.
		Then $ f $ has a finite expansion in $ R $ of the form
		$$
		f = \sum_{N} C_N w^N,
		$$	
		for $ C_N \in R^\times \cup \{ 0 \} $
		(since $ R = \cO_{Z,x} $ is a regular local ring, in particular Noetherian, and $ R \subset \widehat R $ is faithfully flat, where $ \widehat R $ denotes the completion of $ R $ with respect to its maximal ideal,
		see \cite[Proposition~2.1]{CPdim3}).
		We split this expansion into
		$$ 
		f = f_0 + h , 
		$$
		$$
		\mbox{ where } 
		f_0 := \sum_{N: |N|=b} C_N w^N
		\ \mbox{ and } h := f - f_0 \in \langle w  \rangle^{b+1}.
		$$ 
		Note that $ \ini(f,b) = \ini(f_0,b) $.
		If $ \ord_M (f) > b $ (which is for example the case if $ b \in \IQ_+ \setminus \IZ_+$), then $ f_0 \equiv 0 $.
		
		If $ \ord_M(f) = b $, then the residue classes $ C_N \mod M $ (with $ |N|=b $) are unique since they are determined by $ \ini(f,b) $.
		On the other hand, if $ C_N \neq 0 $, then the units $ C_N $ are not necessarily unique; 
		for example, for $ f = w_1^2 + w_1^5 + w_2^3 $ two possible choices are $ f_0 (w) := w_1^2 $ or $ \widetilde{f_0} :=  (1 + w_1^3) w_1^2 $. 
	\end{Rk} 
	
	\begin{proof}[Proof of Proposition \ref{Prop:ItcDirRidUnique}]
		By Remark \ref{Rk:idealistic} the result holds in the case $ \IE = (J, b) \sim (J^a, a b) $ for some $ a \in \IZ_+ $.
		Hence it suffices to consider $ b_1 = b_2 = b $.
	 	Let $ R = \cO_{Z,x} $ with maximal ideal $ \maxIdeal $ and residue field $ \resfield $
		and $ (w) = (w_1, \ldots, w_n) $ be a regular system of parameters for $ R $.
		
		Let $ \IE = (J,b) \in \{ (\IE_1)_x, (\IE_2)_x \} $.	
		We extend $ R $ to $ R_0 := R \times_\IZ \IZ[t]
		$, where  $ t $ is an independent variable.
		We use the notation $ \IE_0 := ( J^{(0)} := J[t], b ) $ and  $ V_0 := \Spec ( R_0 ) $.
		Let $ L_0 \subset V_0 $ be the line $ V(w) $ and $ x_0 \in L_0 \subset Z_0 $ be the closed point $ V(w,t) $.
		For $ \alpha \in \IZ_+ $, we consider the following {\LSB}, which is permissible for $ \IE_0 $ (since $ x \in \Sing ( \IE) $),
		\begin{equation}
		\label{diagram1again}
		\begin{array}{ccccccccc}
		L_0	&	\cong	&	\hspace{30pt} L_1	&	\cong	&	\ldots	&	\hspace{45pt}	L_{\alpha-1}	&	\cong	&	\hspace{30pt} L_\alpha \\[5pt]
		\cap	&&	\hspace{30pt}	\cap	&&			&	\hspace{45pt}	\cap && \hspace{30pt}\cap	\\[5pt]
		V_0		& \stackrel{\pi_1}{\leftarrow} &	Z_1 \supset V_1		& \stackrel{\pi_2}{\leftarrow} &	\ldots	&	Z_{\alpha-1}  \supset V_{\alpha-1}	& \stackrel{\pi_\alpha}{\leftarrow} 	&	Z_\alpha \supset V_{\alpha}	\\[5pt]
		\vin	&&	\hspace{30pt}	\vin	&&			&	\hspace{45pt}	\vin && \hspace{30pt}\vin	\\[5pt]
		x_0		&&	\hspace{30pt}	x_1		&&	\ldots	&	\hspace{45pt}	x_{\alpha-1}	&&	\hspace{30pt}x_{\alpha},	
		\end{array}
		\end{equation} 
		where $ \pi_{i}: Z_{i} \rightarrow V_{i-1} $ is the blowup with center $ x_{i-1} \in V_{i-1} $, $ L_i \cong L_0 $ is the strict transform of $ L_0 $, $ x_i $ denotes the unique intersection point of $ L_i $ with the exceptional divisor of the blowup $ \pi_{i} $ and $ V_i = \Spec (R_i) \subset Z_i $ is the $ T $-chart of the blowup, $ i \in \{ 1, \ldots, \alpha \} $.
		Recall that $ L_0 = V(w) $, hence $ x_i = V \left( \frac{w}{t^i}, t \right) $, $ L_i = V \left( \frac{w}{t^i} \right) $ and $ \left( \frac{w}{t^i}, t \right) $ are the variables in $ R_i $.
		
		Let $ f \in J $ be an arbitrary element.
		Consider an expansion of $ f $ in $ R $ as in the previous remark, $ f (w) = f_0 (w) + h (w) $, where $ f_0(w) = \sum_{N: |N|=b} C_N w^N $ and $ h (w) \in \langle w \rangle^{ b + 1 } $.
		As pointed out before, $  f_0 (w) $ is not unique since only the residue classes $ C_N \mod \langle w \rangle $, but not $ C_N $ itself are unique. 
		Nonetheless, this will not affect our arguments below as the important property for $ C_N \in R^\times $ is that its residue with respect to $ \langle \frac{w}{t^i}, t  \rangle $ does not change. 
		 
		Set 
		$ 
		d := d(f) := \ord_x ( h ) $. 
		By construction, we have
		$ d > b $.
		The transform of $ f $ in $ V_\alpha $ is given by
		\begin{equation}
		\label{niceform}
		f^{ ( \alpha ) } \left( \frac{w}{t^\alpha}, t \right) 
		= f^{ ( \alpha ) }_0 \left( \frac{w}{t^\alpha} \right) + t^{ \alpha \cdot ( d - b ) } \cdot h_\ast \, ,
		\hspace{10pt}
		\mbox{ for some } 
		h_\ast \in  \left\langle \frac{w}{t^\alpha} \right\rangle^{ b + 1}.
		\end{equation}
		Recall that $ x_{\alpha} = V( \frac{w}{t^\alpha}, t ) $.
		It is clear that the generators of the ideal of the tangent cone (and thus also its idealistic version) at $ x $ did not change under the extension to $ R_0 $.
		By the above, we see that the tangent cone at $ x_\alpha $ is the same as the one at $ x_0 $ before the permissible {\LSB} (\ref{diagram1again});
		just replace in $ \ini_x ( J^{(0)}, b ) $ the coordinates $ ( w ) $ by $ ( \frac{w}{t^\alpha} ) $ in order to get  $ \ini_{x_{\alpha}} ( J^{(\alpha)}, b ) $.
		
		Given $ \IE_1 \subset \IE_2 $, then we can perform the above permissible {\LSB} and get $ \IE_1^{ ( \alpha ) } \subset \IE_2^{ ( \alpha ) } $ on $ V_\alpha $.
		Further, every $ f^{(\alpha)} \in J_1^{(\alpha)} $ and $ g^{(\alpha)} \in J_2^{(\alpha)}  $ is of the form (\ref{niceform}).
		Now choose $ \alpha $ large enough such that 
		\begin{equation}
		\label{eq_altes_ast}
		\alpha \cdot ( d(f) - b ) \geq b 
		\hspace{10pt} \mbox{ and } \hspace{10pt}
		\alpha \cdot ( d(g) - b ) \geq b 
		\end{equation} 
		for every $ f^{(\alpha)} \in J_1^{(\alpha)} $ and $ g^{(\alpha)} \in J_2^{(\alpha)}  $.
		
		For simplicity, let us drop the indices and assume from the very beginning that $ \IE_1 \subset \IE_2 $ on $ V_0 $ are of the special type described above.
		By the previous discussion this is justified.
		As usual capital letters $ (W, T) $ denote the images of $ ( w, t) $ in $ \langle w, t \rangle / \langle w, t \rangle^2 $. 
		We want to point out that by (\ref{niceform}) the generators of $ \ini_x (\IE_1) $ and  $ \ini_x (\IE_2) $ are contained in $ \resfield [W] $.
		Hence we consider $ \ITC_x (\IE_1) $ and  $ \ITC_x (\IE_2) $ as pairs on $ \Spec ( \resfield [W] ) $.
		
		Since the tangent cones are generated by homogeneous elements, an extension by independent variables $ (t') = ( t'_1, \ldots, t'_a) $, for some $ a \in  \IZ_+ $, does not affect the situation.
		So it suffices to consider the case without an extension of the base.
		
		\smallskip
		
		For (1) we first want to show
		\begin{equation}
		\label{ITCsubset}	
		\ITC_x (\IE_1) \subset  \ITC_x (\IE_2).
		\end{equation}	
		Suppose this is wrong. 
		Then there exists a {\LSB} $ (\diamondsuit) $ over $ \resfield [W] $ which is permissible for $ \ITC_x (\IE_1) $, but not for $ \ITC_x (\IE_2) $.
		By (\ref{niceform}), $ \ini_x (\IE_1) $ is generated by the images of $ f_0 ( w ) $ in the graded ring and $ \ini_x (\IE_2) $  by the images of $ g_0 ( w ) $ (for $ f \in J_1 $ and $ g \in J_2 $).
		We can lift the centers of $ (\diamondsuit) $ back to $ R $ (by choosing a system of representatives of $ \resfield = R / \maxIdeal $ in $ R $ and using $ (w) $ instead of $ (W) $).
		Moreover, we can intersect the centers with $ V(t) $ and obtain a LSB over $ R_0 $.
		Because of the special form (\ref{niceform}) we get by blowing up these modified centers a {\LSB} $ ( \widetilde{ \diamondsuit } ) $ over $ V_0 $, which is permissible for $ \IE_1 $ by the permissibility of $ (\diamondsuit) $ and property \eqref{eq_altes_ast} of $ \alpha $.
		But since $ (\diamondsuit) $ is not permissible for $ \ITC_x (\IE_2) = (  \ini_x (\IE_2), b ) $, we also have that $ ( \widetilde{ \diamondsuit } ) $ is \emph{not} permissible for $ \IE_2 $.
		This contradicts $ \IE_1 \subset \IE_2 $ and proves (1).
		
		\smallskip
		
		We come to (2), $ \Dir_x (\IE_1) \subseteq \Dir_x (\IE_2) $.
		By Proposition~\ref{Prop:SubsetAndSing} $ \Dir_x ( \IE_i ) \subseteq \Sing ( \ITC_x ( \IE_i ) ) $ and by definition of the directrix, it is a permissible center for $  \ITC_x ( \IE_i ) $, $ i \in \{ 1, 2 \} $.
		Further (\ref{ITCsubset}) implies that $ \Dir_x (\IE_1) $ is a permissible center for $  \ITC_x ( \IE_2 ) $, which contains the origin.
		By the minimality of the directrix $ \Dir_x (\IE_2) $ any permissible center containing the origin must lie in $ \Dir_x (\IE_2) $.
		This implies $ \Dir_x (\IE_1) \subseteq \Dir_x (\IE_2) $.
		The second part of (2) follows by Lemma~\ref{Lem:Basic}(2).
		
		\smallskip
		
		Part (3), $ \IRid_x (\IE_1) \subset \IRid_x (\IE_2) $, \red{is a consequence of Proposition~\ref{Prop:SubsetAndSing}(1).
		We get}
		\[
		\red{\IRid_x(\IE_1) 
		\sim
		\ITC_x(\IE_1)
		\stackrel{\eqref{ITCsubset}}{\subset}  
		\ITC_x(\IE_2)
		\sim
		\IRid_x(\IE_2).}
		\]
		
		\vspace{-18pt}

	\end{proof}

	\color{black}

	We end this section by discussing the role of the codimension of the ridge of an idealistic exponent 
	as an invariant for desingularization. 
	First, we prove the following result:

	\begin{Thm}
		\label{Thm:codim_invariant} 
		Let $ ( I, b ) \sim ( I' , b ' ) $ 
		be two equivalent pairs
		with non-zero homogeneous additive ideals $I$ resp.~$I'$
		in $ \resfield[X] := \resfield[X_1, \ldots, X_n] $ generated by homogeneous additive elements of degrees $b$ resp. $b'$.
		Then $ \operatorname{codim}(I) = \operatorname{codim}(I') =: \tau_{st} $. 
	\end{Thm}
	
	\begin{proof}
		If $ \car{\resfield} =  0 $, then we must have $ b = b ' = 1 $. 
		Further, $ \tau_{st} $ is just the codimension of the singular locus of $ ( I, b) $, which coincides with the directrix of $ ( I, b ) $ by Corollary~\ref{Cor:DirRidITCchar0}.
		Hence, Proposition~\ref{Prop:ItcDirRidUnique} implies the assertion in characteristic zero.
		
		Suppose $ \car{\resfield} = p > 0 $. 
		The additive hypothesis provides $ b = p^c $, for some $ c \in \IZ_{ \geq 0 } $.
		Up to renumbering the variables $ X_i $,
		there exist generators of $ I  $ of the form
		\[
			\sigma_1 = X_1^b + \sum_{i=2}^n \lambda_{1,i} X_i^b,
			\ \ 
			\sigma_2 = X_2^b + \sum_{i=3}^n \lambda_{2,i} X_i^b,
			\ \ 
			\ldots, 
			\ \
			\sigma_{\tau} = X_\tau^b  + \sum_{i=\tau + 1}^n \lambda_{\tau,i} X_i^b,
		\]  
		for $ \lambda_{j,i} \in K $. 
		Add variables $ ( T) = ( T_1 , \ldots, T_m) $ 
		and in $ \resfield[X, T] $,
		let us look at the maximal ideal 
		$ \langle X_1, \ldots, X_n, Q_1 (T_1), \ldots, Q_m(T_m) \rangle $,
		where $ Q_j ( T_j ) \in \resfield [T_j] $ 
		are irreducible and 
		the residues of $ \lambda_{j,i} $ are all $ b $-powers.
		At the corresponding closed point $ x' \in \Spec(\resfield[X,T]) $,
		the reduced tangent cone of $ (I,b) $ is a linear space of codimension $ \tau $.
		It is the codimension of the singular locus of $ \ITC_{x'} (\IE_\sim) $,
		where $ \IE_\sim $ is the idealistic exponent defined by $ (I,b) \sim (I', b') $ and 
		$ \ITC_{x'} (\IE) $ is its tangent space.
		Hence the codimensions coincide and $ \tau = \tau_{st} $. 
	\end{proof}

	\begin{Cor}
		Let $ \IE_\sim $ be an idealistic exponent on a regular scheme $ Z $ defined by a pair $ \IE = ( J,b) $.
		For all $ x \in \Sing (\IE_\sim) $,
		define $ \tau_{st} (x) $ as the codimension of $ \Rid_x ( \IE) $ in $ T_x(Z) $. 
		Then $ -\tau_{st} $ is upper semi-continuous along $ \Sing(\IE_\sim) $.
	\end{Cor}

	Note that $ \tau_{st} (x) $ is an invariant of the idealistic exponent by Theorem~\ref{Thm:codim_invariant}. 
	Hence, the previous definition makes sense.  
	
	\begin{proof}
		First, assume that $ b = 1 $.
		We use the criterion for upper semi-continuity \cite[Lemma~1.34(a)]{CJS}.
		Therefore, we have to prove:
		\begin{enumerate}
			\item[(i)]
			If $ x, y \in \Sing ( \IE) $ and $ x \in \overline{\{ y \}} $, then $ - \tau_{st}(x) \geq - \tau_{st} (y) $. 
			
			\item[(ii)]
			For every $ y \in \Sing (\IE) $ there exists a dense open subset $ U \subset \overline{ \{ y \} } $ such that $ -\tau_{st} (x) = - \tau_{st} (y) $ for all $ x \in U $ . 
		\end{enumerate}
		First, we show (i):
		Let $ V \subset Z $ be an affine subset containing $ x $ and $ y $.
		Let $ f_1, \ldots, f_m $ be local generators for $ J $ in $ V $.
		We assume that $ \ord_y (f_1) = \ldots = \ord_y (f_a) = 1 < \min \{ \ord_y (f_i)  \mid i \in \{ a + 1, \ldots, m \} \} $
		and that $ ( f_1, \ldots, f_a) $ forms a regular sequence, i.e., $ f_i \notin \langle f_1, \ldots f_{i-1} \rangle $ for all $ i \in \{ 2 , \ldots, a \} $. 
		Then, we have $ a = \tau_{st} (y) $.  
		Since the order is upper semi-continuous and $ x \in \overline{ \{ y \} } $,
		we have $ \ord_x ( f_i ) \geq \ord_y (f_i) $.
		This implies that $ \tau_{st} (x) \leq \tau_{st} (y) $, which is the desired equality. 
		
		We come to (ii): 
		Let $ V' \subset Z $ be an affine neighborhood of $ y $.
		Let $ f_1, \ldots, f_m $ be local generators for $ J $ in $ V' $ as in the proof of (i).
		In particular, $ a = \tau_{st} (y) $.  
		Since the order is upper semi-continuous,
		there exists a dense open subset $ U \subset \overline{\{ y \}} $ such that 
		$ \ord_x (f_i) = \ord_y (f_i) $ for all $ i \in \{ 1, \ldots, m \} $ and every $ x \in U $. 
		Therefore, we have $  - \tau_{st} (x) = - \tau_{st} (y) $ for every $ x \in U $.
		
		In conclusion, $ - \tau_{st} $ is upper semi-continuous along $ \Sing ( \IE) $ if $ b = 1 $.

		\smallskip

		Suppose $ b \geq 2 $. 
		Consider an affine subset $ U \subset Z $. 
		let $ f_1, \ldots, f_m $ be local generators for $ J $ in $ U $.
		Let $ T_1, \ldots, T_m, S_1, \ldots, S_m $ be a set of independent new variables. 
		We define $ g_i := T_i S_i^{b-1} + f_i $ for $ i \in \{ 1, \ldots , m \} $.
		Set $ \widetilde{J} := \langle g_1, \ldots, g_m \rangle $ and $ \widetilde{\IE} := (\widetilde J, b ) $, which is a pair on $ U $.
		We have
		\begin{equation}
			\label{eq:proof} 
			\widetilde \IE \sim (\langle T, S \rangle , 1 )\cap ( \langle f_1, \ldots, f_m \rangle , b).
		\end{equation} 
		This implies $ \Sing( \IE) \cap U \cong \Sing ( \widetilde \IE )  $ and $  \Sing ( \widetilde \IE ) $ is contained in the locus, 
		where $ T_i = S_i = 0 $ for all $ i \in \{ 1, \ldots, m \} $. 
		By \eqref{eq:proof},
		$ \IRid_x ( \widetilde \IE ) \sim (\langle T, S \rangle ,1 ) \cap \IRid_x (\IE) $,
		for $ x \in \Sing (\IE) \cap U $. 
		Therefore, we get
		\begin{equation}
		\label{eq:proof2}
			\operatorname{codim}_x (\Rid_x (\widetilde \IE)) =
			\operatorname{codim}_x (\Rid_x (\IE)) + 2 b  
		\end{equation}
		Since we have $ \ord_x ( \widetilde J)  = b $, 
		for every $ x \in \Sing ( \widetilde \IE ) $,
		we get $ \IRid_x ( \widetilde \IE ) = \Rid_x ( \widetilde J ) $,
		where the latter is the ridge of the tangent cone of $ \widetilde J $ at $ x $. 
		In particular, the codimension of both coincide. 
		By \cite[Th\'eor\`eme~1.2]{CPS}, the function defined by 
		$ x \mapsto - \operatorname{codim}_x ( \Rid_x ( \widetilde J )) $ is upper semi-continuous.
		(Note that the Hilbert-Samuel of $ \widetilde J $ function is constant along the singular locus of $ \widetilde \IE $ 
		as $ ( g_1, \ldots, g_m ) $ is a standard basis at every point in there -- for the definition of a standard basis, see \cite[Definition~2.7(1)]{HomeworkDim2}, for example.)  
		Therefore, \eqref{eq:proof2} provides that $ - \tau_{st} $ is upper semi-continuous. 
	\end{proof}
	
	\begin{Cor}
		Let $ \IE_\sim $ be an idealistic exponent on a regular scheme $ Z $ defined by a pair $ ( J,b) $ and $ x \in \Sing (\IE) $.
		Let $ \pi  \colon Z' \to Z $ be a permissible blow-up for $ \IE $
		and $ x' \in \pi^{-1} (x) $.
		Suppose $ x' \in \Sing (\IE')\subset Z' $, where $ \IE' $ is the transform of $ \IE $.
		Then, we have 
		\[
			\tau_{st} (x') \geq \tau_{st} (x). 
		\] 
	\end{Cor}

	\begin{proof}
		By \cite[Proposition 4.2, p.~II-33]{GiraudEtude} 
		(for $ N = b $), we have
		that the dimension of the ridge does not increase 
		after blowing up a permissible center. 
		Since $ \tau_{st} (x) $ is defined as the codimension of the ridge, 
		we obtain the desired inequality.
	\end{proof}

	\color{black}

	%
	%
	%
	%
	%
	%
	%
	%
	%
	%
	%
	%
	%

	\section{Characteristic polyhedra of pairs and idealistic exponents}
	\label{sec:3}
	
	We introduce the Newton polyhedron and the characteristic polyhedron of a pair.
	The latter is achieved by imitating the construction of Hironaka's characteristic polyhedron of a singularity.
	After that we discuss the notion of the characteristic polyhedron of an idealistic exponent.
	
	\smallskip
	
	Let $ \IE  = ( J, b ) $ be a pair on $ Z $ and $ x \in \Sing (\IE) $.
	Denote as usually by $ ( R = \cO_{Z,x}, \maxIdeal, \resfield) $ the local ring of $ Z $ at $ x $.
	By abuse of notation we skip the index $ x $ and write $ \IE = (J,b) $ instead of $ \IE_x $. 	
	Fix a system $ (u) = ( u_1, \ldots, u_e)$ of elements in $ \maxIdeal $ which can be extended to a regular system of parameters for $ R $.
	We consider various choices of a system $ (y) = (y_1, \ldots, y_r) $ such that 
	$ (u, y) $ is a regular system of parameters for $ R $.  
	
	\smallskip
	
	Let $ ( f ) = ( f_1, \ldots, f_m ) $ be a set of generators of $ J $.
	Every element $ g\in J $ has a finite expansion in $ R $ of the following form (see \cite[Proposition~2.1]{CPdim3}):
	\begin{equation}
	\label{expansion}
	g = \sum_{ (A,B) \in \IZ^{e+r}_{ \geq 0 } } \coeff_{A,B} \, u^A \, y^B
	\end{equation}
	with coefficients $ \coeff_{ A, B } \in R^\times \cup \{ 0 \} $.
	Denote by $ C_{ A, B, i } $ the coefficients of the expansion of $ f_i $, $ 1 \leq i \leq m $.

	\begin{Def}
		\label{Def:NewtonPolyandPolyNonIntrinsic}
		For the given data we introduce the following objects.
		\begin{enumerate}
			\item[(1)]	The {\em Newton polyhedron} $ \Delta^\NW (\IE, u, y) $ (or $ \Delta^\NW_x (\IE, u, y) $) of $ \IE = (J,b) $ at $ x $ \wrt $ ( u, y ) $ is defined to be the smallest closed convex subset of $ \IR^{e+r}_{ \geq 0 } $ containing all elements of the set
			$$
			\left\{ \frac{ (A, B)}{ b } + \IR^{e+r}_{ \geq 0 }  \;\bigg|\; 1 \leq i \leq m \,\wedge\, C_{ A, B, i }  \neq 0 \,\wedge\, | B | \leq b \right\}.
			$$ 
			Let $ \IE' $ be another pair on $ Z $ which is singular at $ x $.
			Then the polyhedron $ \Delta^\NW (\IE \cap \IE', u, y) \subset \IR^{e+r}_{ \geq 0 } $ is  the smallest closed convex subset of $ \IR^{e+r}_{ \geq 0 } $ containing $ \Delta^\NW (\IE, u, y) $ and $ \Delta^\NW (\IE', u, y) $.
			
			\smallskip 
			
			\item[(2)]	We define the {\em projected polyhedron}
			$ \Delta (\IE; u; y) $ (or $ \Delta_x (\IE; u; y) $) of $ \IE = (J,b) $ at $ x $ \wrt $ ( u, y ) $ as  the smallest closed convex subset of $ \IR^e_{ \geq 0 } $ containing all elements of the set
			$$
			\left\{ \frac{ A }{ b - | B | } + \IR^e_{ \geq 0 }  \;\bigg|\; 1 \leq i \leq m \,\wedge\, C_{ A, B, i }  \neq 0 \,\wedge\, | B | < b \right\}.
			$$
			Note that $ \Delta (\IE \cap \IE', u, y) \subset \IR^e_{ \geq 0 } $ is the smallest closed convex subset containing $ \Delta (\IE; u; y) $ and $ \Delta (\IE'; u; y) $.
		\end{enumerate}
		If there is no confusion possible, we just say $ \Delta^\NW (\IE, u, y) $ is the Newton polyhedron of $ \IE $ and $ \Delta (\IE; u; y) $ is the polyhedron of $ \IE $.
	\end{Def}
	%
	

	\begin{Lem}
		\label{Lem:DeltaNindepGen}
		The Newton polyhedron does not depend on the choice of the generating set $ ( f ) = (f_1, \ldots, f_m ) $ of $ J $.
		More precisely,
		$ \Delta^\NW (\IE, u, y) $ coincides with  $ \Delta ( \widetilde{S}) $, which we define to be the smallest closed convex set containing all elements of the set
		$$ 
		\widetilde{S} :=
		\left\{ \frac{ (\widetilde{A},\widetilde{B}) }{ b } + \IR^{e+r}_{ \geq 0 } \;\bigg|\, 
		\begin{array}{c}
		\exists\, g = \sum_{(A,B)} \coeff_{A,B} \, u^A \, y^B \in J \,:\\
		C_{ \widetilde{A},\widetilde{B}} \neq 0  \,\wedge\, | \widetilde{ B } | \leq b
		\end{array}
		\, \right\}.
		$$
	\end{Lem}
	
	\begin{proof}
		Since $ f_1, \ldots, f_m \in J $, we get the inclusion
		$ 
		\Delta^\NW (\IE, u, y) \subseteq \Delta (  \widetilde{ S }).
		$
		On the other hand, let $ g \in J = \langle f_1, \ldots, f_m \rangle $.
		Then $ g = \sum_{i = 1}^m \lambda_i f_i $ for some $ \lambda_i \in R $. 
		Therefore we get that, for every $ (A,B) \in \IZ^{e+r}_{ \geq 0 } $ appearing with non-zero coefficient in some $ g $ and with $ | B | \leq b $, we have $ \frac{ ({A},{B}) }{ b } \in  \Delta^\NW (\IE, u, y) $.
		This completes the proof.
	\end{proof}
	
	\begin{Prop}
		\label{Prop:PolyPorojection}
		The polyhedron $ \poly \IE uy $ associated to a pair $ \IE = (J, b )  $ on $ R $ is a certain projection of the corresponding Newton polyhedron $ \Delta^\NW ( \IE, u, y) $.
	\end{Prop}
	
	\begin{proof}
		This follows immediately by investigating how the projection of a point $ \frac{( A,B)}{b} $ from $ (0, \ldots, 0, 1) \in \IR^{e+r}_{ \geq 0 } $ onto $ \IR^{ {e+r}-1} \times \{ 0 \}$ is determined.
		Applying this several times we obtain the assertion.	
		For more details see \cite[Proposition 2.1.3 and Lemma 2.4.1]{BerndThesis}.
	\end{proof}
	
	\begin{Cor}
		\label{Cor:DeltaIEuyIndepGenerators}
		The polyhedron $ \poly \IE uy $ of a pair $ \IE = ( J, b) $ is independent of the chosen set of generators $ ( f ) = ( f_1, \ldots, f_m ) $. 
	\end{Cor}
	
	\begin{proof}
		This is an immediate consequence of Lemma \ref{Lem:DeltaNindepGen} and Proposition \ref{Prop:PolyPorojection}.
	\end{proof}

	The Newton polyhedron as well as the projected polyhedron depend heavily on the choice of the system $ (y) $ (for fixed $ ( u ) $).
	Further, they are not necessarily invariant under the equivalence relation $ \sim $ .
	
	\begin{Ex}
		Consider the pair 
		$$ 
		\IE = ( y^2 + u_1^7 u_2^3, 2 ) = ( z^2 + 2 z u_1 + u_1^2 + u_1^7 u_2^3, 2 ) 
		$$
		over any field $ K $, where $ y := z + u_1 $ and the point of interest $ x $ is the origin.
		As we see in the picture, the projected polyhedra differ:  
		\[
		\begin{tikzpicture}[scale=0.7]
		
		\draw[<->, thick] (0,3.3)--(0,0)--(6.7,0);
		
		\foreach \x in {0,...,6}
		\draw (\x,0.1) -- (\x,-0.1) node [below] {\x};
		
		\foreach \x in {1,...,3}
		\draw (0.1,\x) -- (-0.1,\x) node [left] {\x};

		\fill[lightgray] (1,0)--(6.5,0)--(6.5,3)--(1,3);
		\filldraw[cyan] (1,0) circle (2.5pt);
		\filldraw[cyan] (3.5,1.5) circle (2.5pt);
		\draw[very thick,cyan] (1,3.05)--(1,0)--(6.55,0);
			
		\node at (3.5,-1.2) {$ \Delta ( \IE, u, z) $};

		\end{tikzpicture} 
		\hspace{25pt}
		\begin{tikzpicture}[scale=0.7]
		
		\draw[<->, thick] (0,3.3)--(0,0)--(6.7,0);

		\foreach \x in {0,...,6}
		\draw (\x,0.1) -- (\x,-0.1) node [below] {\x};
		
		\foreach \x in {1,...,3}
		\draw (0.1,\x) -- (-0.1,\x) node [left] {\x};
		
		\fill[lightgray] (3.5,1.5)--(6.5,1.5)--(6.5,3)--(3.5,3);
		\filldraw[cyan] (3.5,1.5) circle (2.5pt);
		\draw[very thick,cyan] (3.5,3.05)--(3.5,1.5)--(6.55,1.5);	
		
		\node at (3.5,-1.2) {$ \Delta ( \IE, u, y) $};

		\end{tikzpicture} 
		\]
	\end{Ex}	
	
	\begin{Ex}
		\label{Ex:PolyNotUnique}
		The Newton polyhedron and the polyhedron of $ \IE $ may change under the equivalence $ \sim $\,.
		The origin of this example is \cite[Example 5.14, p.788]{BMexc} and it has been slightly modified and worked out for our setting together with Vincent Cossart.
		
		Let $ K = \IC $, $ d \in \IZ_+ , d \geq 2 $.
		We look at the origin of $ \IA^4_\IC $. 
		Consider the two pairs 
		$$
		\begin{array}{l}
		\IE_1 = (z^d - x^{d-1} y^{d-1},\, d) \;\cap\; (t,\, 1)  	\\[5pt]
		\IE_2 = (z^d - x^{d-1} y^{d-1},\, d) \;\cap\; (t^{d-1}-x^{d-2} y^{d-1},\, d-1 )
		\end{array} 
		$$
		First, $ (t, z) $ yields the directrix in both cases; 
		therefore $ (u) = (x,y) $ and $ (y) = (t,z) $.
		
		The set of vertices of $ \Delta ( \IE_1; t,z; x,y) $ is 
		$
		V_1 = \left\{\, \left(\frac{d-1}{d},\, \frac{d-1}{d}\right) \right\} 
		$
		and the one for $ \Delta ( \IE_2; t,z; x,y) $ is 
		$
		V_2 = \left\{\, \left(\frac{d-1}{d},\, \frac{d-1}{d}\right);\, \left(\frac{d-2}{d-1},\, 1\right) \right\}  .
		$
		Clearly the polyhedra are different which implies that also the Newton polyhedra differ.
		
		From the Diff Theorem, Proposition \ref{Prop:Diff}, we obtain by applying the differential operators $ \frac{\partial}{\partial x}$ and $ \frac{\partial^{d-2}}{\partial t^{d-2}}$ that
		$$
		\IE_1 \sim (z^d - x^{d-1} y^{d-1},\, d) \cap ( x^{d-2} y^{d-1},\, d - 1)  \cap (t,\, 1) \sim \IE_2.
		$$
		Therefore $ \IE_1 $ and $ \IE_2 $ are two equivalent pairs whose associated polyhedra differ!
		The picture for $ d = 2 $ looks as follows:
		
			
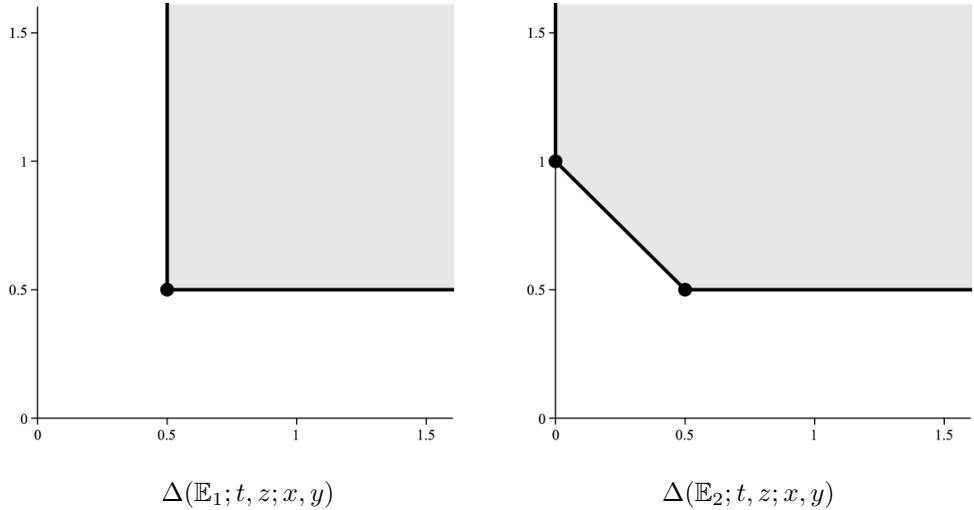
\[
		\begin{tikzpicture}[scale=0.55]
		
		\draw[<->, thick] (0,6.7)--(0,0)--(6.7,0);
		
		\foreach \x in {0,1}
		{
		\draw (4*\x,0.1) -- (4*\x,-0.1) node [below] {\x};
		\draw (0.1, 4*\x) -- (-0.1,4*\x) node [left] {\x};
		}

		\draw (2,0.1) -- (2,-0.1) node [below] {1/2};
		\draw (6,0.1) -- (6,-0.1) node [below] {3/2};
		\draw (0.1,2) -- (-0.1,2) node [left] {1/2};
		\draw (0.1,6) -- (-0.1,6) node [left] {3/2};

		\fill[lightgray] (2,2)--(6.5,2)--(6.5,6.5)--(2,6.5);
		\filldraw[cyan] (2,2) circle (2.5pt);
		\draw[very thick,cyan] (2,6.55)--(2,2)--(6.55,2);	
		
		\node at (3.5,-1.45) {$ \Delta ( \IE_1; t,z; x,y) $};
	
		\end{tikzpicture} 
		\hspace{25pt}
		\begin{tikzpicture}[scale=0.55]
		
		\draw[<->, thick] (0,6.7)--(0,0)--(6.7,0);
		
		\foreach \x in {0,1}
		{
			\draw (4*\x,0.1) -- (4*\x,-0.1) node [below] {\x};
			\draw (0.1, 4*\x) -- (-0.1,4*\x) node [left] {\x};
		}
		
		\draw (2,0.1) -- (2,-0.1) node [below] {1/2};
		\draw (6,0.1) -- (6,-0.1) node [below] {3/2};
		\draw (0.1,2) -- (-0.1,2) node [left] {1/2};
		\draw (0.1,6) -- (-0.1,6) node [left] {3/2};

		\fill[lightgray] (0,4)--(2,2)--(6.5,2)--(6.5,6.5)--(0,6.5);
		\filldraw[cyan] (2,2) circle (2.5pt);
		\filldraw[cyan] (0,4) circle (2.5pt);
		\draw[very thick,cyan] (0,6.55)--(0,4)--(2,2)--(6.55,2);	
		
		\node at (3.5,-1.45) {$ \Delta ( \IE_2; t,z; x,y) $};
		
		\end{tikzpicture} 
		\]
	\end{Ex}
	
	\smallskip
	
	The last example plays also a crucial role if there exist exceptional components of a resolution process.
	It forces us in \cite{BerndBM} (or see also \cite{BerndThesis}) to introduce idealistic exponents with history, which take the preceding resolution process into account.
	
	\medskip
	
	We overcome the dependence on $ ( y) $ by introducing the characteristic polyhedron.
	For this, let us recall the construction of Hironaka's characteristic polyhedron.
	More detailed references are \cite[section 7]{CJS}, \cite[section 2.2]{BerndThesis}, or Hironaka's original work \cite{HiroCharPoly}.
	
	\smallskip
	
	Let $ ( R,  \maxIdeal, \resfield = R / \maxIdeal) $ be a regular local ring, $ J \subset R $ be a non-zero ideal, and $ ( u , y ) = ( u_1, \ldots, u_e; y_1, \ldots, y_r ) $ be a regular system of parameters of $ R $.
	Let $ \hR $ be the $ M $-adic completion of $ R $.
	Note that so far we do not make any other assumptions on $ ( u,y) $, for example, we do not suppose that $ ( y) $ is related to the directrix of $ J $.
	Set $ R' := R / \langle u \rangle $ and $ J' := J \cdot R' $.
	
	\begin{Def}
		\begin{enumerate}
			\item
			Let $ f \in R $ be an element in $ R $, $ f \notin \langle u \rangle $.
			We expand $ f $ in a \emph{finite} sum as in \eqref{expansion},
			$$
			f = \sum_{(A,B) \in \IR^{e + r }_{\geq 0 } } C_{A,B} \, u^A \, y^B
			$$
			with coefficients $ C_{A,B} \in R^\times \cup \{ 0 \} $.
			We define $ n := n_{(u)} ( f ) $ to be the order of $ f \mod \langle u \rangle $ in the ideal generated by $ \overline{ y_j } = y_j \mod \langle u \rangle $, $ j \in \{ 1, \ldots, r \} $.
			The {\em polyhedron $ \poly fuy $ associated to $ ( f ,u, y ) $} is defined as the smallest closed convex subset of $ \IR^e_{ \geq 0 } $ containing all elements of the set
			$$
			\left\{
			\;	\frac{A}{n - |B|} + \IR^e_{\geq 0 } \;\bigg| \; C_{A,B} \neq 0 \, \wedge \, |B| < n
			\;
			\right\} . 
			$$	
			
			\item
			Let $ ( f ) = ( f_1, \ldots, f_m ) $ be a system of elements in $ R $ with $ f_i \notin \langle u \rangle $, for every $ i $.
			The {\em polyhedron $ \poly fuy $ associated to $ ( f, u, y ) $} is defined as the smallest closed convex subset of $ \IR^e_{ \geq 0 } $ containing $ \bigcup_{ i = 1 }^m \poly{f_i}{u}{y} $. 
		\end{enumerate}	
	\end{Def}
	
	The polyhedron $ \poly fuy $ clearly depends on the choice of the system $ (f ) = ( f_1, \ldots, f_m ) $.
	A special class of system of generators for an ideal $ J $ are so called $ ( u ) $-standard bases (see \cite[Definition (2.20)]{HiroCharPoly}).
	Since the polyhedra $ \poly \IE uy $ are independent of the choice of the generators (Corollary \ref{Cor:DeltaIEuyIndepGenerators}) we do not recall this technical definition.
	We only remark that they are generators $ ( f ) = ( f_1, \ldots, f_m ) $ of $ J $ such that $ f_i \notin \langle u \rangle $ and which are ordered by the order of $ f \mod \langle u \rangle $, and moreover $ m $ is as small as possible.
	
	\begin{Def}
		Let $ (0) \neq J \subset R $ be an ideal and $ ( u ) = (u_1, \ldots, u_e ) $ be a system of elements as before.	
		We define
		$$
		\cpoly Ju = \bigcap_{(\hy)} \bigcap_{ (\hf) }  \poly \hf u\hy,
		$$
		where the first intersection ranges over all systems $ ( \hy ) $ extending $ ( u ) $ to a regular system of parameters of $ \hR $ and the second runs over all possible $ ( u ) $-standard bases $ ( \hf ) = ( \hf_1, \ldots, \hf_m ) $ of $ \widehat{J} := J \cdot \hR  $.
		The polyhedron $ \cpoly Ju $ is called the characteristic polyhedron of $ J $ \wrt $ ( u ) $.
	\end{Def}
	
	Note that the above is an ad-hoc definition. 
	In \cite{HiroCharPoly}, the characteristic polyhedron $ \cpoly Ju $ is defined without using the completion. 
	
	Suppose we have any $ ( u ) $-standard basis $ ( f ) $ for $ J $ and a system $ ( y ) $  extending $ ( u ) $ to a regular system of parameters of $ R $ given.
	Further, assume that the images of $  ( y ) $ in $ R' $ define the directrix of $ J $.
	In \cite[Theorem (4.8)]{HiroCharPoly}, Hironaka proves that 
	one can construct from this data a suitable $ ( u ) $-standard basis $ ( \widehat f ) $ of $ J \widehat R $ and a system of elements $ ( \widehat y ) $ extending $ ( u ) $ to a regular system of parameters for $ \widehat R $ such that 
	\begin{equation}
	\label{eq:Hiro_char_computed}
	\poly{\hf}{u}{\hy} = \cpoly Ju.
	\end{equation}

	In \cite{CSCcompl} Cossart and the author extend a result of \cite{CPcompl} investigating under which conditions it is possible to attain the characteristic polyhedron without passing to the completion.
	
	\smallskip
	
	We imitate the definition of $ \cpoly Ju $ in order to define $ \cpoly \IE u $:
	
	\begin{Def}
		Let $ (J, b) $ be a pair on $ R $ and let $ ( u ) = (u_1, \ldots, u_e ) $ be a system of regular elements that can be extended to a regular system of parameters of $ R $.	
		We define
		$$
		\cpoly \IE u := \bigcap_{(\hy)} \poly \IE u\hy,
		$$
		where the intersection ranges over all systems $ ( \hy ) $ extending $ ( u ) $ to a regular system of parameters of $ \hR $.
		We call $ \cpoly \IE u $ the \emph{characteristic polyhedron of the pair $ \IE $ \wrt $ ( u ) $}.
	\end{Def}
	
	In order to prove the analog to \eqref{eq:Hiro_char_computed}, 
	we introduce 
	
	\begin{Def}
		Let $ \IE = ( J, b ) $, $ R $, and $ ( u ) $ as before.
		Let $ ( y ) $ be a system extending $ ( u ) $ to a regular system of parameters for $ R $
		and let $ ( f_1 , \ldots, f_m) $ be a set of generators for $ J $.
		Let $ v \in \poly \IE uy \subset \IR^e_{\geq 0 } $ be a vertex.
		
		\begin{enumerate}
			\item
			The {\em $ (v,b) $-initial form} of an element $ f = \sum C_{A,B} u^A y^B \in J $ (expansion as in \eqref{expansion}) 
			is defined as 
			$$	 
			\ini_v ( f,b)_{u,y} = \sum \overline{C_{A,B}} \, U^A Y^B \in K[U,Y]
			$$
			where the sum ranges over those $ ( A,B) $ such that, $ |A| = 0 $ and $ |B| = b $, or $ \frac{A}{b - |B|} = v $,
			and $ \overline{C_{A,B}} := C_{A,B} \mod M \in K$.
			
			\smallskip 
			
			\item 
			The vertex $ v \in \poly \IE uy $ is called \emph{solvable} if 
			there exist constants $ \lambda := (\lambda_1, \ldots, \lambda_r) \in K^r $ and polynomials $ P_i \in K[X_1, \ldots, X_r] $, for $ 1 \leq i \leq m $, such that
			$$
			\ini_v(f_i, b)_{u,y}  = P_i (Y - \lambda U^{ v}),
			$$
			where we abbreviate $ Y - \lambda U^{v} := ( Y_j - \lambda_j U_1^{v_1} \cdots U_e^{v_e})_{1\leq j \leq r} $. 
			
		\end{enumerate}
	\end{Def}
	
	If $ v $ is solvable then $ v \in \IZ^e_{\geq 0} $.
	Further, we can choose lifts $ \rho_j \in R^\times \cup \{ 0 \} $ such that $ \rho_j \equiv \lambda_j \mod M $ and define 
	$$
	z_j := y_j - \rho_j u^{v},
	\ \ 1 \leq j \leq r.
	$$
	
	\red{In \cite{HiroCharPoly}, Hironaka makes two preparations on the initial forms at a vertex: 
	normalization and dissolution.
	While we discuss the latter here, the normalization is not necessary since it is a procedure to optimize the choice of the generators and our variant of the characteristic polyhedron for idealistic exponents does not depend on the choice of the generators by Corollary~\ref{Cor:DeltaIEuyIndepGenerators}.}
	
	\begin{Lem}
		\label{Lem:cleaning}
		Suppose $ v \in \poly \IE uy $ is a solvable vertex.
		Using the previous notation, we have:
		\begin{enumerate}
			\item $ v \notin \poly \IE uz $.
			
			\smallskip 
			
			\item If $ w \in \poly \IE uy $ is a vertex with $ w \neq v $, then $ w $ is also a vertex of $\poly \IE uz $.
			
			\smallskip 
			
			\item $ \poly \IE uz \subset \poly \IE uy $.
		\end{enumerate} 
	\end{Lem}
	
	\begin{proof}
		Let $ u^A y^B $ be a monomial appearing with non-zero coefficient in an expansion of some $ f_i $.
		After the change from $ ( y) $ to $ ( z ) $, we get
		$$
		u^A (z + \rho u^{v} )^B
		=
		\sum_{C=0}^{B} \binom{B}{C} \rho^C u^A ( u^{v})^{|C|} z^{B-C} 
		$$
		where $ C = (C_1, \ldots, C_r ) $, and we use the abbreviations
		$ \sum\limits_{C=0}^{B}  := \sum\limits_{C_1=0}^{B_1} \cdots \sum\limits_{C_r=0}^{B_r} $, and
		$ \binom{B}{C} := \binom{B_1}{C_1} \cdots \binom{B_r}{C_r} $, 
		and $ \rho^C := \rho_1^{C_1} \cdots \rho_r^{C_r} $. 
		Therefore, the points coming from $ ( A,B) $ that possibly appear in $ \poly \IE uz $ are
		\begin{equation}
		\label{eq:possible_vertex}
		w' :=
		\dfrac{A +  v \cdot |C| }{b - (|B|-|C|)}
		= 
		\dfrac{b-|B|}{b - |B| + |C| } \cdot \dfrac{A}{b-|B|}
		+ \dfrac{|C|}{b - |B|+|C|} \cdot  v.
		\end{equation}
		By the definition of the projected polyhedron, this can only provide a point in $ \poly \IE u z $ if $ |B|- |C| <  b $. 
		Hence we assume from now on $ |B|- |C| <  b $.

		If $ |B| > b $, then $ |C| > b - |B| + |C| $ and the point $ w' $ (potentially lying in $ \poly \IE u z $) fulfills $ w' \in v + \IR^e_{\geq 0} $, but $ w' \neq v $, by the second summand on the right in \eqref{eq:possible_vertex}.
		The same is true if $ |B| = b $ and $ |A| \neq 0 $.
		
		If $ |B| < b $ then the monomial $ u^A y^B $ provides the point $  w := \frac{A}{b - |B|} $ in $ \poly \IE uy $.
		If $ |C| \neq 0 $, then \eqref{eq:possible_vertex} shows that the associated point $ w' $  must be on the line connecting $ w $ and $ v $, but $ w' \neq w $ and $ w' \neq v $.
		For $ C = (0, \ldots, 0 ) $, we get the monomial $ u^A z^B $.
		
		In conclusion, under the above conditions the points that possibly newly appear in $ \poly \IE u z $ are contained in $ \poly \IE uy $, but are not equal to $ v $.
		
		It remains to consider the monomials with $|B| = b $ and $ |A|=0 $, or with $ \frac{A}{b-|B|} = v $.
		These are precisely those appearing in $ \ini_v(f_i,b) $.
		By definition of $ (z) $, we eliminate them all, when we pass from $ ( y ) $ to $ ( z ) $ and terms newly created from them are contained in $ v + \IR^e_{\geq 0} $, but $ \neq v $.
		
		Combining the latter with the conclusion provides that we have after the change to $ ( z ) $:
		There is no monomial contributing to $ v $, i.e., $ v\notin \poly \IE uz $.
		Further, if $ w \in \poly \IE uy $, $ w \neq v $, is a vertex, then so it is in $ \poly \IE uz $.
		Finally, every newly created vertex of $ \poly \IE uz $ is contained in $ \poly \IE uy $, i.e., $ \poly \IE uz \subset \poly \IE uy $.	
		This ends the proof.
	\end{proof}		
	
	\smallskip
	
	The following example show that, in general, it may happen that a vertex is not solvable, but it can be eliminated via a change in $ ( y ) $.
	
	\begin{Ex}
		\label{Ex:AssumpThmNecessary}
		Consider the pair $ \IE = ( \langle f_1, f_2 \rangle, 2 ) $ over a field $ \resfield $,  $ \car \resfield = p \geq 3 $, given by 
		$$ 
		f_1 ( u, y ) = y_1^2 + h_1 (u_1) 
		\hspace{10pt} \mbox{ and } \hspace{10pt}
		f_2 ( u, y ) = u_3 y_2 + ( y_2 + u_2^n )^p + h_2 (u_1).
		$$
		for some $ h_1, h_2 \in \resfield [u_1] $ and an integer $ n \in \IZ_+ $, $ n \geq 2 $.
		The system $ ( y_1, y_2, u_3) $ defines the directrix $ \Dir_x ( \IE ) $. 
		We assume that $ h_1 ( u_1 ) $ is such that $ \poly {f_1}uy = \cpoly{ f_1 }u $ has no solvable vertex.
		Then there is also no solvable vertex in $  \poly{ (f, 2) }uy $.
		
		Consider the system $ ( z ) = ( z_1, z_2 ) := ( y_1, y_2 + u_2^n ) $. 
		We get that 
		$$ 
		f_1 ( u, z ) = z_1^2 + h_1 (u_1) 
		\hspace{10pt} \mbox{ and } \hspace{10pt}
		f_2 ( u, z ) = u_3 z_2 - u_2^n u_3 + z_2^p + h_2 (u_1)
		$$
		and still there are no solvable vertices for $ \poly{ f_1 }uz  =  \poly{ f_1 }uy = \cpoly{ f_1}u $ and $  \poly{ (f, 2) }uz $.
		On the other hand, set $ v := \left( 0, \frac{ n p }{ 2 }, 0 \right) $ and $ w :=  \left( 0, \frac{ n }{ 2 }, \frac{ 1 }{ 2 } \right) $.
		We have that 
		$$ 
		(0,0,1), v \in \poly{(f, 2)}{ u}{ y},
		\ \ \ w \notin  \poly{ (f, 2) }uy,
		$$  
		$$ 
		(0,0,1), w \in  \poly{ ( f, 2)}uz,
		\ \ \  v \notin \poly{( f, 2) }uz.
		$$
		
		Hence, the polyhedra $ \poly{ (f, 2) }uy  $ and $ \poly{( f, 2 ) }uz $ are essentially different.
		In general, it is not possible to make $ \Delta ( (f, b) ; u, y ) $ independent of the choice of the system $ ( y ) $ (with our definitions).

		Nonetheless, we may obtain intrinsic information from the polyhedra.
		Namely, in both cases the point $ (0, 0, 1 ) $ appears in the polyhedra and we have (cf.~Definition \ref{def:delta(Delta)} and Lemma \ref{Lem:DeltaEins})
		$$
		\begin{array}{rl}
		\min\{ |\widetilde v| = \widetilde v_1 + \widetilde v_1 \mid \widetilde v = (\widetilde v_1, \widetilde v_2)\in \poly{( f, 2) }uy \} 
		& = \\[5pt]
		= \min\{ |\widetilde w| \mid \widetilde w \in \poly{( f, 2) }uz \}
		& = 1.
		\end{array}
		$$
	\end{Ex}
	
	\smallskip
	
	In order to be able to speak of a minimal polyhedron, we have to put additional conditions on the system $ ( y )$ that we are starting with.
	For Hironaka's characteristic polyhedron this problem is overcome by imposing on the system $ (y) $ that its images in $ R' $ define the directrix of $ J' $ in $ R' $.
	This is a crucial assumption that holds for all interesting cases when studying singularities.

	\begin{Not}
		\label{nota:E'}
		Given $ \IE= (J,b) $ on $ R $, we define 
		$ \IE' := (J', b) $ to be the pair on $ R' = R/ \langle u \rangle $ induced by $ J' = J R' $.
		Denote by $ x' $ the point in $ \Spec(R') $ defined by the ideal generated by images of $ (y) $ in $ R' $.
		In particular, we can speak of the directrix $ \Dir_{x'} (\IE' )$.
	\end{Not} 
	
	\begin{Lem}
		\label{lem:cleaning2}
		Let $ \IE $ be a pair on a regular local ring $ R $ and
		let $ ( u ,y ) $ be a regular system of parameters for $ R$ such that the initial forms of $ ( y ) $ define the directrix of $ \Dir_{x'}( \IE') $.
		Let $ v \in \poly \IE uy $ be a vertex.
		Then we can eliminate $ v $ via a change $ z_j := y_j - r_j u^{v} $ (i.e., $v \notin \poly \IE uz $), for $ r_j \in R^\times \cup \{ 0 \} $, $ 1 \leq j \leq r $, if and only if $ v $ is solvable. 
	\end{Lem}

	\begin{proof}
		Suppose $ v $ is not solvable, but the change $ z_j := y_j - r_j u^{v} $ eliminates $ v $.
		Let $ ( f_1 \ldots, f_m ) $ be a system of generators for $ J $. 
		Since $ ( y ) $ determines $ \Dir_{x'} (\IE') $, there exists for every $ y_j $ some $ f_{i(j)} $ such that $ B_j \neq 0  $ in a monomial $ y^B $ which appears with non-zero coefficient in an expansion of $ f_{i(j)} $ as in \eqref{expansion}.
		Since $ v $ is not solvable, there exist no polynomials $ P_i(Z) $ such that 
		$ \ini_v(f_i,b)_{u,z} = P_i(Z) $, for all 
		$ i \in \{ 1, \ldots, m \} $.
		In other words, there is at least one $ i $ with	$ \ini_v(f_i,b)_{u,z} \notin K[Z] $.
		But then, by definition, $ v \in \poly \IE uz $ which is a contradiction.	
		
		The if-part follows by definition and this completes the proof.
	\end{proof}

	\begin{Thm}[Theorem \ref{MainThm:CharPolyAchieved}]
		\label{Thm:BerndHiro}
		Let $ \IE = (J,b) $ be a pair on a regular local ring $ R $ and let $ ( u, y ) = ( u_1, \ldots, u_e; y_1, \ldots, y_r ) $  be a regular system of parameters for $ R $ such that the initial forms of $ ( y ) $ define the directrix of $ \Dir_{x'}( \IE') $.
		
		There exist elements $ ( y^* ) = ( y^*_1, \ldots, y^*_r ) $ in $ \hR $ such that $ ( u, y^* ) $ is a regular system of parameters for $ \hR $, $ ( y^* ) $ yields $ \Dir_{x'}( \IE') $, and 
		\[ 
		\poly{ \IE }{ u }{ y^* } = \cpoly{ \IE }{ u }
		\] 
	\end{Thm}

	\begin{proof}
		First, let us explain our variant of Hironaka's procedure to make the polyhedron minimal starting with a given choice for the system $ ( u, y ) $.
		Set
		$$
		\Delta( y ) := \poly{ \IE}uy \subset \IR^e_{ \geq 0}.
		$$
		This polyhedron has finitely many vertices.
		We equip $ \IR^e_{\geq 0 } $ with the total ordering defined by the lexicographical order of $ ( |v| ,  v_1, \ldots, v_e ) $, for $ v = ( v_1, \ldots, v_e ) \in \IR^e_{\geq 0} $.
		We pick the minimal vertex \wrt this total ordering, say $ v^{(1)} $.
		
		If $ v^{(1)} $ is not solvable, we keep it and consider the next smallest vertex \wrt the total ordering.
		
		If $ v^{(1)} $ is solvable, we eliminate it via a change $ y_j \mapsto z_j := y_j - r_j u^A $, as before, $ 1 \leq j \leq r $.
		Set 
		$$ 	
		\Delta ( z )  := \poly{ \IE}uz .
		$$
		By Lemma \ref{Lem:cleaning}, $ \Delta(z) \subset \Delta(y) $ and $ v \notin \Delta(z) $.
		We then repeat the previous for the smallest vertex of $ \Delta(z) $ which is strictly bigger then $ v^{(1)} $.
		
		This procedure is not necessarily finite
		(e.g.~apply it for the ideal $ J = \langle y^{p^2} + y^p - u^{2p} \rangle $ and $ b = p = \car{ \resfield } $).
		But nevertheless we attain an element $ \widehat y = y + \sum_A \lambda_A \, u^A \in \hat R $ for which $  \poly{ \IE}u{\widehat y} \subseteq \Delta(y) $ and no change of the type above can make $  \poly{ \IE}u{\widehat y }$ strictly smaller.
		
		\smallskip
		
		If we start with another choice for $ ( y ) $, say $ ( z ) $, then we can apply the above procedure and obtain $ ( \hz ) $ with $ \poly{\IE}{ u }{ \hz } \subset \poly{\IE}{ u }{ z } $.
		It remains to show
		\begin{equation}
		\label{eq:polystarsame}	
		\poly \IE u{\hy} = \poly \IE u{\hz} ,
		\end{equation}
		which then implies the assertion of the theorem.
		%
		%
		By abuse of notation, we write in the following $ ( y ) $ (resp. $ ( z ) $) instead of $ ( \hy ) $ (resp. $ ( \hz ) $).
		Consider $ h \in J $ and let $ h = \sum_{(A,B)} \coeff_{A,B} \, u^A y^B $ be an expansion as in \eqref{expansion}.
		Then the smallest closed convex subset of $ \IR^e_{\geq 0} $ containing the points of the set
		$$ 
		S (h,b;u,y) := \left\{ \frac{A}{ b - | B | } + \IR^e_{\geq 0 } \;\bigg|\; \coeff_{A,B} \neq 0 \wedge |B| < b \right\}.
		$$
		 coincides with the polyhedron $ \poly {(h,b)}{u}{y} $.
		Moreover, $ \poly \IE uy $ is the smallest closed convex subset of $ \IR^e_{\geq 0} $ containing $ \bigcup_{ h \in J } \poly {(h,b)}{u}{y} $.
		
		For every $ y_j $, $ 1 \leq j \leq r $, we have an expansion which is of the form
		\begin{equation}
		\label{eq:expan_y_by_z_u}
		y_j = L_j ( z ) + H_j ( u, z)  + Q_j ( u ), \hspace{15pt} \mbox{ where }
		\end{equation}
		\begin{itemize}
			\item[$ \diamond $]	
			$ L_j ( z ) = \sum_{i=1}^r \epsilon_{ij} z_i $  with $ \epsilon_{ij} \in R^\times \cup \{ 0 \} $, and 
			$$ 
			\langle \,  L_1 (z) , \ldots, L_r (z) \, \rangle = \langle z_1, \ldots, z_r \rangle \subset R,
			$$ 
			\item[$ \diamond $]	
			$ H_j ( u, z ) \in  \langle u, z \rangle^2 \cap \langle z \rangle $,
			and 
			\item[$ \diamond $]					
			$ Q_j (u) = \sum_A D_{A,j} u^A $, for $ D_{A,j} \in R^\times \cup \{ 0 \} $ and $ |A|\geq 1 $ if $ D_{A,j} \neq 0 $.
		\end{itemize} 
		We split the substitution from $ ( y ) $ to $ ( z ) $ into the following three steps: 
		$$
		y_j 
		\;\;\stackrel{ \displaystyle{( 1 )} }\longmapsto\;\;
		L_j ( z ) 
		\;\;\stackrel{ \displaystyle{( 2 )} }\longmapsto\;\;
		L_j ( z )  +  H_j ( u, z)
		\;\;\stackrel{ \displaystyle{( 3 )} }\longmapsto\;\;
		L_j ( z )  +  H_j ( u, z) + Q_j ( u ), 
		$$
		for $ 1 \leq j \leq r  $.
		We show that the polyhedra after each step coincide with $ \poly \IE uy \subset \IR^e_{\geq 0} $. 
		
		In {\bf step (1)}, a monomial $ u^A y^B $ is mapped to $ u^A \prod_{j=1}^r L_j ( z )^B $.
		By the special form of $ L_j ( z ) $, we obtain the same point $ \frac{A}{ b - |B|} $ after the substitution.
		Although some monomials contributing to a point $ v $ of the polyhedron might vanish under the change from $ ( y ) $ to $ ( z ) $ it can never happen that all of them disappear, i.e., $ v $ is still appearing in the polyhedron \wrt $ ( z ) $.
		
		Next we come to {\bf step (2)}.
		By the first step, we may assume that it is given by 
		$$ 
		y_j = z_j + H_j ( u, z) .
		$$
		Consider an expansion of $ H_j ( u , z ) $ as in \eqref{expansion},
		$ H_j ( u , z ) = \sum_{(C,D)} \lambda_{j,C,D} \,u^C z^D $,
		for certain $ \lambda_{j,C,D} \in R^\times \cup \{ 0 \} $.
		If $ \lambda_{j,C,D} \neq 0 $, we have $ |D| \geq 1 $, and if $ |D| = 1 $ then $ |C| \neq 0 $. 
		(Otherwise the monomial could be shifted into $ L_j ( z ) $). 
		
		Pick $ (C,D) \in \IZ^e_{ \geq 0 } \times \IZ^r_{ \geq 0 } $ with $ \lambda_{j,C,D} \neq 0 $, for some $ j \in \{ 1 ,\ldots, r \} $.
		Let us consider the substitution 
		$$
		y_j \stackrel{ \displaystyle{( 2_{C,D} )} }  = z_j + \lambda_{j,C,D} \,u^C z^D ,
		\hspace{15pt} 1 \leq j \leq r .
		$$
		An easy computation shows
		$$
		\begin{array}{ll}
		y^B 
		&
		\displaystyle
		= \sum_{M_1 = 0}^{B_1} \cdots \sum_{ M_r = 0}^{ B_r } \lambda_{M,C,D}\, ( u^{C} z^D )^{ M_1 + \ldots + M_r }  \, z^{B - (M_1, \ldots , M_r)} =
		\\[8pt]
		&
		\displaystyle
		= \sum_{ M = 0}^{B} \lambda_{M,C,D} \,( u^{C} )^{ |M| } \, z^{B - M + D \cdot |M| }, 
		\end{array}
		\eqno{(\ast)}
		$$
		where we set $ \lambda_{M,C,D} := \prod_{ j = 1 }^r \binom{B_j}{M_j}\lambda_{j,C,D}^{ M_j } $
		Using this for $ C_{A,B} u^A y^B $, we see that all points which might appear are of the form
		\[ 
		v' := \frac{A + C \cdot |M|}{ b - ( \, |B - M| + |D| \cdot |M| \, ) } 
		\]
		For $ M = ( 0, \ldots, 0 ) $ we obtain the same point $ \frac{A}{b - |B|} $ with the same coefficient $ C_{A,B} $. 
		Since $ |D| \geq 1 $, and $ C \neq 0 $ if $ |D| = 1 $, we get, for $ |M| \neq 0 $,
		$$ 
		v' \in  \left( \, \frac{A}{b - |B|} + \IR^e_{\geq 0 } \,\right) \setminus \left\{ \frac{A}{b - |B|} \right\}.
		$$	
		Therefore vertices are not touched and the polyhedron does not change under $ ( 2_{C,D} ) $.
		We apply this for each $ (C,D) $ with non-zero coefficients and get that the polyhedron does not change in step (2).
		
		Finally, we have to deal with {\bf step (3)}:
		By the first two steps we may assume that the substitution is given by
		$$ 
		y_j = z_j + Q_j ( u ), \hspace{15pt} 1 \leq j \leq r,
		$$
		for some $ Q_j (u) = \sum_{ A \in \IZ^e_{ \geq 0 } } \,D_{A,j} u^A $, for $ D_{A,j} \in R^\times \cup \{ 0 \} $ and $ |A|\geq 1 $.
		By the construction of $ ( y ) = ( \widehat{y}) $ and $ ( z ) = ( \widehat{z} ) $ at the beginning of the proof, the corresponding polyhedra must coincide, $ \poly \IE u{\hy} = \poly \IE u{\hz} $.
		\red{Suppose this is not the case. 
			Then there exists $ v \in \IR_{\geq 0}^e $ such that 
			either $ v \in \poly \IE u{\hy} $ and $ v \notin \poly \IE u{\hz} $
			or $ v \notin \poly \IE u{\hy} $ and $ v \in \poly \IE u{\hz} $.
			By the construction of $ ( \hy  ) $ (resp.~$ (\hz) $) no change of the form 
			$ y_j \mapsto  y_j + D_{A,j} u^A $ (resp.~$ z_j \mapsto  z_j - D_{A,j} u^A $) can make $ \poly \IE u {\hy} $ (resp.~$ \poly \IE u {\hz} $) strictly smaller.
			Therefore, both cases are impossible and we must have equality.}
		%
		%
		This completes the proof of Theorem \ref{Thm:BerndHiro}.
	\end{proof}

	\smallskip
	
	Note that for the definition of $ \cpoly Ju $ the points of the form $ \frac{A}{n_i - |B|} $, $ n_ i = n_{(u)} ( f_i ) $, are considered.
	On the other hand, $ \cpoly \IE u $ is defined by those of the form $ \frac{A}{b  - |B|} $ and in general $ b \leq n_i $.
	Therefore the two polyhedra do not necessarily coincide:

	\begin{Ex}
		Let $ K $ be a field of characteristic $ p = 3 $ and let $ b = 2 $.
		Let 
		$$ 
		f_1 = z_1^2 + u_1^3 
		\hspace{10pt}
		\mbox{ and }
		\hspace{10pt}
		f_2 = z_2^3 + z_2^2 u_2^2  + u_2^9 
		$$ 
		and $ J = \langle f_1, f_2 \rangle \subset K [ u, z ]_{\langle u, z \rangle} $.
		The set of vertices of the polyhedron $ \poly fuz $ is 
		$
		\left\{  v := \left( \frac{ 3}{2}, 0 \right),  w := (0, 2) \right\}.
		$
		One can show that we have $ \poly fuz = \cpoly Ju $.
		
		On the other hand, let $ \IE := ( \langle f_1, f_2 \rangle , b = 2 ) $.
		The set of vertices for $ \poly \IE uz  $ is
		$
		\left\{  v = \left( \frac{ 3}{2}, 0 \right),  \widetilde{v} :=  \left(0, \frac{9}{2} \right) \right\}.
		$
		The vertex $ \widetilde{v} $ can be eliminated by changing the coordinates to $ (x_1, x_2) := (z_1, z_2 + u_2^3 ) $.  
		Thus $  \poly \IE uz  \neq \cpoly \IE u $. 
		Since the order of $ f_2 $ at the origin is three and thus bigger than $ b = 2 $, the directrix of  $ \IE = ( \langle f_1, f_2 \rangle, 2 ) $ is only given by $ Z_1 $! 	 
		If we set $ y_1 : = z_1 $ and $ u_3 := z_2 $, then $ f_2 \in \langle u_1, u_2, u_3 \rangle $, which means that Hironaka's characteristic polyhedron $ \cpoly J{u_1,u_2,u_3} $ is not defined.
		
		\emph{Therefore there is an essential difference between the polyhedron of the ideal $ J $ and the polyhedron of the pair $ \IE = ( J, b) $.}
	\end{Ex}
	
	\smallskip
	
	\noindent {\bf Characteristic polyhedra of idealistic exponents.} ---
	As we have seen in Example \ref{Ex:PolyNotUnique}, the characteristic polyhedron $ \cpoly{ \IE }{ u } $ may not behave well under the equivalence relation $ \sim $.
	Indeed, the polyhedra in the mentioned example are already minimal \wrt the choice of $ ( y ) $ (cf.~Proposition \ref{Prop:maxSamePoly}) and although the pairs in this example are equivalent, the polyhedra do not coincide.
	
	In Theorem \ref{Thm:deltaINV} below we proof that 
	$$ 
	\delta_x (  \IE ; u  ) := \inf\{ |v| \mid v \in \cpoly \IE u \}  
	\ \ \ 
	\mbox{(cf.~Definition \ref{def:delta(Delta)})}	
	$$
	coincides for equivalent pairs.
	Therefore this is an invariant of the idealistic exponent $ \IE_\sim $.
	In \cite{BerndBM} this result is used to deduce that the invariant of Bierstone and Milman for their constructive resolution of singularities in characteristic zero can be obtained solely by considering polyhedra.
	This means from this point of view it is not necessary to have a unique polyhedron of an idealistic exponent.
	
	One way to obtain a unique characteristic polyhedron for an idealistic exponent would be to characterize a canonical representative.
	Since the changes of the polyhedra occur when we apply differential operator a candidate for the canonical representative could be by applying all differential operators $ \Diff^{< b}_\IZ ( R ) $ on $ \IE = (J,b) $ and then reducing via $ (I^a, ac) \sim ( I, c) $, $ a \in \IZ_+ $, as much as possible.
	In fact, we show in Lemma \ref{Lem:PolyEquiv} that the reduction in the last step does not change our polyhedron.
	These considerations already appear in Hironaka's work \cite{HiroThreeKey}, where he uses the notion of \emph{Diff-full} pairs (\cite[Definition 11.2]{HiroThreeKey}) and shows how to obtain such a situation (\cite[Lemma 11.2]{HiroThreeKey}).
	The previous idea also appears in \cite[Theorem 3.11]{AnaMariOrlando}, where a canonical representative for a Rees algebra given over a perfect field is determined.

	\smallskip
	
	\begin{Rk}[Quasi-homogeneous characteristic polyhedra]
		\label{Rk:qh_setting}
		Let $ R $ be a regular local ring and $ ( u, y ) = ( u_1, \ldots, u_e; y_1, \ldots, y_r ) $ be a regular system of parameters for $ R $.
		So far we never used weights on the elements of $ ( u, y ) $, respectively, to be more precise, we assigned to each of them the weight $ 1 $.
		For an element $ g = \sum C_{A,B} u^A y^B \in R $ and a positive rational number $ b \in \IQ_+ $ the polyhedron $ \poly{ (g,b) }uy $ was then defined via the points $ \frac{ A }{ b - |B| } $ with $ C_{A,B} \neq 0 $ and $ |B| < b $ (Definition \ref{Def:NewtonPolyandPolyNonIntrinsic}).
		But, in principle, we are not forced to consider only this situation -- and in order to obtain refined information on the singularity it might also be useful to change the view in certain directions.
		The following generalization can also be done for Hironaka's characteristic polyhedron, see \cite[section 2]{PCQO}, where it is used to prove a characterization theorem for quasi-ordinary hypersurface singularities.
		
		Let $ \nu : R \to \IQ \cup \{ \infty \} $ be a monomial valuation on $ R $ defined by 
		$$
		\nu ( u_i ) = \alpha_i , \hspace{10pt}
		\nu ( y_j ) = \beta_j, \hspace{10pt}
		\nu ( \lambda ) = 0, \hspace{10pt}
		\mbox{ and }\hspace{10pt}
		\nu ( 0 ) = \infty,
		$$
		where $ \lambda \in R^\times $ is a unit in $ R$ and $ \alpha_i, \beta_j \in \IQ_{\geq 0 } $ are non-negative rational numbers, $ 1 \leq i \leq e $ and $ 1 \leq j \leq r $.
		(In fact, one might even consider monomial valuations of rank strictly bigger than $ 1 $, e.g., with values in $ \IQ^d_{\geq 0 } $, for some $ d \geq 2 $, see \cite{PCQO}).
		The example which we are having in mind and which will appear in \cite{BerndBM} is, when $ ( y) $ determines the directrix of a pair $ \IE $, $ \poly \IE uy = \cpoly \IE u $ is minimal, $ \alpha_i = \frac{1}{\delta} < 1  $, for all $ i $, where $ \delta := \delta (  \IE ; u ) $  (Definition \ref{def:delta(Delta)}), and $ \beta_j = 1 $, for all $ j $. 
		
		For $ g \in R $ and $ b \in \IQ_+ $ as above set $ \IE : = ( g,b) $.
		Then we define the \textit{associated $ \nu $-polyhedron} $ \Delta^\nu ( \IE ; u; y ) $ as the smallest closed convex subset of $ \IR^e_{\geq 0 } $ containing all elements of
		$$
		\left\{ \;
		\frac{\atA}{ b - |\btB| } + \IR^e_{\geq 0 } \mid C_{A,B} \neq 0 \wedge | \btB | < b 
		\; \right\},
		$$ 
		where we use the abbreviations $ \atA := ( \alpha_1 A_1, \ldots, \alpha_e A_e ) $ and $ \btB := ( \beta_1 B_1, \ldots, \beta_r B_r ) $.
		One possibility to define the \emph{characteristic $ \nu $-polyhedron} is by 
		$$ 
		\Delta^\nu ( \IE ; u ) := \bigcap_{ ( \hy ) } \Delta^\nu ( \IE ; u; \hy ),
		$$ 
		where the intersection runs over all systems $ ( \hy ) $ extending $ ( u ) $ to a regular system of parameters for $ \hR $ and which fulfill the additional condition $ \nu ( \hy_j ) = \beta_j $, $ 1 \leq j \leq r $.
		We set $ \delta^\nu := \delta ( \Delta^\nu ( \IE ; u ) ) $.
		
		All the notions and results of before can then be developed and proven in this setting.
		We only remark that in the $ \nu $-variant of Theorem \ref{Thm:BerndHiro} the assumption on the system $ ( y ) $ becomes: $ ( y ) $ determines the $ \nu $-directrix which is the directrix of the $ \nu $-initial forms on the weighted graded ring, where the weight is induced by $ \nu $.
		Furthermore, Theorem \ref{Thm:deltaINV} and Lemma \ref{Lem:DeltaEins} below are also true in the quasi-homogeneous situation.
	\end{Rk}

	%
	%
	%
	%
	%
	%
	%
	%
	%
	%
	%
	%
	%

	\section{First properties, \textcolor{black}{maximal contact,} and invariants of the polyhedron}

	In this section, we discuss some properties of the characteristic polyhedron and deduce crucial information on the singularity from it. 
	
	\smallskip 
	
	We have seen that the polyhedron may change under the equivalence relation $ \sim $, Example \ref{Ex:PolyNotUnique}.
	But we can also say when the polyhedron is stable.
	
	\begin{Lem}
		\label{Lem:PolyEquiv}
		Let $ \IE = (J,b) $ and $ \IE_i = (J_i,b_i) $, $ i \in \{1, 2 \}$, be pairs on $ Z $ and $ x \in \Sing ( \IE ) $ resp.~$ x \in \Sing ( \IE_1 \cap \IE_2 ) $.
		As usual $ (R, \maxIdeal , \resfield ) $ denotes the regular local ring of $ Z $ at $ x $ and $ (u,y) \color{black} = (u_1, \ldots, u_e, y_1, \ldots, y_r )$ is a regular system of parameters for $ R $.
		We abbreviate the notation by setting $ \Delta (J,b) :=  \poly{(J_x,b)}uy $.
		\begin{enumerate}
			\item 
			If $ a \in \IZ_+ $, then $ \Delta (J,b) = \Delta (J^a, ab) $. 
			
			\item 
			Suppose $ b_1, b_2 \in \IZ_+ $ and let $ c \in \IZ_+ $ with $ b_1 \mid c $ and $ b_2 \mid c $.
			Then  
			$$
			\Delta ( (J_1,b_1) \cap (J_2,b_2) ) = \Delta\left ( J_1^{\frac{c}{b_1}} + J_2^{\frac{c}{b_2}}, c \right).
			$$
			
			\item 
			Suppose $ R $ is of the form $ R = S [w_1,\ldots, w_d]_{\langle v, w \rangle } $, for a regular local ring $ S $ with regular system of parameters $ ( v ) = ( v_1, \ldots, v_c ) $.
			For $ Q \in \IZ^d_{ \geq 0 } $, let $ \cD_{Q,\textcolor{black}{w}} \in \Diff^{\leq q}_{S} (  R ) $, $ q := |Q| $, be the differential operator defined by 
			\[  
			\cD_{Q,\textcolor{black}{w}} ( \coeff\, w^D)  = \binom{D}{Q} \coeff \, w^{D-Q} 
			\]
			for $ \coeff \in S $.
			We set $ \cD_{Q,\textcolor{black}{w}}^{log} := w^Q \cD_{Q,\textcolor{black}{w}}  \in \Diff^{\leq q}_{S} ( R ) $.
			Then
			\[ 
			\Delta \big( \, (J, b) \cap ( \cD_{Q,\textcolor{black}{w}}^{log} \textcolor{black}{(}J \textcolor{black}{)}, b - q)\,\big) =  \Delta ( J, b).
			\]
			
			\item 
		 	\color{black} 
		 	Suppose $ R $ is of the form $ R = S [y_1,\ldots, y_r]_{\langle u, y \rangle } $, for a regular local ring $ S $ with regular system of parameters $ ( u ) = ( u_1, \ldots, u_e ) $.
			Let $ \cD_{Q,y} \in \Diff^{\leq q}_{S} (  R ) $ be as in (3),
			for $ Q \in \IZ^r_{ \geq 0 } $ and $ q := |Q| $.
			We have
			\begin{equation}
			\label{eq:poly_D_y_same}   
			\Delta \big( \, (J, b) \cap ( \cD_{Q,y} (J), b - q)\,\big) =  \Delta ( J, b).
			\end{equation} 
		\end{enumerate}
	\end{Lem}
	
	\begin{proof}
		The proofs emerge from a study of the vertices' behavior under the equivalences $ ( J, b ) \sim ( J^a, ab) $, $ ( J_1, b_1 ) \cap ( J_2, b_2 ) \sim ( J_1^{\frac{c}{b_1}} + J_2^{\frac{c}{b_2}}, c ) $, $ (J, b) \cap ( \cD_{Q, \textcolor{black}{w}}^{log} J, b - q) \sim ( J, b ) $.
		For a detailed proof of \textcolor{black}{(1)--(3)}, see \cite[Lemma 2.4.4 and Lemma 2.4.5]{BerndThesis}.
		
		\color{black} Let us discuss (4).
		As $ \poly{\IE_1 \cap \IE_2}{u}{y} $ is the smallest closed convex set containing $ \poly{\IE_1}uy $ and  $ \poly{\IE_2}uy $,
		it suffices to show 
		$    \Delta \big( \, ( \cD_{Q,y} (J), b - q)\,\big) 
		\subseteq  \Delta ( J, b) $.
		We prove this inclusion for any $ f \in J $,
		i.e., we show
		\begin{equation}
		\label{eq:poly_D_y_same_f}   
		\Delta \big( \, ( \cD_{Q,y} (f), b - q); u ;  y \,\big) \subseteq  \Delta ( (f, b);u;y).
		\end{equation} 
		By \cite[Remark~1.9]{BerndPartial}, we can write $ f $ as
		$ f = \sum_{\substack{B \in \IZ^e_{ \geq 0 }\\|B|<b}} f_B (u)\, y^B + h, $
		for some $ h \in \langle y \rangle^b $ and coefficients $ f_B(u) \in S $.
		We have $ \cD_{Q,y} ( h  ) \in \langle y \rangle^{b-q} $.
		Hence, this element cannot contribute to $ \poly{( \cD_{Q,y} (f), b - q)}uy $.
		
		Let $ u^A y^B $ be a monomial appearing in $ f $ with non-zero coefficient in $ S^\times $.
		Suppose that $ |B|< b $.
		Hence, the monomial corresponds to $ v := \frac{A}{b-|B|} \in  \poly{(f,b)}uy $. 
		For $ \binom{B}{Q} u^A y^{B- Q} $, we obtain in $ \poly{( \cD_{Q,y} (f), b - q)}uy $ the same point:
 		\[ 
			\frac{A}{b- q - |B-Q|} 
			= \frac{A}{b- q - |B|+|Q|} = \frac{A}{b-|B|} = v 
		\] 
		This provides \eqref{eq:poly_D_y_same_f}, which implies \eqref{eq:poly_D_y_same}. 
	\end{proof}
	
	\color{black}
	
	The following examples shows that, in general, 
	part (4) does not hold if the differential operator involves the variables $ (u) $.
	
	\begin{Ex}
		\label{Ex:poly_derivative_bad}
		Let $ k $ be a field of characteristic different from $ 2 $ and $ 3 $. 
		Consider the pair $ ( f, 3 ) $ in $ k [u_1, u_2, y]_{\langle u_1, u_2, y \rangle} $,where
		\[
			f = y^3 + u_1^2 ( y + u_2^\alpha ) + u_2^\beta, 
		\]
		for some $ \alpha, \beta \in \IZ_{\geq 4} $ with $ 3 \alpha < \beta $. 
		The polyhedron $ \poly{(f,3)}uy $ has vertices 
		$ ( 1, 0) $, 
		$ \frac{(2, \alpha)}{3} $,
		$ \frac{(0, \beta)}{3} $.
		If we apply the derivative $ \frac{\partial}{\partial u_1} $ twice, we get
		\[
			(f,3) \sim (f,3) \cap (y + u_2^\alpha , 1).
		\] 
		The second pair on the right hand side provides the new vertex 
		$ (0,\alpha) $ in the polyhedron $ \poly{(f,3) \cap (y + u_2^\alpha , 1)}uy $. 
		So, we have
		\[ 
		\Delta \big( \, (f, 3) \cap ( \tfrac{\partial^2}{\partial u_1^2} (f), 1 ); u ;  y\,\big) \supsetneq   \Delta ( (f, 3); u ;  y).
		\]
	\end{Ex}

	\color{black}

	\smallskip
	
	Let us now introduce the concept of maximal contact, which is an important tool in the proof of resolution of singularities in characteristic zero.
	It allows to reduce the local resolution problem to one in smaller dimension so that one can apply an induction.
	Classical references are \cite{GiraudMaxZero,AHV}.
	
	\begin{Def}
		Let $ \IE = (J,b) $ be a pair on $ Z $ and $ x \in \Sing ( \IE ) $.
		Let $ (z) = (z_1, \ldots, z_s) $ be a system of elements in the local ring $ ( R = \cO_{Z,x}, \maxIdeal) $ which can be extended to a regular system of parameters for $ R $.
		We say $ W := V(z) $ has {\em maximal contact with $ \IE $ at $ x $} if the following equivalence holds
		\[ 
		\IE_x = ( J_x, b ) \sim (z, 1) \cap (J_x,b).
		\]
	\end{Def}
	
	In particular, the images of $ (z) $ in $ \maxIdeal / \maxIdeal^2 $ are part of a minimal generating system for the ideal of the directrix $ \Dir_x (\IE) $.
	By the Numerical Exponent Theorem \ref{Prop:NumExp}, this can happen only if $ \ord_x ( \IE) = 1 $. 
	We discuss later (Lemma \ref{Lem:ExistenceMaxContact}) assumptions under which a maximal contact subvariety exists. 
	
	\smallskip
	
	Now, we prove that the projected polyhedron $ \poly \IE uy $ is independent of the choice of $ ( y ) $ if $ V(y) $  has maximal contact with $ \IE$ (at $ x $) 
	\textcolor{black}{and if the elements $ ( y ) $ are obtained by applying differential operators in the variables $ ( y) $}.
	In particular, it coincides with the characteristic polyhedron if $ (y) $ define the directrix of $ \IE' $ at $ x' $ (cf.~Theorem \ref{Thm:BerndHiro}).
	Hence we obtain the analogue of Cossart's result \cite{CosPoly} (which is in the complex analytic case) in our setting (cf.~\cite[Corollaire, p.~18]{CosPoly}). 
	
	\begin{Prop}
		\label{Prop:maxSamePoly}
		Let $ \IE = (J,b) $ be a pair on $ ( R, \maxIdeal ) $.
		Fix a system of elements  $ ( u ) = ( u_1,, \ldots, u_{ d }) $ which can be extended to a regular system of parameters for $ R $.
		Let $ (y) = ( y_1, \ldots, y_s ) \subset R $ 
		\color{black} 
		be elements such that
		\begin{enumerate}
			\item[$(i)$] 
			$ ( u,y ) $ is a regular system of parameters for $ R $, 
			
			\item[$(ii)$]
			 $ R \cong S[y]_{\langle u ,y \rangle} $,
			 where $ S $ is a regular local ring with regular system of parameters $ ( u ) $,
			 and
			 
			 \item[$(iii)$] 
			 that, for every $ j \in \{ 1, \ldots, d \} $,
			 there exists a differential operator $ \cD_{Q_j,y} \in  \Diff^{\leq b-1}_{S} (  R ) $
			 and an element $ f_j \in J $ such that
			 $
			 \cD_{Q_j,y} (f_j) = y_j. 
			 $
		\end{enumerate}
	
	\color{black} 	
		
		Then 
		\textcolor{black}{$ V ( y ) $ has maximal contact with $ \IE $ at the origin and}
		the polyhedron 
		$ 
		\poly \IE uy
		$
		is independent of the choice of $ ( y ) $ with these properties.
		This means if $ (z) \subset R $ is 
		\textcolor{black}{a system of elements such that analogue of (i), (ii), (iii) holds,} 
		then 
		\[   
			\poly \IE uy  =  \poly \IE uz .
		\] 
		Furthermore, we have that $ \poly \IE uy = \cpoly \IE u $ if the images of $ ( y ) $ in $ R' = R/\langle u \rangle $ define the directrix of $ \IE' = ( JR', b ) $.
	\end{Prop}
	
	\begin{proof}
		\color{black} 
		The Diff Theorem (Proposition~\ref{Prop:Diff}) and $ \cD_{Q_j,y} \in  \Diff^{\leq b-1}_{S} (  R ) $ imply 
		\[
			(J,b) 
			\sim
			(J,b) \, \cap \, \bigcap_{j=1}^s ( \cD_{Q_j,y}(f_j),1)
			=
			(J,b) \, \cap \, (y,1).
		\]
		Hence, $  V(y) $ has maximal contact.
		
		Let us consider $ ( z ) = (z_1, \ldots, z_s) $ as in the statement. 
		Since $ (u, z ) $ is a regular system of parameters for $ R $, we can express $ (y) $ by these elements.
		As in the proof of Theorem \ref{Thm:BerndHiro}  (cf.~(\ref{eq:expan_y_by_z_u})), we have
		$
		y_j =  L_j ( z ) + Q_j' ( u ) + H_j ( u ,z ) ,
		$
		where $ L_j ( z ) = \sum_{i=1}^r \epsilon_{ij} z_i $  with $ \epsilon_{ij} \in R^\times \cup \{ 0 \} $, and $\langle \,  L_1 (z) , \ldots, L_r (z) \, \rangle = \langle z_1, \ldots, z_r \rangle \subset R $,
		$ H_j ( u, z ) \in  \langle u, z \rangle^2 \cap \langle z \rangle $,
		and 
		$ Q_j' (u) = \sum D_{A,j} u^A $, for $ D_{A,j} \in R^\times \cup \{ 0 \} $ and $ |A|\geq 1 $.
		As we already have seen in the proof of Theorem \ref{Thm:BerndHiro} (steps (1) and (2)), we do not change the polyhedron if we replace $ y_j $ by $  L_j ( z )  + H_j ( u ,z ) $.
		
		Let us assume that $ y_j = z_j + Q_j' ( u ) $.
		Hence, the differential operators $ \cD_{Q_j,z} $ can be replaced by $ \cD_{Q_j,y} $. 
		By Lemma~\ref{Lem:PolyEquiv}(4) and assumption $ (iii) $ for $ (z) $, we have
		\[
			\poly{(J,b)}uy 
			= 
			\Delta \Big( 
			(J,b) \cap \bigcap_{j=1}^s(y_j - Q_j'(u),1)
			; u ; y \Big) . 
		\]
		In particular, $ A \in \poly{(J,b)}uy $ for each $ A \in \IZ_{\geq 0}^d $ with $  D_{A,j} \neq 0 $. 
		By the analogous argument, we have $ A \in  \poly{(J,b)}uz $ for every $ A \in \IZ_{\geq 0}^d $ such that $  D_{A,j} \neq 0 $.
		
		Let $ g \in J $ and let $ u^\alpha y^\beta $ be a monomial appearing in an expansion of $ g $. 
		Consider the change $ y_j \mapsto  y_j + D_{A,j} u^A $, for $ j \in \{ 1, \ldots, s \} $. 
		As in the proof of Theorem~\ref{Thm:BerndHiro} (Step 2),
		the points possibly created by the change in the polyhedron are of the form
		\[
			v' := \frac{\alpha + A \cdot |M|}{ b - ( \, |\beta - M|  \, ) } ,
		\] 
		where $ M = (M_1, \ldots, M_s) $ and $ M_j \in \{ 0, \ldots, \beta_j \} $ for $ j \in \{ 1, \ldots, s \} $. 
		
		If $ | \beta | < b $, then $ u^\alpha y^\beta $ provides a point $ v \in \poly{(J,b)}uy $
		which remains (case $ M = (0, \ldots, 0 ) $).
		For $ |M|> 0 $, we have
		\[
			v' = \frac{b-|\beta|}{b - |\beta|+| M|} \cdot v 
			+ \frac{|M|}{b - |\beta|+ | M|} \cdot  A.
		\]
		Hence, $ v' $ is a point lying on the line segment connecting $ v $ and $ A $, but is different from $ v $ and $ A $ since $ |M|>0 $. 
		Therefore, $ v' $ is contained in both polyhedra,
		$ \poly{(J,b)}uy $ and $ \poly{(J,b)}uz  $. 
		
		On the other hand, if $ |\beta| \geq b $, 
		then we have $ b - | \beta | + |M| \geq |M| $ and thus
		\[
			\frac{|M|}{b - |\beta|+ | M|} \cdot  A  \in A + \IR_{\geq 0}^d.
		\]
		This implies that $ v' \in A + \IR_{\geq 0}^d $,
		i.e., $ v' \in  \poly{(J,b)}uy $ and $ v' \in \poly{(J,b)}uz $. 
		
		Applying this argument for all terms of $ Q_j' (u) $,
		we obtain the desired equality $ \poly \IE uy  =  \poly \IE uz  $. 
		
		\color{black}

		Suppose $ ( y ) $ fulfills the \textcolor{black}{directrix} condition of the last part.
		Then Lemma \ref{lem:cleaning2} holds.
		\textcolor{black}{If a change $ z_j := y_j + \lambda_j u^v $ eliminates a vertex $ v \in \poly{\IE}uy $, for some $ \lambda_j \in R^\times \cup \{ 0 \} $,
			then $ (z) = (z_1, \ldots, z_s) $ fulfills the properties $ (i), (ii), (iii) $ above.
		This implies $ \poly \IE uy  =  \poly \IE uz  $ in contradiction to the assumption that the change eliminate $ v $.
		Hence, the last} statement follows.
	\end{proof}

	While maximal contact is one of the key ingredients for resolution of singularities over fields of characteristic zero, it does not always exist in general \cite{Nara, GiraudMaxPos, CosIsThere}. 
	Further, the polyhedron $ \Delta (\IE, u) $ may change under $ \sim $.
	Nonetheless, we can extract invariants of the idealistic exponent $ \IE_\sim $ from the characteristic polyhedron of a representative.

	\begin{Def}
		\label{def:delta(Delta)}
		Let  $ \Delta \subset \IR^d_{ \geq 0 } $ be any subset.
		We define 
		$$ 
		\delta ( \Delta ) := \inf \{ \, |v| = v_1 + \ldots + v_d \mid v = (v_1, \ldots, v_d) \in \Delta  \, \}.
		$$
		If $ \Delta = \Delta_x(\IE; u;y) $, then we set $ \delta_x (\IE; u;y) := \delta (\Delta_x( \IE; u;y) ) \in \frac{1}{b!} \IZ_{+}  \cup \{ \infty \}$.
	\end{Def}
	
	Note the following: 
	We have $ \delta_x (\IE;u;y) = \infty $ if and only if $ f \in \langle y \rangle^b $ for every $ f \in J_x $, $ \IE= (J,b) $,
	(Definition \ref{Def:NewtonPolyandPolyNonIntrinsic}(2)). 
	Furthermore, the latter is equivalent to $ V(y) $ being a permissible center for $ \IE$ (Definition \ref{Def:perm}).
	
	\smallskip
	
	\begin{Prop}
		\label{Prop:deltaEuyInv}
		Let $ \IE_1 \sim \IE_2 $ be two equivalent pairs on $ Z $ and $ x \in \Sing (\IE_1 ) $.	
		Let $ (u,y) = (u_1, \ldots, u_d; y_1, \ldots, y_s) $ be a regular system of parameters for $ \cO_{Z,x} $.
		\begin{enumerate}
			\item Then we have:
			$$ 
			\delta_x ( \IE_1; u; y ) = \delta_x ( \IE_2; u; y ) .
			$$
			\item Let $ ( u, z ) $ be another choice for the regular system of parameters and suppose $ (z,1) \cap \IE_1 \sim (y,1) \cap \IE_1 $.
			Then
			$$ 
			\delta_x ( \IE_1; u; y ) = \delta_x ( \IE_1; u; z ) .
			$$
			\red{In particular, if there are maximal contact
			coordinates, this number is independent of the choice of the maximal contact coordinates.}
		\end{enumerate}
	\end{Prop}

	Note that the condition $ (z,1) \cap \IE_1 \sim (y,1) \cap \IE_1 $ does not necessarily imply that $ V(y) $ has maximal contact with $ \IE_1 $ at $ x $ (i.e., $ \IE_1 \sim (y,1) \cap \IE_1 $).

	\begin{Rk}
		The proof of Proposition~\ref{Prop:deltaEuyInv} is  similar to the one for the Numerical Exponent Theorem, Proposition \ref{Prop:NumExp}.
		The idea is to construct a {\LSB} contradicting the equivalence $ \IE_1 \sim \IE_2 $ 
		if the desired statement would be false.
		An analogous result for the characteristic polyhedron of ideals is shown in \cite[Theorem 4.18]{HomeworkDim2}.
		More generally, \cite{CosRevista} (over $ \IC $) and \cite{CJSc} (arbitrary case) discuss results concerning the question which data obtained from the characteristic polyhedron of an ideal $ J $ is an invariant of the singularity $ \Spec(R/J) $ (and additional data connected to exceptional divisors arising along a resolution process). 
	\end{Rk}

	\begin{proof}
		Set $ \IE = ( J, b) := \IE_i := (J_i, b_i ) $, for some $ i \in \{ 1, 2\} $.
		We abbreviate 
		\[ 
		 	\delta := \delta_x ( \IE; u; y ) .
		\] 
		Since $ x \in \Sing(\IE) $, we have that 
		$
			\delta \geq 1 
		$:
		Suppose that $ \delta < 1 $, then there would be an element $ h \in J $ such that a monomial $ u^A y^B $ with 
		$ \frac{|A|}{b - |B|} = \delta < 1 $ appears with non-zero coefficient in an expansion of $ h $.
		But this would imply $ |A|+|B| < b $, i.e., $ \ord_x(h) < b $, which contradicts the hypothesis $ x \in \Sing (\IE) $.
		
		\smallskip 
		
		Next, we construct a permissible {\LSB}, whose length is closely connected to $ \delta $:
		Let $ f \in J $.
		From an expansion \eqref{expansion}, we can deduce that $ f $ can be written as a finite sum
		\[ 		
		f = f_b(y) + \sum_{(A,B): |B|<b} C_{A,B} u^A y^B + h_+ ,
		\] 
		with $ f_b(y) = \sum_{B: |B|=b} C_{0,B} y^B \in \langle y \rangle^b $, units $ C_{0,B}, C_{A,B} \in R^\times \cup \{ 0 \} $, and $ h_+ \in \langle y \rangle^{b+1} $.
		We introduce a new independent variable $ t $ and work in $ S := R[t] $.
		We blow up with center $ V ( t, u,y ) $ which is permissible for $ \IE $ as $ x \in \Sing(\IE) $.
		We consider the origin of the $ T $-chart, i.e., the point with coordinates 
		$ (t', u',y') = (t, \frac{u}{t}, \frac{y}{t})$.
		Then $ f $ becomes
		$$
		f_b(y') + \sum_{(A,B): |B|<b} C_{A,B} \, t'^{|A|+|B|-b} u'^A y'^B + \widetilde  h_+ ,
		\, \, \, \mbox{ for } \widetilde h_+ \in \langle y' \rangle^{b+1}.
		$$
		Note that the exponent of $ t' $ in the sum can be written as
		$$ 
		|A|+|B|-b = \left(\frac{|A|}{b- |B|} - 1 \right)(b - |B|) \geq (\delta - 1) (b - |B|) .
		$$
		Furthermore, we may choose $ f \in J $ such that there exist $ ( A, B) $ with $ C_{A,B} \neq 0 $ and
		$|A|+|B|-b = (\delta - 1) (b - |B|) $.
		
		We repeat the last step $ N $-times (i.e., next we blow up $ V(t',u',y' ) $ and consider the origin of the $ T' $-chart), for some $ N \in \IZ_+ $, $ N \gg  0 $.
		Let $ ( \widetilde t, \widetilde u, \widetilde y) := ( t, \frac{u}{t^N}, \frac{y}{t^N}) $ be the resulting local coordinates and let $ \widetilde \IE $ be the transform of $ \IE $. 
		In the transform of $ f \in J $, the $ t $-power for a monomial resulting from $ u^A y^B $ with $ |B|<b $ is 
		\[
		N(|A|+|B|-b) \geq N (\delta - 1) (b - |B|)
		\]
		and there exist $ f \in J $, where equality holds. 
		Therefore, we have:
		\[
			\forall \, v = (v_t, v_u) \in \poly{\widetilde{\IE}}{\widetilde t, \widetilde u}{\widetilde{y}} 
			\, : \, 
			v_t \geq N( \delta - 1)  ,
		\]
		\[
			\exists \, v = (v_t, v_u) \in \poly{\widetilde{\IE}}{\widetilde t, \widetilde u}{\widetilde{y}} 
			\, : \, 
			v_t = N( \delta - 1)  .
		\]
		Thus, the polyhedron $  \poly{\widetilde{\IE}}{\widetilde t, \widetilde u}{\widetilde{y}} \subset \IR^{1+d} $ can be pictured as:
		\[
		\begin{tikzpicture}[scale=0.5]
		
			\foreach \x in {0,...,8}
			\draw (\x,0.1) -- (\x,-0.1);
			
			\foreach \x in {0,...,6}
			\draw (0.1,\x) -- (-0.1,\x);
	
			\draw[<->, thick] (0,6.7)--(0,0)--(8.7,0);
			\node at (9.2,0) {$ v_t $};
			\node at (-0.5, 6.7) {$ v_u $};

			\fill[lightgray] (2.5, 6.5)--(2.5,4.5)--(4,2)--(6.5,1)--(8.5,1) -- (8.5,6.5);
			\draw[thick,dashed,gruen] (2.5,4.5)--(2.5,-0.25);
			\draw[very thick,cyan] (2.5, 6.5 )--(2.5, 4.5)--(4,2)--(6.5,1)--(8.5,1);
			
			\node at (2.5,-0.6) {\textcolor{gruen}{\small $ N(\delta - 1) $}};
			\node at (10.5,3.5) {$ \poly{\widetilde{\IE}}{\widetilde t, \widetilde u}{\widetilde{y}} $};

		\end{tikzpicture} 
		\]	
		If $ \delta > 1 $ and $ N $ is large enough, we have $ N(\delta - 1) \geq 1 $, i.e., $ V (\widetilde t, \widetilde y) $ is permissible for $ \widetilde \IE $. 
		If we consider the origin of the $ \widetilde T $-chart, we have coordinate 
		$ (\widetilde t', \widetilde u', \widetilde y') := (\widetilde t, \widetilde u, \frac{\widetilde y}{\widetilde t}) $.
		Let $ \widetilde E ' $ be the transform of $ \widetilde E $. 
		We observe that $ (v_t, v_u) \in \poly{\widetilde{\IE}}{\widetilde t, \widetilde u}{\widetilde{y}} $
		is mapped to $ (v_t - 1, v_u)\in \poly{\widetilde{\IE}'}{\widetilde t', \widetilde u'}{\widetilde{y}'} $.
		We may repeat the last step $ P := \lfloor N (\delta - 1) \rfloor $ times and obtain a {\LSB}, say $ \mathcal{S}(\delta, N) $, that is permissible for $ \IE $. 
		Clearly, if we choose $ N $ such that $ N \delta \in \IZ_+ $, then $ P = N (\delta - 1) $.   
		
		\smallskip
		
		\noindent 
		\emph{Let us come to (1):}
		Set $ \delta_ 1 := \delta_x ( \IE_1; u; y ) $ and $ \delta_2 := \delta_x ( \IE_1; u; y ) $. 
		Suppose the statement is wrong;
		without loss of generality, we assume that $ \delta_2 < \delta_1 $. 
		If $  \delta_1 = \infty $, then $ V(y) $ is permissible for $ \IE_1 $ and hence also for $ \IE_2 $ since $ \IE_2 \sim \IE_1 $.
		But this is a contradiction to $ \delta_2 < \delta_1 = \infty $.
		
		Thus, we may suppose $ \delta_2 < \delta_1 < \infty $. 
		Since $ \delta_2 \geq 1 $, we have $ \delta_1 > 1 $. 
		Choose $ N \in \IZ_+ $ such that $ N \delta_1, N\delta_2 \in \IZ_+ $.
		By construction, $ \mathcal{S}(N,\delta_1) $ is permissible for $ \IE_1 $.
		Since $ \IE_2 \sim \IE_1 $, the {\LSB} $ \mathcal{S}(N,\delta_1) $ is also permissible for $ \IE_2 $. 
		But this provides a contradiction: 
		Using the analogous notation of before, we have
		\[
		\exists \, v = (v_t, v_u) \in \poly{\widetilde{\IE_2}}{\widetilde t, \widetilde u}{\widetilde{y}} 
		\, : \, 
		v_t = N( \delta_2 - 1)  .
		\]
		Further, along the blowups of the second part of $ \mathcal{S} (N,\delta_1) $ the point $ (v_t,v_u) $ is mapped to 
		$ (v_t - P, v_u ) $, where $ P = N(\delta_1 - 1) $. 
		But we have 
		\[ 
			v_t  - P = v_t = N( \delta_2 - 1) -  N(\delta_1 - 1) =  N(\delta_2 - \delta_1) < 0 ,
		\]
		which implies that $ \mathcal{S}(N,\delta_1) $ is not permissible for $ \IE_2 $. 
		
		\smallskip
		
		\noindent 
		\emph{For {(2)}:}
		Set $ \IE := \IE_1 $.
		Observe that $ \delta_x ( \IE \cap (y,1); u; y ) = \delta_x (\IE; u ; y) $.
		We define $ \delta_y := \delta_x (\IE; u ; y) $ and $ \delta_z := \delta_x (\IE; u ; z) $.
		Suppose the statement is wrong;
		without loss of generality, we assume that $ \delta_y < \delta_z $. 
		If $  \delta_z = \infty $, then $ V(z) $ is permissible for $ \IE \cap (z,1) $ and hence also for $ \IE \cap (y,1)$. 
		On the other hand, every center that is permissible for $ \IE \cap (y,1)$ must be contained in $ V (y) $.
		Hence, we have $ V(z) \subseteq V (y) $ and equality follows as both have the same dimension. 
		But $ V(y) $ begin permissible for $ \IE $ contradicts $ \delta_y  < \delta_z = \infty $.
		
		Thus, we may suppose $ \delta_y < \delta_z < \infty $. 
		Since $ \delta_y \geq 1 $, we have $ \delta_z > 1 $. 
		Choose $ N \in \IZ_+ $ such that $ N \delta_y, N\delta_z \in \IZ_+ $.
		After the first block of permissible blowups of $ \mathcal{S}(N,\delta_z) $,
		we obtain that $ V(\widetilde t, \widetilde z)  $ is permissible for $ \widetilde \IE $ 
		(using the notation analogous to the beginning of the proof).
		By repeating the argument of the $ \delta_z = \infty $ case and using that $ t $ is a variable that is independent of $ (y) $ and $ (z) $, we obtain that 
		$ V(\widetilde t, \widetilde y)  $ is also permissible for $ \widetilde \IE $.
		Continuing this, we obtain a step-by-step identification of $ \mathcal{S}(N,\delta_z) $ and $ \mathcal{S}(N,\delta_y) $, which eventually leads to a contradiction 
		since both have different lengths as 
		$ \delta_y < \delta_z $.
	\end{proof}

	{\red{Proposition~\ref{Prop:deltaEuyInv} provides one of our main results.}}
	
	\begin{Thm}[Theorem \ref{MainThm:deltaINV}]
		\label{Thm:deltaINV}
		Let $ \IE $ be a pair on $ Z $ and $ x \in \Sing ( \IE ) $. 
		Let $ ( u, y ) $ be a regular system of parameters of $ \cO_{ Z,x } $ such that $ ( y ) $ determines $ \Dir_{x'} ( \IE' ) $ (Notation \ref{nota:E'}).
		
		The rational number $ \delta_x ( \IE, u ) $ does not depend on $ ( y ) $ and is invariant under the equivalence relation $ \sim $.
		Therefore $ \delta_x ( \IE, u )_\sim $ is an invariant of the idealistic exponent $ \IE_\sim $ and $ ( u ) $.
	\end{Thm}
	
	\begin{proof}
		By definition, $ \delta_x ( \IE, u ) $  does not depend on $ ( y ) $.
		By Theorem \ref{Thm:BerndHiro} it is attained by some $ ( \hy ) $ living in the completion of the local ring at $ x $.	
		Proposition \ref{Prop:deltaEuyInv} implies then the invariance under $ \sim $.
	\end{proof}
	
	If we drop the assumption on $ ( y ) $ to give the directrix, then we do not know if there is a polyhedron which is independent of the system $ ( y ) $;
	we have shown in Example \ref{Ex:AssumpThmNecessary} that we are not able to make $ \Delta ( (f, b) ; u, y ) $ independent of this choice.
	
	But still we can say something in the case, where $ (y) = (y_1, \ldots, y_s ) $ can be extended to a system $ (y_1, \ldots, y_r), r > s ,$ which yields the directrix:
	
	\begin{Lem}
		\label{Lem:DeltaEins}
		Let $ \IE = (J,b) $ be a pair on $ Z $ and $ x \in \Sing ( \IE ) $ as before 
		(thus $ \ord_x ( J ) \geq b $).
		Fix a system of elements $ ( u_1, \ldots, u_d ) $ in $ R = \cO_{Z,x} $ which can be extended to a regular system of parameters for $ R $.
		Let $ ( y ) = ( y_1, \ldots, y_s ) $ be such an extension of $ (u) $.
		Assume that $ (y, u_{ e + 1 }, \ldots, u_d ) $, $ e < d $, defines the directrix $ \Dir_{x'} ( \IE' ) $.
		Then we have
		$$
		\delta_x ( \, \poly{\IE}{u_1,\ldots, u_d}{ y_1, \ldots, y_s} \, ) = 1 
		$$
	\end{Lem}
	
	In particular, this is independent of the choice of $ ( y ) $ and invariant under $ \sim $.
	Therefore it is an invariant only depending on the idealistic exponent $ \IE_\sim $ and $ ( u ) $.
	
	\begin{proof}
		By assumption there is an $ f \in J $ with $ \ini ( f , b ) \notin \resfield [ Y_1, \ldots, Y_s ] $.
		Hence its expansion $ f = \sum_{(A,B)} \coeff_{A,B} \, u^A \, y^B $ there is an $ (A,B) $ such that
		$$
		\coeff_{A,B} \neq 0, 
		\hspace{10pt}
		|A| \neq 0 
		\hspace{10pt}
		\mbox{ and }
		\hspace{10pt}
		|A| + |B| = b.
		$$
		Since $ I\Dir_x (\IE) = \langle Y_1, \ldots, Y_s, U_{ e + 1 }, \ldots U_d \rangle $, we can choose $ ( A,B) $ such that the corresponding monomial cannot be deleted by any coordinate changes.
		Then $ ( A,B) $ yields in the polyhedron $  \Delta_x (  \IE; u_1,\ldots, u_d; y_1, \ldots, y_s  ) $ the point $ v : = \frac{ A }{ b - |B| } $ with $ | v | = 1 $.
		Further, $ \ord_x ( J ) \geq b $ implies
		$$
		\delta_x (  \,  \poly { \IE }{ u_1,\ldots, u_d }{ y_1, \ldots, y_s } \, ) \geq 1.
		$$
		Together this yields the assertion.
	\end{proof}

	\color{black}
	
	Using the invariant $ \delta_x (\IE;u;y) $ of Proposition~\ref{Prop:deltaEuyInv}, we may introduce the following refinement of the $ b $-initial form (and the tangent cone). 
	
	\begin{Def}
		\label{Def:in_delta_E}
		Let $ R $ be a regular local ring
		with maximal ideal $ \maxIdeal $ and residue field $ \resfield = R/\maxIdeal $.
		Let $ \IE = ( J ,b ) $ be a pair on $ R $.
		We assume that the closed point $ x $ of $ \Spec(R) $ is contained in $ \Sing (\IE ) $.	
		Let $ (u,y) = (u_1, \ldots, u_d; y_1, \ldots, y_s) $ be a regular system of parameters for $ R $
		and set 
		\[  
		\delta := \delta_x (\IE;u;y) .
		\] 
		For $ f = \sum_{(A,B)} C_{A,B} u^A y^B \in J $ with $ C_{A,B} \in R^\times \cup \{ 0 \} $,
		we define the 
		{\em $ (\delta,b) $-initial form of $ f $}
		as 
		\[
		\ini_\delta (f,b) 
		:= 
		\sum_{|B|=b} \overline{C_{0,B}} Y^B 
		+
		\sum_{\frac{A}{b- |B|} =  \delta } \overline{C_{A,B}} U^A Y^B
		\in \resfield [U_1, \ldots, U_d, Y_1, \ldots, Y_s] , 
		\]
		where $ \overline{C_{A,B}} $ are the residue classes in $ \resfield $. 
		Furthermore, we introduce
		\[
		\ini_\delta ( \IE) = \ini_\delta ( J,b )
		= \langle \ini_\delta (f,b) \mid f \in J \rangle. 
		\]
	\end{Def}
	
	\smallskip 
	
	We have the following corollary of Proposition~\ref{Prop:deltaEuyInv}.
	
	\begin{Thm}[Theorem \ref{MainThm:deltaINV_2}]
		\label{Thm:deltaINV_2}
		Let $ \IE_1 \sim \IE_2 $ be two equivalent pairs on $ Z $ and $ x \in \Sing (\IE_1 ) $.	
		Let $ (u,y) = (u_1, \ldots, u_d; y_1, \ldots, y_s) $ be a regular system of parameters for $ R := \cO_{Z,x} $.
		Suppose that $ \IE_i $ is determined by $ (J_i,b_i) $ in $ R $, for $ i \in \{ 1, 2 \} $.  
		If we define $ \delta := \delta_x ( \IE_1; u; y ) $, 
		then we have  
		\[ 
		(\ini_{\delta}(\IE_1),b_1) \sim (\ini_{\delta}(\IE_2),b_2)
		.
		\]
	\end{Thm}
	
	\begin{proof}
		Since $ x \in \Sing(\IE_1 ) $, we have $ \delta \geq 1 $.
		If $ \delta = 1 $ and $ I\Dir_x (\IE_1) =  \langle Y_1, \ldots, Y_s \rangle $, then $ \ini_\delta (f) = \ini (f,b) $
		and $ (\ini_{\delta}(\IE_i),b_i) = \ITC_x ( \IE_i) $ is the idealistic tangent cone pair for $ i \in \{ 1, 2 \} $. 
		Hence,  Proposition~\ref{Prop:ItcDirRidUnique}(1) provides the statement. 
		 
		The idea for the general case is analogous to the one for the proof of Proposition~\ref{Prop:ItcDirRidUnique}(1). 
		Assume the claimed equivalence is wrong,
		$ (\ini_{\delta}(\IE_1),b_1) \not\sim (\ini_{\delta}(\IE_2),b_2) $. 
		Then there exists a {\LSB} that is permissible for one, but not the other.
		In order to lift this to the original pairs $ (J_1, b_1) $ and $(J_2, b_2)$, we first blow up $ V(u,y,T) $, where $ T $ is a new variable, and consider the origin of the $ T $-chart. 
		We repeat this $ d $ times and we assume that $ d \gg 1 $ is large enough such that $ V(y,T) $ is permissible and such that $ d \delta \in \mathbb{Z} $ is an integer. 
		Then, we blow up $ N := d ( \delta -1 ) $ times the center $ V(y,T) $ (here, we use that this number is an integer; 
		even though it is possibly zero). 
		Let $ \IE = (J,b) \in \{ \IE_1, \IE_2 \} $.
		Along the described sequence of blow-ups, an element 
		\[ 
		f = 
		\sum_{|B|=b} C_{0,B} y^B 
		+ 
		\sum_{\frac{|A|}{b-|B|} =\delta } C_{A,B} u^A y^B 
		+ 
		\sum_{\frac{|A|}{b-|B|} > \delta} C_{A,B} u^A y^B 
		+ h \in J ,
		\]
		for $ h \in \langle y \rangle^{b+1} + ( \langle y \rangle^b \cap \langle u \rangle ) $
		(and the other notation as in the previous definition), 
		transforms to $ f' := T^{-db-Nb}f $, which is
		\[ 
		\sum_{|B|=b} \widetilde{C_{0,B}} y^B 
		+ 
		\sum_{\frac{|A|}{b-|B|} =\delta } \widetilde{C_{A,B}} u^A y^B 
		+ 
		\underbrace{ 
			\sum_{\frac{|A|}{b-|B|} > \delta} \widetilde{C_{A,B}} u^A y^B T^{ (\frac{d |A|}{b-|B|} - d - N ) (b- |B|)} 
		}_{\displaystyle =: g }
		+ 
		T^\alpha \widetilde{h},
		\]
		where $ \widetilde{h} \in  \langle y \rangle^{b+1} + ( \langle y \rangle^b \cap \langle u \rangle ) $ and $ \alpha \in \IZ_+ $.
		(Clearly, we abuse notation here since we still denote the variables by $ (u,y, T) $ after the sequence of blow-ups.)
		Note that we write $ \widetilde{C_{A,B}} $ since the units $ C_{A,B}$ might have changed slightly, but that is not a problem. 
		
		If we choose $ d $ large enough, then we can get that $ T^b $ divides $ g  $ since $ \frac{|A|}{b-|B|} > \delta \geq 1 $ for all appearing $ (A,B) $.
		In particular, we have
		\[ 
		f' 
		\equiv \ini_\delta(f,b) \mod \langle T \rangle 
		\]
		Let us point out that we are working in $  S := R[\frac{u}{T^{d+N}}, \frac{y}{T^{d+N}}, T ] $
		(where $ (u,y) $ are the original variables in $ R $)
		and hence, $ S / \langle T \rangle \cong \resfield [U,Y] $
		by sending the residues of 
		$ \frac{y_j}{T^{d+N}} $ (resp.~of $ \frac{u_i}{T^{d+N}} $)
		modulo $ \langle T \rangle $
		to $ Y_j $ (resp.~$ U_i $) for all $ j $ (resp.~$ i $),
		and $ r \in R $ is sent to its residue in $ \resfield = R / \maxIdeal $, 
		as $ R \cap \langle T \rangle_S  = \maxIdeal $. 
		
		Applying the same arguments as in the proof of Proposition~\ref{Prop:ItcDirRidUnique}(1),
		we obtain that any {\LSB},
		which is permissible for
		$ (\ini_{\delta}(\IE_1),b_1) $, but not for 
		$ (\ini_{\delta}(\IE_2),b_2) $, 
		lifts to a {\LSB} in $ R[T] $, 
		which is permissible for $ \IE_1 $, but not for $ \IE_2 $. 
		This contradicts the hypothesis $ \IE_1 \sim \IE_2$.
		The case that there is a {\LSB}, which is permissible for
		$ (\ini_{\delta}(\IE_2),b_2) $, but not for 
		$ (\ini_{\delta}(\IE_1),b_1) $,
		follows analogously. 
	\end{proof}
	
	\color{black}

	A first application of the characteristic polyhedra of idealistic exponents and these results is given in \cite{BerndBM}, where the author deduces the invariant of Bierstone and Milman for constructive resolution of singularities in characteristic zero purely by considering certain polyhedra and their projections.

	%
	%
	%
	%
	%
	%
	%
	%
	%
	%
	%
	%
	%

	\section{Idealistic coefficient exponents}
	
	A useful tool to study the singularities at a point $ x \in Z $ in characteristic zero is the coefficient ideal \wrt a closed subscheme of maximal contact.
	
	We now give the precise definition of the coefficient ideal in the idealistic setting.
	But we do not restrict our attention to characteristic zero and admit an arbitrary residue field of $ Z $ at $ x $.
	It is known that the concept of maximal contact does not work in full generality. 
	Therefore we define the coefficient pair \wrt any regular subvariety $ W = V(z) = V(z_1, \ldots, z_s ) $ containing $ x $;
	we only want to assume that $ (z) $ is part of a regular system of parameters for the local ring $ R $ of $ Z $ at $ x $.
	(The interesting case for us is, when $ W = V ( y_1, \ldots, y_s ) $ ($ s \leq r $), where $ (y) = (y_1, \ldots, y_r ) $ is such that the image of $ (y) $ in $ \gr_x (Z) $ defines the directrix $ \Dir_x ( \IE ) $).
	
	\begin{Def}
		\label{Def:IdCoeffExp}
		Let $ \IE = (J,b) $ be a pair on $ Z $ and $ x \in Z $.
		Let $ (R = \cO_{Z,x}, \maxIdeal, \resfield) $ be the regular local ring of $ Z $ at $ x $.
		We consider a fixed system of elements $ (u) = (u_1, \ldots, u_d) $ which can be extended to a regular system of parameters for $ R $.
		Let $ (z) = ( z_1, \ldots, z_s ) $ be elements of $ R $ such that $ (u,z) $ is a regular system of parameters for $ R $.
		We define the {\em coefficient pair $ \ID_x (\IE, u, z )  $ of $ \IE $ at $ x $ \wrt $ (z) $} as the pair on $ W = \Spec ( R/\langle z \rangle ) $ which is given by the following construction:
		By \cite[Remark~1.9]{BerndPartial}, any $ f \in J_x $ may be written as
		$$
		f = f (u,z) = \sum_{\substack{B \in \IZ^s_{ \geq 0 }\\|B|<b}} f_B (u)\, z^B + h, 
		$$
		for some $ h \in \langle z \rangle^b $ and coefficients $ f_B(u) \in R $ that do not depend on $ (z) $.
		Then we set $ \ID(f,u,z)  := \bigcap\limits_{\substack{B \in\IZ^s_{ \geq 0 } \\[3pt] |B| < b}} ( f_B (u) , \, b -|B| ) $ and define further 
		$$
		\ID_x (\IE, u, z) := \bigcap_{ f \in J_x } \ID(f, u, z)  =  \bigcap_{ l = 0 }^{ b - 1 } \;\; (\; I(l, u, z) ,\; b - l \;),
		$$
		where $ I (l, u, z)  = \langle \, f_B \mid f \in J_x, B \in \IZ_{ \geq 0 }^s : |B| = l \,\rangle $.
	\end{Def}
	
	The idea of coefficient ideals goes back to Hironaka 
	(in the context of idealistic exponents this appears in \cite[Theorem 1.3, p.908]{HiroKorean} and \cite[section 8, Theorem 5, p.111]{HiroIdExp})
	and was developed by Villamayor (for basic objects) and Bierstone-Milman (for presentations).
	
	\smallskip 
	
	We may consider $ \ID_x (\IE, u, z )  $ as a pair on $ R/\langle z \rangle $ as well as on $ R $.
	
	An important invariant of the singularity of $ \IE $ at $ x $ is the order of the coefficient pair \wrt a system $ ( y ) $ which determines $ \Dir_x ( \IE ) $.
	Using Definition~\ref{def:delta(Delta)}, this can be recovered from the polyhedron $ \poly \IE uy $:

		\begin{Lem}
			\label{Lem_d_xPolyhedral}
			Let $ \IE = (J,b)  $ be a pair on $ Z $, $ x \in \Sing (\IE ) $, and $ (u,y) $ a regular system of parameters for the local ring $ \cO_{Z,x} $.
			Then we have
			\begin{enumerate}
				\item
				$ \Delta (\IE, u, y) = \Delta^\NW ( \ID_x (\IE, u, y) , u) \subseteq \IR^d_{ \geq 0 }  $
				and 
				
				\item 
				$
				\delta_x ( \Delta (\IE, u, y) )
				$
				coincides with the order of $ \ID_x (\IE, u, y)  $ at $ x $. 
			\end{enumerate}
		\end{Lem}
		
		This is an immediate consequence of Definition~\ref{Def:NewtonPolyandPolyNonIntrinsic}.		
		Note that we did not make any further assumptions on the system $ ( y ) $ (e.g. that it yields the directrix of $ \IE $).

	\vspace{5pt}
	In our context one of the first questions coming into one's mind may be the following:
	Are the coefficient pairs of equivalent pairs also equivalent?
	For the idealistic approach there is no reference known to the author where this is proven.
	Hence we give the affirmative answer in
	
	\begin{Thm}[Theorem~\ref{MainProp:IdCoeffExpIndp}(1)]
		\label{Thm:IDequiv}
		Let $ \IE_1 \subset \IE_2 $ be two pairs on $ Z $, $ x \in Z $, 
		and consider a regular system of parameters $ (u,z) = (u_1, \ldots, u_d; z_1, \ldots, z_s ) $ for $ (R = \cO_{Z,x}, \maxIdeal, \resfield) $.
		Then we have 
		\[
		\ID_x (\IE_1 , u, z) \subset \ID_x (\IE_2 , u, z) .
		\]
		By symmetry, $ \IE_1 \sim \IE_2 $ implies $ \ID_x (\IE_1 , u, z) \sim \ID_x (\IE_2 , u, z)  $.
	\end{Thm}
	
	This implies that an idealistic exponent $ \IE_\sim $ determines a unique idealistic exponent $ \ID_x (\IE , u, z)_\sim $, called the \emph{idealistic coefficient exponent} (with respect to $ (z) $).
	
	\begin{proof}
		For $ i \in \{ 1, 2 \} $, we consider
		$ \IE_{i,x} = (J_i, b) $ on $ R $.
		In order to simplify the notation we suppress the index $ x $ and write $ \IE_i = \IE_{i,x} $.
		Further, we define $ \ID_i := \ID_x (\IE_i , u, z)  $. 
		
		Suppose $ \ID_1 \not\subset \ID_2 $. 
		Then there exists a {\LSB} $ ( \diamondsuit ) $ over $ (R/ \langle z \rangle)[t] $ which is permissible for $  \ID_1[t] $, but it is not permissible for $ \ID_2[t] $, where $ (t) = (t_1, \ldots, t_a ) $ is a finite system of independent variables.
		We can lift it to a {\LSB} $ ( \widetilde{\diamondsuit} ) $ over $ R[t] $ by intersecting the centers with $ V(z) $.
		Then $ ( \widetilde{\diamondsuit} ) $ is permissible for $ (z,1) \cap \ID_1[t] $,
		but it is not permissible for $ (z,1) \cap \ID_2[t] $. 
		
		By construction, the local charts considered in $ ( \widetilde{\diamondsuit} ) $ contain the strict transform of $ V(z) $ and hence the construction of the coefficient is compatible with $ ( \widetilde{\diamondsuit} ) $.
		In particular, we observe that 
		$ \ID_i = \bigcap_{f \in J_i} \bigcap_{|B|<b } (f_B, b - |B|) $, $ i \in \{ 1, 2 \} $,
		(with $ f = \sum_{|B|<b} f_B(u) z^B + h $ an expansion as before) 
		provides that $ ( \widetilde{\diamondsuit} ) $ is permissible for $ \IE_1 $, but it is not permissble for $ \IE_2 $. Contradiction.
	\end{proof}

	As a consequence from the proof, we obtain 
	
	\begin{Cor}	
		\label{Cor:zE=zCoeff}
		Using the same notation as in the proof of Theorem~\ref{Thm:IDequiv}, we have:
		$$
		(z,1) \cap \IE \sim (z,1) \cap \ID_x (\IE , u, z)
		$$
		(Keep in mind that we have here the local situation at a point $ x $).
	\end{Cor}
	
	By the last theorem, $ \ID_x (\IE , u, z) $ is invariant under the equivalence relation $ \sim $ if we fix $ ( u, z ) $.
	But we might also consider various choices for $ ( z ) $. 
	In this case we have 
	
	\begin{Prop}[Theorem \ref{MainProp:IdCoeffExpIndp}(2)]
		\label{Prop:IDEz=IDEy}
		Let $ \IE $ be a pair on $ Z $ and $ x \in Z $. 
		Fix a system of elements $ ( u ) = (u_1, \ldots, u_d) $ which can be extended to a regular system of parameters for  $ (R = \cO_{Z,x}, \maxIdeal, \resfield) $.
		Let $ (z) = ( z_1, \ldots, z_s ) $ and $ (y) = ( y_1, \ldots, y_s ) $ be two possible extensions of $ (u) $.
		Assume $ (z,1) \cap \IE \subset (y,1) \cap \IE $.
		Then
		$$
		\ID_x (\IE , u, z) \subset \ID_x (\IE , u, y) .
		$$
		By symmetry, $ (z,1) \cap \IE \sim (y,1) \cap \IE $ implies $ \ID_x (\IE , u, z) \sim \ID_x (\IE , u, y)   $.
	\end{Prop}
	
	\begin{proof}
		First of all, Corollary \ref{Cor:zE=zCoeff} and the assumption imply
		\begin{equation} 
		\label{eq:noch_ein_altes_ast}
		(z,1) \cap \ID_x (\IE , u, z) \sim (z,1) \cap \IE 
		\subset 
		(y,1) \cap \IE \sim  (y,1) \cap \ID_x (\IE , u, y).
		\end{equation} 
		Let $ ( \diamondsuit ) $ be a {\LSB} over $ (R/ \langle z \rangle)[t] $ which is permissible for $  \ID_x (\IE , u, z) $.
		We can lift it to a {\LSB} $ ( \widetilde{\diamondsuit} ) $ over $ R $  by intersecting the centers with $ V(z) $. 
		Then $ ( \widetilde{\diamondsuit} ) $ is permissible for $ (z,1) \cap \ID_x (\IE , u, z) $ and by \eqref{eq:noch_ein_altes_ast} it is so for   
		$ (y,1) \cap \ID_x (\IE , u, y) $.
		In particular, it is permissible for $ \ID_x (\IE , u, y) $ and since the latter lives on $ (R/ \langle z \rangle)[t] $, the {\LSB} $ (\diamondsuit) $ is permissible for $ \ID_x (\IE , u, y) $.
		This shows the assertion.
	\end{proof}
	
	Therefore under the special assumption  $ (z,1) \cap \IE \sim (y,1) \cap \IE $ the coefficient pair for a fixed system $ ( u ) $ does not depend on the choice of $ (z) $. 
	In particular this holds,
	\begin{itemize}
		\item	if $ (z, 1) \sim (y, 1) $ or
		\item 	if $ (y,1) \cap \IE \sim \IE \sim  (z,1) \cap \IE $. 
	\end{itemize}

	\smallskip
	
	We have the following result on the existence of maximal contact:

	\begin{Lem}
		\label{Lem:ExistenceMaxContact}
		Let $ \IE = (J,b) $ be a pair on $ Z $, $ x \in \Sing ( \IE ) $,
		and consider a regular system of parameters $ (u,y) = (u_1, \ldots , u_e, y_1, \ldots, y_r ) $ for $ ( R = \cO_{ Z, x}, \maxIdeal, \resfield ) $ such that the images of $ (y) $ in $ \maxIdeal / \maxIdeal^2 $ define the $ \Dir_x (\IE) $.
		Assume $ \car{\resfield} = 0 $ or $ b < \car{\resfield} $.
		Furthermore, suppose that every differential operator $ \cD_{ N } \in \Diff_\resfield^{ \leq b - 1 } ( \gr_x(Z) ) $ defined via $ \cD_{ N } ( \coeff \, U^A Y^B ) = \binom{B}{ N } \, \coeff \, U^A  Y^{ B - N } $, $ C \in K $, lifts to a differential operators $ \cD_N'  \colon R \to R $, for every $ N \in  \IZ^{r}_{\geq 0} $ with $ |N| < b $.
			
		Then there exists a system $ (z) = (z_1, \ldots, z_r) $ of elements in $ R $ such that we have for every $ j \in \{ 1, \ldots, r \} $:
		\begin{enumerate}
			\item 
			The images of $ z_j $ and $ y_j $ in $ \maxIdeal / \maxIdeal^2 $ coincide.
			\item 
			If we set $ ( \widetilde{u}^ {(j)} ) := (u, z_1, \ldots, z_{j-1}, z_{j+1}, \ldots, z_r ) $, then 
			$$ 
			\IE_x \sim (z_j,1) \cap  \ID_x (\IE, \widetilde{u}^{(j)}, z_j ) .
			$$
			In particular, we have $ \IE_x \sim (z,1) \cap  \ID_x (\IE, u, z ) $,
			i.e., each $ V ( z_j ) $ (and thus $ V (z_1, \ldots, z_r ) $) has maximal contact with $ \IE $ at $ x $.
			
			\item 
			There exist $ \cD_{ j } \in \Diff_\resfield^{ \leq b - 1 } ( \resfield [Y] ) $ and $ F(j) \in \ini_x (\IE) $ such that $ \cD_{ j } (F (j)) = { \epsilon_j } \, Z_j $ for some units $ { \epsilon_j } \in \resfield $. 
			Moreover, there are $ f(j) \in J $ which map in $ \gr_x (Z) $ to $ F (j) $ and $ ( \cD'_{ j } ( f (j) ), 1 ) \sim ( z_j, 1 ) $, where $ \cD'_j $  denotes the differential operator on $ R  $ induced by $ \cD_j $.
		\end{enumerate}
	\end{Lem}
	
	\begin{proof}
		In fact, (1) and (2) follow from the proof of (3).
		Thus we focus on this part.
		Recall the following from the proof of Corollary \ref{Cor:DirRidITCchar0}:
		
		Every element $ F \in \langle Y \rangle^b \setminus \langle Y \rangle^{ b + 1 } $ that is part of a system of generators for $ \ini_x (\IE) $ can be written as			
		$
		F = \sum_{ B \in\IZ^r_{ \geq 0 } : |B| = b } \coeff_B \,Y^B,
		$
		for some $ \coeff_B \in \resfield $.
		For any $ j \in \{ 1, \ldots, r \} $, there exists a generator $ F(j) $ of $ \ini_x (\IE) $ such that there is a $ B(j) = (B_1, \ldots, B_r) \in \IZ_{ \geq 0 }^r $ with $ \coeff_{B(j)} \neq 0 $ and $ B_j \geq 1 $ ($ Y_j $ appears).
		Set $ M (j) := B(j) - e_j \in \IZ^r_{ \geq 0 } $, $ | M(j) | = b - 1 $.
		Let  $ \cD_{ j } := \cD_{ M(j) } \in \Diff_\resfield^{ \leq b - 1 } ( \resfield [Y] ) $ the differential operator which is defined via $ \cD_{ M(j) } ( \coeff \, Y^B ) = \binom{B}{ M(j) } \, \coeff \, Y^{ B - M(j) } $.
		Consequently, 
		
		\vspace{9pt}
		\centerline{$ 
			\cD_{ M(j) } (F (j)) = \coeff_{B(j)} \, B_j \, Y_j + \sum_{B'(i)} \coeff_{ B'(i) } \, B'_i \, Y_i,
			$}
		
		\vspace{9pt}
		\noindent
		where $ B' ( i ) = ( B'_1, \ldots, B'_r ) \in \{ M(j) + e_i \mid i \in \{ 1, \ldots, r \} \setminus \{ j \} \} $.
		The assumptions on $  \car{ \resfield } $ imply that $  B_j $ (and thus $  \coeff_{B(j)}\, B_j $) is a unit in $ \resfield $.
		Define $ Y^\ast_j  \in \resfield [ Y ] $ via
		$$
		Y^\ast_j := ( \coeff_{B(j)}\, B_j  )^{ - 1 } \, \cD_{ M(j) } \, F(j) = Y_j + \sum_{B'(i)} ( \coeff_{B(j)}\, B_j  )^{ - 1 }\coeff_{ B'(i) } \, B'_i \, Y_i.
		$$
		We choose a system of representatives of $ \resfield = R / \maxIdeal $ in $ R $ and define using the previous for $ j = 1 $ the element $ y^\ast_1 \in R $ by replacing $ (Y) $ by $ ( y ) $ in $ Y^\ast_1 $.
		The system $ ( y^\ast_1, y_2, \ldots, y_r ) $ fulfills the same properties as $ ( y ) $. 
		So we may consider the regular system of parameters $ ( u; y^\ast_1, y_2, \ldots, y_r ) $ instead of $ ( u, y ) $ and put $ \cD_1 := \cD_{M(1)} $.
		Then we repeat the above procedure for $ j = 2 $ to obtain $ y^\ast_2 $ and $ \cD_2 $.
		We continue this until we have obtained $ ( y^\ast ) = ( y^\ast_1, \ldots, y^\ast_r ) $ .
		
		Denote by $ \cD'_j $ the differential operator on $ \widehat{ R } $ induced by $ \cD_j $,  $ 1 \leq j \leq r $.
		($\cD_j $ extends by acting trivially on $ (u) $).
		Further, there exist $ f(j) \in J \widehat{ R } $, which are mapped to $ F(j) \in \gr_x (Z) $ and
		$
		\cD'_{ j } ( f (j) ) = \tilde\epsilon_j y^\ast_j + h_j
		$
		for some units $ \tilde\epsilon_j \in R $ and elements $ h_j \in \hR $, which do not involve $ y^\ast_j $.
		Set, for every $ j \in \{ 1, \ldots, r \} $,
		$$
		z_j := y^\ast_j + \tilde\epsilon_j^{ -1 } h_j .
		$$
		Then $ \cD_{ j } (F (j)) = { \epsilon_j } \, Z_j $, $ { \epsilon_j }  = { \tilde\epsilon_j }  \mod \maxIdeal $, $ ( \cD'_{ j } ( f (j) ), 1 ) = ( z_j, 1 ) $ and by the Diff Theorem \ref{Prop:Diff} we have 
		$ \IE_x \sim \IE_x \cap (z_j,1) $.
		Together with Corollary \ref{Cor:zE=zCoeff} we get
		$
		\IE_x \sim (z,1) \cap  \ID_x (\IE, u, z ).
		$
		This proves (1) and (2) holds by construction of the elements $ ( z ) $.
	\end{proof}

	\begin{Cor}
		Fix $ (u) $ as above and let $ (y) $ and $ (z) $ be two extensions of $ (u) $ to a regular system of parameters such that $ V(y) $ and $ V(z) $ have maximal contact.
		Then we have $  \ID_x (\IE, u, y ) \sim   \ID_x (\IE, u, z ) $.
	\end{Cor}
	
	\begin{proof}
		By the previous lemma and Corollary \ref{Cor:zE=zCoeff}
		$$ 
		(y,1) \cap \IE_x \sim (y,1) \cap  \ID_x (\IE, u, y ) \sim \IE_x \sim (z,1) \cap  \ID_x (\IE, u, z )  \sim (z,1) \cap \IE_x
		$$ 
		and Proposition \ref{Prop:IDEz=IDEy} implies $  \ID_x (\IE, u, y ) \sim   \ID_x (\IE, u, z ) $.
	\end{proof}

	%
	%
	%
	%
	%
	%
	%
	%
	%
	%
	%
	%
	%

	\nocite{*}
	\bibliographystyle{cdraifplain}
	\bibliography{xampl}

\end{document}